\documentclass[moor,sglanonrev]{informs4_arxiv}

\usepackage{eqndefns-left} 
\usepackage{bm}
\RequirePackage{endnotes}

\OneAndAHalfSpacedXII 

\usepackage{algorithm}
\usepackage{algpseudocode}
\usepackage{tikz}

{\theoremstyle{THkey}\newtheorem{mytheorem}{XXXXX}}
\usepackage[sort&compress]{natbib}
 \bibpunct[, ]{[}{]}{,}{n}{}{,}%
 \def\bibfont{\small}%
\EquationsNumberedBySection 

\TheoremsNumberedThrough     
\ECRepeatTheorems  %

\MANUSCRIPTNO{MOOR-0001-2024.00}

\newcommand{\twhat}[1]{\mathbf{\widehat{\text{$#1$}}}}
\newcommand{\twtilde}[1]{\mathbf{\widetilde{\text{$#1$}}}}

\usepackage[colorinlistoftodos,bordercolor=orange,backgroundcolor=orange!20,line
color=orange,textsize=scriptsize]{todonotes}
\catcode`\@=11
\catcode`\@=12

\usepackage{subcaption}

\newcommand{\opt}{^{\star}}
\newcommand{\tb}{\textcolor{black}}

\newcommand{\cH}{\mathcal{H}}

\newcommand{\X}{\mathcal{S}}
\newcommand{\cS}{\mathcal{S}}

\newcommand{\A}{\mathcal{A}}
\newcommand{\cA}{\mathcal{A}}
\newcommand{\cP}{\mathcal{P}}
\newcommand{\cU}{\mathcal{U}}
\newcommand{\cW}{\mathcal{W}}
\newcommand{\cE}{\mathcal{E}}
\newcommand{\cM}{\mathcal{M}}
\newcommand{\E}{\mathbb{E}}
\newcommand{\N}{\mathbb{N}}
\newcommand{\R}{\mathbb{R}}

\newcommand{\tr}{^{\top}}

\newcommand{\PiS}{\Pi_{\sf S}}

\newcommand{\PiH}{\Pi_{\sf H}}

\newcommand{\ext}{{\sf ext}}
\newcommand{\conv}{{\sf conv}}
\newcommand{\comp}{{\sf comp}}
\newcommand{\card}{{\sf Card}}

\begin{document}


\RUNAUTHOR{Grand-Cl\'ement, Si and Wang}

\RUNTITLE{DP-compatible Robust MDPs}

\TITLE{{Dynamic Programming–Compatible Uncertainty Sets in Robust Markov Decision Processes}}

\ARTICLEAUTHORS{%
\AUTHOR{Julien Grand-Cl\'ement}
\AFF{Information Systems and Operations Management Department,
Ecole des Hautes Etudes Commerciales de Paris, \EMAIL{grand-clement@hec.fr}}

\AUTHOR{Nian Si}
\AFF{Department of Industrial Engineering and Decision Analytics,
Hong Kong University of Science and Technology, \EMAIL{niansi@ust.hk}}

\AUTHOR{Shengbo Wang}
\AFF{Daniel J. Epstein Department of Industrial and Systems Engineering, University of Southern California, \EMAIL{shengbow@usc.edu}}
} 

\ABSTRACT{%
{In this paper, we investigate the compatibility of robust Markov Decision Processes (RMDPs) with dynamic programming under various assumptions on the uncertainty set, i.e., we investigate when one can solve an RMDP by solving a fixed point equation. We show that in all generality, s-rectangular and sa-rectangular uncertainty sets are the only models of uncertainty that are compatible with dynamic programming. Our analysis shows that existing non-rectangular models, including r-rectangularity, are only weakly compatible with dynamic programming, as they require the assumption that rewards do not depend on the next state. In this case, our results imply that one can always construct a rectangular uncertainty set that is equivalent, for both policy evaluation and optimization, to the dynamic programming-compatible non-rectangular model. This highlights a key limitation: dynamic-programming-compatible non-rectangular uncertainty sets, although practically relevant for uncertainty quantification, do not provide a genuinely distinct assessment of policy performance. Interestingly, our proof techniques rely on identifying a novel simultaneous solvability property, which we show is central to several important properties of RMDPs, including the existence of stationary optimal policies and dynamic programming-based formulations. The simultaneous solvability property enables a unified approach to studying all existing models of uncertainty, rectangular and non-rectangular alike.}
}%

\FUNDING{The work described in this paper was partially supported by a grant from the ANR/RGC Joint Research Scheme sponsored by the Research Grants Council of the Hong Kong Special Administrative Region, China and French National Research Agency (Project No. A25 HKUST603/25). J. Grand-Cl{\'e}ment was also supported by the Labex Ecodec (reference project ANR-11-LABEX-0047).}




\maketitle

\section{Introduction}\label{sec:Intro}
Markov Decision Processes (MDPs) are a framework for modeling sequential decision problems~\citep{puterman2014markov}, where the decision-maker optimizes for the total return obtained over a given horizon (potentially infinite). When the model parameters are known, finding an optimal policy can be done efficiently with value iteration, policy iteration, linear programming or gradient descent. However, it has been observed repeatedly that misestimating the model parameters may lead to severe deterioration of the performance of the computed policies~\cite{delage2010percentile,goh2018data}. {\em Robust} MDPs ameliorate this issue by considering a pessimistic formulation and optimizing for a policy that maximizes the worst-case return over the {\em uncertainty set}, which represents the set of all plausible scenarios for the realizations of the parameters' values~\citep{iyengar2005robust,nilim2005robust,wiesemann2013robust}.

In all generality, i.e. when no assumption is made on the uncertainty set, even the {\em policy evaluation} problem is computationally difficult, i.e., computing the worst-case of a stationary policy may be strongly NP-hard~\citep{wiesemann2013robust,goh2018data}, and optimal policies may have to be history-dependent~\citep{wiesemann2013robust}.
Therefore, an active line of research in the literature on robust MDPs focuses on finding {\em sufficient} conditions on the uncertainty set such that the corresponding RMDP problem {is compatible with dynamic programming, i.e., can be solved by computing the fixed point of a contracting operator}.
The seminal papers~\citep{iyengar2005robust,nilim2005robust} prove that robust MDPs {are compatible with dynamic programming} under the {\em sa-rectangularity} assumption, which states that the uncertainty set is a Cartesian product over all pairs of state-action. This means that the transition probabilities out of every state-action pair can be chosen independently. The authors in \cite{wiesemann2013robust} show that robust MDPs {are compatible with dynamic programming} under the more general assumption of {\em s-rectangularity}, where the uncertainty set is a Cartesian product over all states, i.e., the transition probabilities can be chosen independently across different states. The sa-rectangular and s-rectangular models of uncertainty are by far the most widely used in the literature, and non-rectangular models of uncertainty have been considered difficult to solve, but recent work has highlighted the existence of other interesting models of uncertainty.

In particular, \cite{goh2018data} introduce the {\em r-rectangularity assumption} that generalizes sa-rectangularity and assumes that the transition probabilities are convex combinations of some underlying factors that can be chosen independently. The authors in \cite{goh2018data} show that for the r-rectangular model of uncertainty, the policy evaluation problem {is compatible with dynamic programming but only in a weak sense}, i.e., {\em when the rewards only depend on the current state-action pair} and not on the next state. Under the same assumption on the rewards, \cite{goyal2022robust} show that stationary deterministic optimal policies exist and {can be computed efficiently by dynamic programming} for convex compact r-rectangular uncertainty sets. 
In the case of finite horizon problems, k-rectangular models~\citep{mannor2016robust} require history-dependent optimal policies that can be computed efficiently via state augmentation. 
{Another recent model of uncertainty is the {\em ~$(\xi,\eta)$-rectangular uncertainty set}~\cite{hu2024efficient}, a low-rank model where the transitions depend on some underlying factors and some coefficients that may both vary (with constraints across factors). The authors in~\cite{hu2024efficient} study the computational complexity of the worst-case Bellman updates in this case, without focusing on the properties of optimal policies or the compatibility with dynamic programming. 
}

In this paper, we will use the term {\em rectangular} models to designate s-rectangular or sa-rectangular models of uncertainty, and {\em non-rectangular} models for any other models of uncertainty, e.g. r-rectangular models. We note that this is the original meaning of the word {\em rectangularity} as introduced in \cite{iyengar2005robust,nilim2005robust}, which describe a set that is based on a Cartesian product over the set of states.

{We note that the compatibility with dynamic programming, i.e. the existence of fixed point equations describing the optimal solutions, is central to virtually all current algorithms for solving robust MDPs. Dynamic programming is not an algorithm per se, but it enables the computation of optimal value functions via value iteration (corresponding to simply iterating the Bellman operator, as shown in the seminal papers~\cite{iyengar2005robust,nilim2005robust,wiesemann2013robust}), policy iteration~\cite{ho2021partial}, and even convex programming (via the contraction lemma, e.g.~\cite{nilim2005robust} or Lemma 3.1 in~\cite{grand2024convex}). Similarly, the recent advances in policy mirror descent for robust MDPs~\cite{wang2022policy,li2022robust} are based on dynamic programming equations for the worst-case value functions (e.g., Section 2 in~\cite{wang2022policy} or Proposition 2.2 in~\cite{li2022first}). We do not claim that solving a robust MDP problem is impossible when dynamic programming does not hold: one can resort to approximation methods in general~~\cite{li2025towards} or design algorithms assuming some specific structure for the uncertainty sets, e.g. normed-based uncertainty~\cite{kumar2025non}.}

{Overall, the picture is quite mixed regarding the compatibility of robust MDPs with dynamic programming under different models of uncertainty: several models are currently known, some leading to robust MDPs compatible with dynamic programming in all generality (s-rectangularity and sa-rectangularity), some leading to robust MDPs compatible with dynamic programming only under some assumptions on the rewards (r-rectangularity), some for which this question has not been studied ($(\xi,\eta)$-rectangularity), some for which this question has only been studied for finite horizon (k-rectangularity), and the proof techniques appear quite different across the papers introducing the different models without any obvious connections.}
This raises the following research question:
{
\begin{mytheorem}[Research Question.] 
What are the models of uncertainty for which robust MDPs are compatible with dynamic programming? Is there a unifying approach to studying the compatibility with dynamic programming of all existing uncertainty models?
\end{mytheorem}
}
Answering this research question is the main objective of this paper. Since policy evaluation plays a central role in robust MDPs, we will focus on obtaining {\em necessary and sufficient} conditions for the policy evaluation problem to be solvable by dynamic programming. If this is the case in all generality, we will say that the uncertainty set is {\em ``compatible with dynamic programming"}; if this is the case when the rewards are independent of the next state, we will say that the uncertainty set is {\em ``weakly compatible with dynamic programming"}.
To the best of our knowledge, this paper is the first to focus on {\em necessary and sufficient} conditions, which means that we will be able to characterize {\em all} uncertainty sets compatible with dynamic programming, in contrast to prior work focusing on {\em sufficient} conditions and introducing a single model of uncertainty per paper, as in \cite{iyengar2005robust,nilim2005robust} (studying sa-rectangularity), in \cite{wiesemann2013robust} (studying s-rectangularity) and in \cite{goh2018data,goyal2022robust} (studying r-rectangularity).
Our {\bf main contributions} are as follows.
\paragraph{Rectangular MDPs are the only models compatible with dynamic programming.} We show that in all generality, only s-rectangularity and sa-rectangularity are compatible with dynamic programming. Our proof techniques relate dynamic programming with the {\em simultaneous solvability property} of a collection of optimization problems over the uncertainty set. These optimization problems appear from the Bellman operator, which operates over the value functions as vectors indexed by the set of states, thereby involving a different optimization problem for each state of the MDP instance. This provides a unifying view on the underlying reasons behind the compatibility (or incompatibility) of uncertainty models with dynamic programming.
\paragraph{Weakly compatible RMDPs in some special cases.} We then study the weak compatibility of robust MDPs with dynamic programming, i.e. under the assumption that the rewards do not depend on the next state. We prove the equivalence between weak compatibility and a particular weaker form of the simultaneous solvability property. Our analysis shows that solving policy evaluation via dynamic programming has several implications, including the ability to find the best stationary policies and, when the uncertainty set is convex, to find optimal policies (among all policies), as well as an equivalence between RMDP instances with stationary and non-stationary adversaries. We then describe several new weakly compatible non-rectangular uncertainty sets not previously considered in the literature, a surprising result given that before our work, there was a paucity of weakly compatible uncertainty sets known (only three were known: sa-rectangularity, s-rectangularity, and r-rectangularity).
\paragraph{New perspectives on robust MDPs.} Overall, our results answer several important open questions in the literature on robust MDPs and shed new light on important aspects of RMDPs, that were previously poorly understood. We summarize some of our main takeaways here:
\begin{itemize}
    \item We emphasize the crucial role of rectangular models for robust MDPs, by proving that it is hopeless to search for other uncertainty models that can be solved via dynamic programming {\em in all generality} (Theorem \ref{th:reformulation adv ssp rsas' - srec} and Theorem \ref{th:reformulation adv ssp rsas' - sarec}).
    \item We show that the compatibility of policy evaluation with dynamic programming implies that we can efficiently compute optimal stationary policies (Theorem \ref{th:solving max min}), emphasizing that the gist of the difficulty in solving robust MDPs (a max-min problem) lies in solving the policy evaluation problem (the inner minimization problem). We also highlight that the existence of a stationary optimal policy is a consequence of strong duality (Lemma \ref{lem:strong duality implies stationary opt policies}) and provides a very general equivalence result for RMDPs with stationary and non-stationary adversaries (Theorem \ref{th:stationary vs history dependent adversaries}).
    \item In future work on robust MDPs, the (weak) compatibility of new models of uncertainty can be verified directly using our necessary and sufficient simultaneous solvability properties as in~\eqref{eq:adv ssp - rsa - srec} (in all generality) and as in~\eqref{eq:adv ssp - rsa - sarec} (under the assumption that the rewards do not depend on the next states). \tb{From a practical standpoint, these conditions are simple enough for new uncertainty sets to be (weakly) compatible ``by design", i.e. to be built purposely to satisfy \eqref{eq:adv ssp - rsa - srec} or \eqref{eq:adv ssp - rsa - sarec}.}
    \item Non-rectangular models were originally introduced to overcome the conservativeness of the rectangularity assumption. Crucially, we prove that policy evaluation over an important class of non-rectangular models of uncertainty (r-rectangularity) can actually be solved by solving policy evaluation {\em over their s-rectangular or sa-rectangular extensions}, i.e. over the smallest s-rectangular or sa-rectangular uncertainty sets containing the non-rectangular uncertainty set. \tb{This shows that the only known non-rectangular models are, in fact, {\em equivalent to rectangular models}, which undermines the practical need for non-rectangular models and the fact that they may be more general than rectangular models.}
\end{itemize}

The rest of this paper is organized as follows. We provide important notations in the remainder of this section. In Section \ref{sec:preliminaries} we introduce some preliminaries on robust MDPs and several important models of uncertainty. In Section \ref{sec:ssp} we study the compatibility of uncertainty sets with dynamic programming in all generality. In Section \ref{sec:weak ssp} we study the weak compatibility of uncertainty sets (i.e., under some structural assumption on the rewards). Section \ref{sec:discussion} provides some discussion on the main takeaways of our paper.
\subsection{Notations.}
Given an integer $r \in \N$, we write $[r]$ for the set $\{1,\ldots,r\}$. Given a set $\cE$, we denote $\ext(\cE)$ its set of extreme points and $\conv(\cE)$ its convex hull. If the set $\cE$ is finite, we denote $\card(\cE)$ its cardinality and $\Delta(\cE)$ for the simplex in $\R^{\card(\cE)}$. For $n \in \N$, the canonical basis of $\R^{n}$ is written as $(\bm{e}_{1},...,\bm{e}_{n})$. Given two vectors $\bm{u},\bm{v} \in \R^{n}$, we write $\bm{u} \leq \bm{v}$ for the collection of component-wise inequalities $u_{s} \leq v_{s}, \forall \; s \in [n].$ The usual scalar product $\sum_{s=1}^{n} u_{s}v_{s}$ between vectors $\bm{v},\bm{u} \in \R^{n}$ is denoted interchangeably by $\langle \bm{v},\bm{u} \rangle$ and $\bm{v}\tr\bm{u}$, depending of what is the most readable in a given context. For matrices $\bm{A} \in \R^{n \times m},\bm{B} \in \R^{n \times m}$, we use $\langle \bm{A},\bm{B} \rangle = \sum_{i=1}^{n}\sum_{j=1}^{m} A_{ij}B_{ij}$. We use the convention that $0^0 = 1.$
\section{Preliminaries on robust MDPs}\label{sec:preliminaries}
A robust Markov Decision Process instance is given by the tuple $\cM=\left(\cS,\A,\bm{r},\gamma,\bm{\mu},\cP\right)$, {where $\cS$ and $\A$ are the finite sets of states and actions, with cardinality $S:=\card(\cS),A := \card(\cA)$. With a slight abuse of notation, we identify the set $\cS$ with $\{1,...,S\}$ and the set $\cA$ with $\{1,...,A\}$.} We note $\bm{r} =(r_{sas'}) \in \R^{S \times A \times S}$ for the instantaneous reward obtained at each period. {Without loss of generality, we assume that every action is admissible at every state.} The scalar $\gamma \in [0,1)$ is the {\em discount factor}, which models the diminishing importance of the rewards obtained over time, and the vector $\bm{\mu} \in \Delta(\cS)$ is the initial distribution over the set of states. The set $\cP$ is the uncertainty set, which models all the plausible realizations of the transition probabilities: it is a subset of the set of all transition probabilities, i.e. $\cP \subset \Delta(\cS)^{S \times A} = \{ \left(\bm{P}_{sa}\right)_{sa} \; | \; \bm{P}_{sa} \in \Delta(\cS), \forall \; (s,a) \in \cS \times \A\}$. {The set of {\em histories} up to period $t \in \N$ is defined as $\cH_{t}:= \{(s_0,a_0,s_1,a_1,...,s_{t-1},a_{t-1},s_{t}) \; | \; (s_{\tau},a_{\tau}) \in \cS \times \cA, \forall \; \tau \leq t-1, s_{t} \in \cS\} =  \left(\times_{\tau =0}^{t-1} \cS \times \cA\right) \times \cS $ and the set of history-dependent policies $\PiH$ is the set of maps from the set of all finite histories $\cup_{t \in \N} \cH_t$ to $\Delta(\cA)$.}
Given a robust MDP instance, the goal of the decision-maker is to find an optimal policy, defined as a solution to the following optimization problem:
\begin{equation}\label{eq:robust mdp}
    \sup_{\pi \in \PiH} \inf_{\bm{P} \in \cP} \bm{\mu}\tr\bm{v}^{\pi,\bm{P}}
\end{equation}
where $\bm{v}^{\pi,\bm{P}} \in \R^{S}$ is the value function associated with a pair of policy-transitions $(\pi,\bm{P}) \in \PiH \times \cP$, whose component gives the expected cumulative discounted return obtained starting from any state:
\begin{equation}\label{eq:value function}
    v^{\pi,\bm{P}}_{s} = \E^{\pi,\bm{P}}\left[ \left. \sum_{t=0}^{+\infty} \gamma^t r_{s_{t}a_{t}s_{t+1}} \; \right| \; s_{0} =s  \right], \forall \; s \in \cS.
\end{equation}
{Note that in \eqref{eq:robust mdp}, the transition probabilities do not change over time; we discuss the case of time-varying transition probabilities in Section~\ref{sec:implications of ddp - srec}. We will also use the notation $R(\pi,\bm{P}):=\bm{\mu}\tr\bm{v}^{\pi,\bm{P}}$ for the expected discounted return.}
The set of {\em stationary policies} that only depend on the current state is denoted by $\PiS := \Delta(\cA)^{S}$.
Note that when $(\pi,\bm{P}) \in \PiS \times \cP$, the vector $\bm{v}^{\pi,\bm{P}}$ is the unique fixed point of the contraction operator $T^{\pi}_{\bm{P}}:\R^{S} \rightarrow \R^{S}$, defined as, for $\bm{v} \in \R^{S}$,
\begin{equation}\label{eq:operator T pi P}
    T^{\pi}_{\bm{P}}(\bm{v})_{s} =  \sum_{a \in \cA} \pi_{sa}\bm{P}_{sa}\tr\left(\bm{r}_{sa}+\gamma \bm{v}\right), \forall \; s \in \cS.
\end{equation}
\paragraph{Solving the policy evaluation problem.} Following the terminology from \cite{wiesemann2013robust}, the {\em policy evaluation} problem \eqref{eq:policy evaluation} consists in computing the worst-case value obtained by a stationary policy, i.e. in solving, for $\pi \in \PiS$, the optimization problem:
\begin{equation}\label{eq:policy evaluation}
    \inf_{\bm{P} \in \cP} \bm{\mu}\tr\bm{v}^{\pi,\bm{P}}
\end{equation}
Being able to efficiently solve the policy evaluation problem for all stationary policies is a desirable and important property for the practical use of a model of uncertainty. For instance, \cite{goh2018data} use robust MDPs to evaluate the worst-case cost between two existing screening guidelines for colorectal cancer; for this purpose, it is sufficient to be able to solve \eqref{eq:policy evaluation}. Solving the policy evaluation problem also arises as a subroutine in recently proposed gradient descent algorithms for robust MDPs, e.g. Algorithm 1 in \cite{wang2023policy} or Algorithm 1 in \cite{li2022first}. 

The authors in \cite{wiesemann2013robust,goh2018data} show that without any assumption on the set $\cP$, solving \eqref{eq:policy evaluation} is NP-hard. In the {\em sa-rectangular} model of uncertainty~\citep{iyengar2005robust,nilim2005robust}, the set $\cP$ can be written as a Cartesian product over all state-action pairs $(s,a) \in \cS \times \A$:
\begin{equation*}
    \cP = \times_{(s,a) \in \cS \times \A} \cP_{sa}
\end{equation*}
where $\cP_{sa} \subset \Delta(\cS)$ is the set of marginals induced by $\cP$ when only considering the transition probabilities out of the pair $(s,a) \in \cS \times \A$:
\begin{equation}\label{eq:marginal sa}
    \cP_{sa} := \{\bm{p} \in \Delta(\cS) \; | \; \exists \; \bm{P} \in \cP, \bm{P}_{sa} = \bm{p} \}.
\end{equation}
The authors in \citep{iyengar2005robust,nilim2005robust} show that under the sa-rectangularity assumption, the optimization problem \eqref{eq:policy evaluation} can be solved efficiently when the set $\cP$ is compact, and that in this case we have
\begin{equation}\label{eq:adv dpp intro sarec}
    \min_{\bm{P} \in \cP} \bm{\mu}\tr\bm{v}^{\pi,\bm{P}} = \bm{\mu}\tr\twhat{\bm{u}}^{\pi}
\end{equation}
where $\twhat{\bm{u}}^{\pi} \in \R^{S}$ is the unique vector satisfying the following dynamic programming equations:
\[\twhat{u}^{\pi}_{s} = \sum_{a \in \A} \pi_{sa}\min_{\bm{P} \in \cP} \bm{P}_{sa}\tr\left(\bm{r}_{sa} + \gamma \twhat{\bm{u}}^{\pi} \right), \forall \; s \in \cS.\]
Importantly, under the sa-rectangularity assumption, an optimal policy in \eqref{eq:robust mdp} exists and can be chosen stationary and deterministic.

There are two generalizations of sa-rectangularity. In the case of {\em r-rectangularity}~\citep{goh2018data,goyal2022robust}, the transition probabilities in $\cP$ are assumed to be convex combinations of the same underlying $r$ factors for some integer $r \in \N$, i.e., there exists some {\em coefficients} $\bm{u}_{sa} \in \Delta([r])$ for $(s,a) \in \cS \times \A$, some {\em factors} $\bm{w}^{i} \in \cW^{i}$ for $i \in [r]$, and some sets $\cW^{i} \subset \Delta(\cS)$ such that
\begin{equation}\label{eq:r-rectangular uncertainty}
    \cP = \left\{ \left(\sum_{i=1}^{r} u^{i}_{sa}\bm{w}^{i}\right)_{sa} \; | \; (\bm{w}^{1},...,\bm{w}^{r}) \in \cW^{1} \times ... \cW^{r}\right\}.
\end{equation}
Any sa-rectangular uncertainty set can be reformulated as an r-rectangular uncertainty set with $r= S \times A$ (Proposition 2.1 in \cite{goyal2022robust}). Under the additional assumption that the reward is independent of the next state: $r_{sas'} = r_{sas''}$ for all $(s,a) \in \cS \times \A$ and $s',s'' \in \cS$, \cite{goyal2022robust} show that one can still solve the policy iteration problem by a fixed point argument and that an optimal policy in \eqref{eq:robust mdp} can be chosen stationary and deterministic. In fact, we will show in Section \ref{sec:weak ssp} that for compact r-rectangular uncertainty sets, the dynamic programming equation~\eqref{eq:adv dpp intro sarec} still holds, simplifying the approach taken in \cite{goh2018data,goyal2022robust}.

A second generalization of sa-rectangularity is {\em s-rectangularity}, introduced in \cite{wiesemann2013robust}. In this model, $\cP$ is equal to the Cartesian product of its marginals across each state $s$:
\begin{equation*}
    \cP = \times_{s \in \cS} \cP_{s}
\end{equation*}
with, for each $s \in \cS$,
\begin{equation}\label{eq:marginal s}
    \cP_{s}:= \left\{ \left(\bm{p}_{a}\right)_{a \in \A} \in \Delta(\cS)^{A} \; | \; \exists \; \bm{P} \in \cP, \left(\bm{P}_{sa}\right)_{a \in \A} = \left(\bm{p}_{a}\right)_{a \in \A} \right\}.
\end{equation}
For compact s-rectangular uncertainty sets, \cite{wiesemann2013robust} prove that
\begin{equation}\label{eq:adv dpp intro srec}
    \min_{\bm{P} \in \cP} \bm{\mu}\tr\bm{v}^{\pi,\bm{P}} = \bm{\mu}\tr\bm{u}^{\pi}
\end{equation}
where $\bm{u}^{\pi}\in \R^{S}$ is defined as the unique vector satisfying the following fixed point equations:
\begin{equation}
\label{eq:bellman s rec}
u^{\pi}_{s} = \min_{\bm{P} \in \cP} \sum_{a \in \A} \pi_{sa} \bm{P}_{sa}\tr\left(\bm{r}_{sa} + \gamma \bm{u}^{\pi} \right), \forall \; s \in \cS.
\end{equation}
It is important to note that \cite{wiesemann2013robust} show that \eqref{eq:adv dpp intro srec} holds for s-rectangular uncertainty {\em without any structural assumption on the rewards}, in stark contrast to the case of r-rectangular uncertainty. We also note that \cite{wiesemann2013robust} show the existence of a stationary optimal solution to the robust MDP problem for compact convex s-rectangular uncertainty, which can be computed efficiently via {value iteration~\cite{wiesemann2013robust}, policy iteration~\cite{ho2021partial}, or convex programming~\cite{grand2024convex}.} 
\paragraph{The rectangular extensions of an uncertainty set.}
Overall, we note that for all the widely-studied models of uncertainty, the policy evaluation \eqref{eq:policy evaluation} can be solved efficiently by computing the fixed point of a certain operator:
\begin{itemize}
    \item For compact sa-rectangular models and compact r-rectangular models (the latter requiring the additional assumption that the rewards are independent of the next state), one can solve the policy evaluation problem \eqref{eq:policy evaluation} by leveraging \eqref{eq:adv dpp intro sarec}, i.e., by computing the fixed point $\twhat{u}^{\pi}$ of the contracting operator $\twhat{T}^{\pi}:\R^{S} \rightarrow \R^{S}$, defined as, for $\bm{v} \in \R^{S}$,
    \begin{equation}\label{eq:operator hat T pi}
        \twhat{T}^{\pi}(\bm{v})_{s} =  \sum_{a \in \cA} \pi_{sa}\min_{\bm{P} \in \cP}\bm{P}_{sa}\tr\left(\bm{r}_{sa}+\gamma \bm{v}\right), \forall \; s \in \cS.
    \end{equation}
    \item For s-rectangular models, one can solve the policy evaluation problem \eqref{eq:policy evaluation} by leveraging \eqref{eq:adv dpp intro srec}, i.e., by computing the fixed point $\bm{u}^{\pi}$ of the contracting operator $T^{\pi}:\R^{S} \rightarrow \R^{S}$, defined for $\bm{v} \in \R^{S}$,
    \begin{equation}\label{eq:operator T pi}
        T^{\pi}(\bm{v})_{s} = \min_{\bm{P} \in \cP} \sum_{a \in \cA} \pi_{sa}\bm{P}_{sa}\tr\left(\bm{r}_{sa}+\gamma \bm{v}\right), \forall \; s \in \cS. 
    \end{equation}
\end{itemize}
We note that from the definition of the operators $T^{\pi}_{\bm{P}},\twhat{T}^{\pi}$ and $T^{\pi}$, we always have
\begin{equation}\label{eq:ordering u's}
    \twhat{\bm{u}}^{\pi} \leq \bm{u}^{\pi} \leq \bm{v}^{\pi,\bm{P}}, \forall \; (\pi,\bm{P}) \in \PiS \times \cP.
\end{equation}
Equation \eqref{eq:ordering u's} follows from classical arguments based on the properties of the operators $T^{\pi},\twhat{T}^{\pi}$ and $T^{\pi}_{\bm{P}}$, see Appendix \ref{app:proof value functions} for completeness. {Since $\twhat{T}^{\pi}$ and $T^{\pi}$ are always contractions, irrespective of the rectangularity properties of $\cP$, their fixed point can be computed efficiently via fixed point iterations, among other algorithms (see more details in Section \ref{sec:implications of ddp - srec} and Appendix~\ref{app:fixed point iteration}).}
We note that Equality \eqref{eq:adv dpp intro sarec} and Equality \eqref{eq:adv dpp intro srec} are related to the recently introduced {\em dynamic programming principle} framework~\cite{wang2023foundation}, which relates the optimization problem~\eqref{eq:robust mdp} in the case of s-rectangular and sa-rectangular uncertainty sets, and dynamic programming formulations for stationary or history-dependent decision-maker and transition probabilities. Our analysis in this paper will focus specifically on the relation of the policy evaluation problem~\eqref{eq:policy evaluation} and dynamic programming formulations, for general uncertainty sets.

Even though r-rectangular uncertainty sets may not be sa-rectangular, it is important to note that since \eqref{eq:adv dpp intro sarec} holds, we can solve policy evaluation problems over an r-rectangular uncertainty set $\cP$ by solving policy evaluation problems over its {\em sa-rectangular extension} $\twhat{\cP}$, i.e. with $\twhat{\cP}$ defined as
\begin{equation*}
    \twhat{\cP} = \times_{(s,a) \in \cS \times \A} \cP_{sa}
\end{equation*}
for $\cP_{sa} \subset \Delta(\cS)$ as defined in \eqref{eq:marginal sa}. We can define the analogous notion of {\em s-rectangular extension} of an uncertainty set, as 
\begin{equation}\label{eq:s-rectangular extension}
    \twtilde{\cP} = \times_{s \in \cS} \cP_{s}
\end{equation}
with $\cP_{s}$ as defined in \eqref{eq:marginal s}. Note that we always have $\cP \subseteq \twtilde{\cP} \subseteq \twhat{\cP}$
with $\cP = \twtilde{\cP}$ if and only if $\cP$ is s-rectangular, and $\cP = \twhat{\cP}$ if and only if $\cP$ is sa-rectangular. Equation \eqref{eq:adv dpp intro srec} can be interpreted as solving the policy evaluation problem over $\cP$ by reducing it to a policy evaluation problem over its s-rectangular extension $\twtilde{\cP}$, which itself can be solved by dynamic programming~\citep{wiesemann2013robust}.
{The next lemma shows that even when $\cP$ is non-rectangular, the vector $\bm{u}^{\pi}$ is closely related to the s-rectangular extension $\twtilde{\cP}$ and the vector $\twhat{\bm{u}}^{\pi}$ is closely related to the s-rectangular extension $\twhat{\cP}$.
\begin{lemma}\label{lem:u pi solves PE on s-rec extension}
    Let $\cP$ be an uncertainty set.
    \begin{enumerate}
        \item For $\bm{u}^{\pi} \in \R^{S}$ the fixed point of the operator $T^{\pi}$ defined in~\eqref{eq:operator T pi}, we have
    \[\bm{\mu}\tr\bm{u}^{\pi} = \inf_{\bm{P} \in \twtilde{\cP}} \bm{\mu}\tr\bm{v}^{\pi,\bm{P}}.\]
    \item For $\twhat{\bm{u}}^{\pi} \in \R^{S}$ the fixed point of the operator $\twhat{T}^{\pi}$ defined in~\eqref{eq:operator hat T pi}, we have
    \[\bm{\mu}\tr\twhat{\bm{u}}^{\pi} = \inf_{\bm{P} \in \twhat{\cP}} \bm{\mu}\tr\bm{v}^{\pi,\bm{P}}.\]
    \end{enumerate}
\end{lemma}
\begin{proof}{Proof of Lemma~\ref{lem:u pi solves PE on s-rec extension}.}
\begin{enumerate}
    \item Consider the optimization programs~\eqref{eq:operator T pi} appearing in the definition of the operator $T^{\pi}$ for a given uncertainty set $\cP$. We note that these optimization programs only involve the s-th component of $\cP$. By definition of the s-rectangular extension $\twtilde{\cP}$, the s-th component of $\cP$ coincides with the s-th component of $\twtilde{\cP}$.  Therefore we have, for any $\bm{v} \in \R^{S}$ and any $s \in \cS$,
\[\min_{\bm{P} \in \cP} \sum_{a \in \cA} \pi_{sa}\bm{P}_{sa}\tr\left(\bm{r}_{sa}+\gamma \bm{v}\right) = \min_{\bm{P} \in \twtilde{\cP}} \sum_{a \in \cA} \pi_{sa}\bm{P}_{sa}\tr\left(\bm{r}_{sa}+\gamma \bm{v}\right).\]
This implies that the operator $T^{\pi}$ induced by any uncertainty set $\cP$ coincides with the same operator induced by its s-rectangular extension $\twtilde{\cP}$. Therefore, $\bm{u}^{\pi}$ is also the fixed point of the operator $T^{\pi}$ induced by $\twtilde{\cP}$. Now for s-rectangular uncertainty sets we know that \eqref{eq:adv dpp intro srec} holds, which concludes the proof. 
\item The proof follows verbatim the lines of the first item, replacing $\bm{u}^{\pi},T^{\pi}$ by $\twhat{\bm{u}}^{\pi},\twhat{T}^{\pi}$. \hfill \Halmos
\end{enumerate}
\end{proof}
}
\paragraph{DP-Compatible uncertainty sets.}
We focus on the class of uncertainty sets for which we can solve policy evaluation by dynamic programming. We will distinguish two notions of compatibility with dynamic programming, depending on whether the policy evaluation problem \eqref{eq:policy evaluation} can be solved via its s-rectangular extension (as in \eqref{eq:adv dpp intro srec}) or via its sa-rectangular extension (as in \eqref{eq:adv dpp intro sarec}). Since the compatibility properties of policy evaluation will follow entirely from the properties of the uncertainty sets, with a slight abuse of language, we will say that the uncertainty sets themselves are {\em compatible} (with dynamic programming). 
{
\begin{definition}[DP-compatible Uncertainty Sets]\label{def:compatibility}
Let $\cP$ be an uncertainty set.
\begin{enumerate}
    \item The set $\cP$ is {\em s-DP-compatible} if Equation~\eqref{eq:adv dpp intro srec} holds for all policies $\pi \in \PiS$ and for all choices of the rewards $\bm{r} \in \R^{S \times A \times S}$, the initial distribution $\bm{\mu} \in \Delta(\cS)$ and the discount factor $\gamma \in [0,1)$. 
    \item The set $\cP$ is {\em sa-DP-compatible} if Equation~\eqref{eq:adv dpp intro sarec} holds for all policies $\pi \in \PiS$ and for all choices of the rewards $\bm{r} \in \R^{S \times A \times S}$, the initial distribution $\bm{\mu} \in \Delta(\cS)$ and the discount factor $\gamma \in [0,1)$. 
\end{enumerate}
\end{definition}
}
In the rest of the paper, we will use the term {\em DP-compatibility} if we refer to s-DP-compatibility or sa-DP-compatibility indifferently.
Note that in Definition \ref{def:compatibility}, Equation~\eqref{eq:adv dpp intro sarec} or Equation~\eqref{eq:adv dpp intro srec} must hold under all choices of the MDP parameters $(\bm{r},\bm{\mu},\gamma)$. Indeed, if one is willing to consider only specific values of the MDP parameters, several corner cases allow for very simple solutions for the policy evaluation problem, e.g. when the rewards are all zero ($\bm{r}=\bm{0}$), or when $\gamma=0$ and the rewards do not depend on the next states. For a model of uncertainty to be useful in all applications, it is crucial that we can {\em always} solve the policy evaluation problem, not only in some special cases.
Equipped with the concept of compatible uncertainty sets, we can reformulate the main research question of this paper, outlined in the introduction:
\begin{resq}\label{resq: dpp}
   Can we fully characterize the class of uncertainty sets that are sa-DP-compatible?  Can we fully characterize the class of uncertainty sets that are s-DP-compatible?
\end{resq}
In the next section, we will provide a complete answer to this research question in the most general case (without any assumption of the rewards).  We then treat the case of {\em weak DP-compatibility} in Section \ref{sec:weak ssp}, i.e., the case where the rewards are independent of the next state. By doing so, we also provide a unifying approach to all existing models of uncertainty in the literature.
{
The following proposition shows the relation that sa-DP-compatibility implies s-DP-compatibility, mirroring the fact that sa-rectangularity implies s-rectangularity.
}
\begin{proposition}\label{prop:sa tractable are s tractable}
    Let $\cP$ be a compact sa-DP-compatible uncertainty set. Then $\cP$ is s-DP-compatible and $\bm{u}^{\pi} = \twhat{\bm{u}}^{\pi}$ for any $\pi \in \PiS$.
\end{proposition}
\begin{proof}{Proof of Proposition \ref{prop:sa tractable are s tractable}.}
Let $\cP$ be compact and sa-DP-compatible. Let $\pi \in \PiS$ and $\bm{r} \in \R^{S \times A \times S}, \gamma \in [0,1),\bm{\mu}\in \Delta(\cS)$.
We have
    \[ \min_{\bm{P} \in \cP} \bm{\mu}\tr\bm{v}^{\pi,\bm{P}} = \bm{\mu}\tr\twhat{\bm{u}}^{\pi} \leq \bm{\mu}\tr\bm{u}^{\pi} \leq \min_{\bm{P} \in \cP} \bm{\mu}\tr\bm{v}^{\pi,\bm{P}}
    \]
    where the equality comes from sa-DP-compatibility, and the inequalities come from \eqref{eq:ordering u's}. Therefore, all terms above are equal and $\cP$ is s-DP-compatible. Since $\bm{\mu}\tr\twhat{\bm{u}}^{\pi} = \bm{\mu}\tr\bm{u}^{\pi}$ for all $\bm{\mu} \in \Delta(\cS)$, we have $\bm{u}^{\pi} = \twhat{\bm{u}}^{\pi}$.
    \hfill \Halmos
\end{proof}
{We conclude this section with the next remarks about other models of robust MDPs.
\begin{remark}[RMDPs with other objective functions.]\label{rmk:average return}
    This paper focuses on RMDPs with {\em discounted return}, i.e. on the case where the value function is defined as in~\eqref{eq:value function}. While this is the most widely studied RMDP model, it is worth highlighting that when the discount factor is close to $1$, computing an optimal policy for RMDPs with discounted returns may be difficult, since the complexity of the main algorithms (e.g. value iteration) scales as $O(1/(1-\gamma))$~\cite{wiesemann2013robust}, leading to very long solving times when $\gamma$ approaches 1. Other models of RMDPs, without discounting, include finite horizon MDPs (e.g. section 3 in~\cite{nilim2005robust}) or the case of average return where the value function is the {\em gain}, defined as the limit of the running-average of the rewards, which has seen recent progress~\cite{grand2023beyond,wang2023robust} after the seminal work~\cite{tewari2007bounded}. However, for average return RMDPs, a Bellman equation has been obtained only for {\em unichain} instances~\cite{wang2023robust}. We leave studying our main research question (Research Question~\ref{resq: dpp}) for average return RMDPs as an interesting future direction.
\end{remark}
}
{
\begin{remark}[RMDPs with optimistic objectives.]
There exists a variant of robust MDPs where the decision-maker optimizes for the {\em best-case} return over the set $\cP$, i.e., where one wants to solve
\begin{equation}\label{eq:optimistic MDP}
    \sup_{\pi \in \PiH} \sup_{\bm{P} \in \cP} \bm{\mu}\tr\bm{v}^{\pi,\bm{P}}.
\end{equation}
This problem is considered for bounded parameter MDPs~\cite{givan2000bounded}.
Variants of~\eqref{eq:optimistic MDP} with the average return and sa-rectangular uncertainty sets are studied in~\cite{tewari2007bounded} and used as subroutines in extended value iteration~\cite{auer2008near} and KLUCRL~\cite{filippi2010optimism}. The corresponding {\em optimistic policy evaluation} for \eqref{eq:optimistic MDP} is to solve, for each $\pi \in \PiS$, the optimization program
\begin{equation}\label{eq:policy evaluation - optimistic}
    \sup_{\bm{P} \in \cP} \bm{\mu}\tr\bm{v}^{\pi,\bm{P}}
\end{equation}
Because of the homogeneity of the expectation (appearing in the definition of the value functions as in~\eqref{eq:value function}), it is clear that we can turn the optimistic policy evaluation problem~\eqref{eq:policy evaluation - optimistic} into a corresponding policy evaluation problem~\eqref{eq:policy evaluation} by flipping the signs of the rewards $\bm{r}\in \R^{S \times A \times S}$. Because of this, all of our results for the (classical) policy evaluation problem~\eqref{eq:policy evaluation} also hold for the optimistic policy evaluation problem~\eqref{eq:policy evaluation - optimistic}.
\end{remark}
}

\section{DP-compatible uncertainty sets}\label{sec:ssp}
In this section, we show that, in general, the only DP-compatible models of uncertainty are the s-rectangular and the sa-rectangularity models of uncertainty. Our main results are based on reformulating DP-compatibility as an equivalent {\em simultaneous solvability property}, which unifies most of the existing approaches to studying dynamic programming for robust MDPs.
We start with our results for s-rectangular uncertainty and then provide the analogous results for sa-rectangular uncertainty.

\paragraph{s-DP-compatible uncertainty sets.} 
We start with the following theorem, which provides {\em necessary and sufficient} conditions for \eqref{eq:adv dpp intro srec} to hold.
\begin{theorem}\label{th:adv DPP, feasible upi, adv SSP equivalent - rsas' - srec}
Let $\cP$ be compact.
The following statements are equivalent. 
    \begin{enumerate}
    \item The set $\cP$ is s-DP-compatible:
    \begin{equation}\label{eq:adv dpp rsas' - srec}
        \forall \; \pi \in \PiS, \forall \; \bm{r} \in \R^{S \times A \times S}, \forall \; \gamma \in [0,1), \forall \; \bm{\mu} \in \Delta(\cS), \min_{\bm{P} \in \cP} \bm{\mu}\tr\bm{v}^{\pi,\bm{P}} = \bm{\mu}\tr\bm{u}^{\pi}
    \end{equation}
        \item The vector $\bm{u}^{\pi}$ is feasible:
        \begin{equation}\label{eq:feas upi rsas' - srec}
            \forall \; \pi \in \PiS,\forall \; \bm{r} \in \R^{S \times A \times S}, \forall \; \gamma \in [0,1), \exists \; \twhat{\bm{P}} \in \cP, \bm{u}^{\pi} = \bm{v}^{\pi,\twhat{\bm{P}}}.
        \end{equation}
        \item The following {\em Simultaneous Solvability Property} (SSP) holds:
        \begin{equation}\label{eq:adv ssp rsas' - srec} 
            \forall \; \bm{V} \in \R^{S \times A \times S}, \cap_{s \in \cS} \arg \min_{\bm{P} \in \cP} \; \langle \bm{P}_{s},\bm{V}_{s}\rangle  \neq \emptyset.
        \end{equation}
    \end{enumerate}
\end{theorem}           
\begin{proof}{Proof of Theorem \ref{th:adv DPP, feasible upi, adv SSP equivalent - rsas' - srec}.}
    Our proof proceeds in three parts.
    
{\bf Part 1:} \eqref{eq:adv dpp rsas' - srec} $\Rightarrow$ \eqref{eq:feas upi rsas' - srec}. 
We employ a proof by contradiction. Assume that \eqref{eq:feas upi rsas' - srec} does not hold for some $\pi \in \PiS,\bm{r}\in \R^{S \times A \times S},\gamma \in [0,1)$. Then for all transition probabilities $\bm{P} \in \cP$ there exists a state $\bar{s} \in \cS$ such that  $u^{\pi}_{\bar{s}} < v^{\pi,\bm{P}}_{\bar{s}}.$ Recall, additionally, that $u^{\pi}_{s} \leq v^{\pi,\bm{P}}_{s}$ holds for all state $s \in \cS$ from \eqref{eq:ordering u's}.
For $\bm{\mu} = \frac{1}{S}(1,...,1)$, let $\twhat{\bm{P}} \in \arg \min_{\bm{P} \in \cP} \bm{\mu}\tr\bm{v}^{\pi,\bm{P}}$, and let $\bar{s} \in \cS$ such that $u^{\pi}_{\bar{s}} < v^{\pi,\twhat{\bm{P}}}_{\bar{s}}$. We have 
\[ \min_{\bm{P} \in \cP} \bm{\mu}\tr\bm{v}^{\pi,\bm{P}} = \min_{\bm{P} \in \cP} \frac{1}{S} \sum_{s \in \cS} v^{\pi,\bm{P}}_{s} = \sum_{s \in \cS}  \frac{1}{S} 
 v^{\pi,\twhat{\bm{P}}}_{s} > \sum_{s \in \cS}  \frac{1}{S} 
 u^{\pi}_{s} = \bm{\mu}\tr\bm{u}^{\pi},\]
 where the strict inequality comes from $u^{\pi}_{\bar{s}} < v^{\pi,\twhat{\bm{P}}}_{\bar{s}}$ and $u^{\pi}_{s} \leq v^{\pi,\bm{P}}_{s}$ for all state $s \in \cS$. Therefore, \eqref{eq:adv dpp rsas' - srec} does not hold. This shows that \eqref{eq:adv dpp rsas' - srec} $\Rightarrow$ \eqref{eq:feas upi rsas' - srec}.

{\bf Part 2:} \eqref{eq:feas upi rsas' - srec}  $\Rightarrow$ \eqref{eq:adv ssp rsas' - srec}.
Let $\bm{V} \in \R^{S \times A \times S}$. 
Let $\pi \in \PiS$ be the uniform policy, $\gamma=0$ and $\bm{r} \in \R^{S \times A \times S}$ such that $r_{sas'} = V_{sas'}, \forall \; (s,a,s')$.
Then from \eqref{eq:feas upi rsas' - srec} we know that there exists $\twhat{\bm{P}} \in \cP$ such that $\bm{u}^{\pi} = \bm{v}^{\pi,\twhat{\bm{P}}}.$
But with this definition of the MDP instance {and since $\pi$ is the uniform policy}, we have, for all $s \in \cS$,
\begin{align*}
    v^{\pi,\twhat{\bm{P}}}_{s} & = \sum_{a \in \A} \pi_{sa}\twhat{\bm{P}}_{sa}\tr\bm{r}_{sa} = \frac{1}{|\A|} \langle \twhat{\bm{P}}_{s},\bm{r}_{s}\rangle,\\
    v^{\pi,\twhat{\bm{P}}}_{s}& = u^{\pi}_{s}  = \frac{1}{|\A|}\min_{\bm{P} \in \cP} \langle \bm{P}_{s},\bm{r}_{s}\rangle.
\end{align*}
Since $\bm{r}_{s} = \bm{V}_{s}$ for all $s \in \cS$, this shows that $\twhat{\bm{P}} \in \cap_{s \in \cS} \arg \min_{\bm{P} \in \cP} \langle \bm{P}_{s},\bm{V}_{s}\rangle$ so that $\cap_{s \in \cS} \arg \min_{\bm{P} \in \cP} \langle \bm{P}_{s},\bm{V}_{s}\rangle  \neq \emptyset.$

{\bf Part 3:} \eqref{eq:adv ssp rsas' - srec}  $\Rightarrow$ \eqref{eq:adv dpp rsas' - srec}. 
Assume that \eqref{eq:adv ssp rsas' - srec} holds. Let $\bm{r} \in \R^{S \times A \times S},\gamma \in [0,1), \pi \in \PiS$. Recall that we always have
$\bm{u}^{\pi} \leq \bm{v}^{\pi,\bm{P}}, \forall \; \bm{P} \in \cP$
so that we always have, for any $\bm{\mu} \in \Delta(\cS)$,
\[ \bm{\mu}\tr\bm{u}^{\pi} \leq \min_{\bm{P} \in \cP} \bm{\mu}\tr\bm{v}^{\pi,\bm{P}}.\]
We now prove the reverse inequality.
Recall that by definition, $\bm{u}^{\pi}$ is the unique vector such that
\[ u^{\pi}_{s} = \min_{\bm{P} \in \cP} \sum_{a \in \A} \pi_{sa}\bm{P}_{sa}\tr\left(\bm{r}_{sa} + \gamma \bm{u}^{\pi}\right), \forall \; s \in \cS.\]
Let $\bm{V} \in \R^{S \times A \times S}$ be defined as $V_{sas'} = \pi_{sa}(r_{sas'}+\gamma u^{\pi}_{s'})$ for $(s,a,s') \in \cS \times \A \times \cS$.
From \eqref{eq:adv ssp rsas' - srec} we know that there exists $\twhat{\bm{P}} \in \cP$ such that
\[  \twhat{\bm{P}} \in  \cap_{s \in \cS} \arg \min_{\bm{P} \in \cP} \langle \bm{P}_{s},\bm{V}_{s}\rangle.
\]
This can be rewritten 
\[ u^{\pi}_s = \sum_{a \in \A} \pi_{sa}\twhat{\bm{P}}_{sa}\tr\left(\bm{r}_{sa} + \gamma \bm{u}^{\pi}\right) = T^{\pi}_{\twhat{\bm{P}}}(\bm{u}^{\pi})_{s}, \forall \; s \in \cS. \]
From the uniqueness of the fixed point of the contraction $T^{\pi}_{\bm{P}}$, we know that $\bm{u}^{\pi} = \bm{v}^{\pi,\twhat{\bm{P}}}$.
From this we obtain that for any $\bm{\mu} \in \Delta(\cS)$ we have
\[ \bm{\mu}\tr\bm{u}^{\pi} = \bm{\mu}\tr\bm{v}^{\pi,\twhat{\bm{P}}} \geq \min_{\bm{P} \in \cP} \bm{\mu}\tr\bm{v}^{\pi,\bm{P}}\]
and therefore we conclude that \eqref{eq:adv dpp rsas' - srec} holds.
\hfill \Halmos
\end{proof}
Theorem \ref{th:adv DPP, feasible upi, adv SSP equivalent - rsas' - srec} provides {\em necessary and sufficient} conditions for Equation~\eqref{eq:adv dpp intro srec} to hold in all generality. As highlighted by \eqref{eq:feas upi rsas' - srec}, Equation~\eqref{eq:adv dpp intro srec} holds if and only if the vector $\bm{u}^{\pi}$, defined as the unique fixed point of the contraction $T^{\pi}$ as in \eqref{eq:operator T pi}, is {\em feasible}, i.e. only if we can find $\twhat{\bm{P}} \in \cP$ such that $\bm{u}^{\pi} = \bm{v}^{\pi,\twhat{\bm{P}}}$. {Since the operator $T^{\pi}$ involves a minimization over $\cP$ {\em independently} across each component $s \in \cS$, Equation~\eqref{eq:feas upi rsas' - srec} is, in turn, equivalent to a {\em simultaneous solvability property} (SSP), i.e. to the property that a feasible point in $\cP$ can be recovered as an optimal solution of multiple optimization problems, each of them involving the marginals of $\cP$ across different $s \in \cS$, as formulated in \eqref{eq:adv ssp rsas' - srec} where each of the optimization problems is over the entire uncertainty set (with variables $\bm{P} \in \cP$) but where the objective function only depends on the transition probabilities $\bm{P}_{s}$ out of a single state $s$. }
We emphasize that contrary to most prior work, we do not only study sufficient conditions for \eqref{eq:adv dpp intro srec} to hold, we also discover some necessary condition~\eqref{eq:adv dpp rsas' - srec}. We next provide a complete characterization of the sets satisfying \eqref{eq:adv dpp rsas' - srec}.
\paragraph{Uncertainty sets satisfying the SSP property~\eqref{eq:adv ssp rsas' - srec}.}
It is straightforward to note that compact s-rectangular uncertainty sets always satisfy \eqref{eq:adv ssp rsas' - srec}. 
However, if we allow the uncertainty set to be nonconvex, some non-rectangular uncertainty sets may satisfy \eqref{eq:adv ssp rsas' - srec}. We provide a simple example below.
\begin{example}[Non s-rectangular sets may satisfy \eqref{eq:adv ssp rsas' - srec}]\label{ex:non s-rec satisfy ssp rsas'}
    Assume that $\cS = \{1,2\}$ and that there is only one action to choose ($\A = \{1\}$). Then any uncertainty set $\cP$ is a subset of $\{ (p,1-p),(q,1-q) \; | \; p,q \in [0,1] \}$ where $(p,1-p)$ is the transition probabilities out of State $1$ and $(q,1-q)$ is the transition probabilities out of State $2$. For the sake of conciseness, we write $\bm{P} = (p,q)$ for $\bm{P} \in \cP$. For instance $\bm{P} = (0,0)$ corresponds to deterministic transitions $1 \rightarrow 2$ and $2 \rightarrow 1$.
Consider the following finite uncertainty set: 
$\displaystyle \cP = \{(0,0),(0,1),(1,0),(1,1),(1/2,1/2)\}.$
Clearly, $\cP$ is not convex since it is finite. It is also not s-rectangular, because $(1/2,1/2) \in \cP$ and $(0,0) \in \cP$ but $(0,1/2) \notin \cP$. 
However \eqref{eq:adv ssp rsas' - srec} always holds for $\cP$, since the point $(1/2,1/2)$ is irrelevant from an optimization standpoint - it is never the {\em unique} minimizer of a linear form over $\cP$, and $\cP \setminus \{(1/2,1/2)\}$ is s-rectangular.
\end{example}

As highlighted in Example \ref{ex:non s-rec satisfy ssp rsas'}, the SSP~\eqref{eq:adv ssp rsas' - srec} is fundamentally a statement about the {\em extreme points} of $\cP$, since it only involves minimization of linear objectives over $\cP$. When $\cP$ is convex and compact, it is fully described by its set of extreme points $\ext(\cP)$, since in this case $\cP = \conv(\ext(\cP))$ from the Krein-Milman theorem. Therefore, we can completely characterize the sets satisfying \eqref{eq:adv ssp rsas' - srec} in the case of compact {\em convex} uncertainty sets. Our main theorem in this section shows that, in this case, the only s-DP-compatible uncertainty sets are s-rectangular.
\begin{theorem}\label{th:reformulation adv ssp rsas' - srec}
   Assume that $\cP$ is convex and compact. Then
       $\eqref{eq:adv ssp rsas' - srec} \iff \cP = \times_{s \in \cS} \cP_{s}$, \; for $\cP_{s}$ the marginals of $\cP$ for state $s \in \cS$, defined in \eqref{eq:marginal s}.
\end{theorem}
We have already shown that we always have $\cP = \times_{s \in \cS} \cP_{s} \Rightarrow \eqref{eq:adv ssp rsas' - srec}$ (even when $\cP$ is nonconvex).
Thus, the main difficulty in proving Theorem \ref{th:reformulation adv ssp rsas' - srec} is proving the other implication under the convexity and the compactness of the uncertainty set. 
\begin{proof}{Proof of Theorem \ref{th:reformulation adv ssp rsas' - srec}.}
    We start by proving that $\times_{s \in \cS} \ext(\cP_{s}) \subset \cP.$

Let $\left(\twhat{\bm{P}}_{s}\right)_{s \in \cS} \in \times_{s \in \cS} \ext(\cP_{s})$. We want to show that $\twhat{\bm{P}} \in \cP$.
From Straszewicz's theorem~\citep{straszewicz1935exponierte} (see Theorem 18.6 in \cite{rockafellar2015convex}), we know that for each $s \in \cS$, there exists a sequence of {\em exposed points} $\left(\twhat{\bm{P}}_{s}^{n}\right)_{n \in \N}$ and a sequence of vectors $\left(\bm{V}_{s}^{n}\right)_{n \in \N}$ with {$\twhat{\bm{P}}_{s}^{n} \in \cP_{s}$, $\twhat{\bm{P}}^{n}_{s} \rightarrow \hat{\bm{P}}_{s}$} for $n \rightarrow + \infty$, and 
\[ \{ \twhat{\bm{P}}_{s}^{n} \} = \arg \min_{\bm{P}_{s} \in \cP_{s}} \langle \bm{P}_{s},\bm{V}_{s}^{n}\rangle.\]
Let us consider $ n \in \N$. We want to show that $\twhat{\bm{P}}^{n} = \left(\twhat{\bm{P}}_{s}^{n}\right)_{s \in \cS}$ is feasible in $\cP$. From \eqref{eq:adv ssp rsas' - srec}, we know that there exists $\twtilde{\bm{P}}^{n} \in \cP$ such that
\[ \twtilde{\bm{P}}^{n} \in \cap_{s \in \cS} \arg \min_{\bm{P}_{s} \in \cP_{s}} \langle \bm{P}_{s},\bm{V}_{s}^{n}\rangle.\]
From the uniqueness of the minimizers in each of the optimization programs $\arg \min_{\bm{P}_{s} \in \cP_{s}} \langle \bm{P}_{s},\bm{V}_{s}^{n}\rangle$ for each $s \in \cS$, we obtain that $\twtilde{\bm{P}}_{s}^{n} = \twhat{\bm{P}}_{s}^{n}$ for each $s \in \cS$, from which we conclude that $\twhat{\bm{P}}^{n} = \twtilde{\bm{P}}^{n}$ and therefore that $\twhat{\bm{P}}^{n} \in \cP$.
Therefore, there exists a sequence $\left(\twhat{\bm{P}}^{n}\right)_{n \in \N}$ such that $\twhat{\bm{P}}^{n} \in \cP$ for all $n \in \N$ and $\lim_{n \rightarrow + \infty} \twhat{\bm{P}}^{n} = \twhat{\bm{P}}$. Since $\cP$ is compact, we conclude that $\twhat{\bm{P}} \in \cP$, which shows that $\times_{s \in \cS} \ext(\cP_{s}) \subset \cP$.
It is now straightforward to conclude since we have
\[ \times_{s \in \cS} \ext(\cP_{s}) \subset \cP  \Rightarrow \conv(\times_{s \in \cS} \ext(\cP_{s})) \subset \conv(\cP)   \Rightarrow \times_{s \in \cS} \conv(\ext(\cP_{s})) \subset  \cP,\]
where the first implication follows from taking the convex hull on both sides of the first inclusion, and the second implication follows from the convexity of $\cP$ and from {$\conv(A \times B) = \conv(A) \times \conv(B)$ for two sets $A$ and $B$.}
Since $\cP$ is convex and compact, $\cP_{s}$ is convex and compact for each $s \in \cS$, i.e., $\conv(\ext(\cP_{s})) = \cP_{s}$. We conclude that $\times_{s \in \cS} \cP_{s} \subset  \cP$. Since it is always true that $\cP \subset \times_{s \in \cS} \cP_{s}$, we conclude that $\cP = \times_{s \in \cS} \cP_{s}$.
\hfill \Halmos
\end{proof}
We remark that it is common to read in the RMDP literature that {\em ``RMDPs with r-rectangular models  can be solved via fixed-point computation"}. In contrast, Theorem \ref{th:reformulation adv ssp rsas' - srec} shows that this is not the case in all generality: the r-rectangular model is only {\em weakly} DP-compatible, i.e., the assumption is required that the instantaneous rewards do not depend on the next state, as stated in~\citep{goh2018data,goyal2022robust}. We will provide a more thorough discussion on this point in Section \ref{sec:weak ssp}.
\paragraph{sa-DP-compatible uncertainty sets.}
We have the exact converse as Theorem \ref{th:adv DPP, feasible upi, adv SSP equivalent - rsas' - srec} for the case of sa-rectangular extensions. 
\begin{theorem}\label{th:adv DPP, feasible upi, adv SSP equivalent - rsas' - sarec}
Let $\cP$ be compact.
The following statements are equivalent. 
    \begin{enumerate}
    \item The set $\cP$ is sa-DP-compatible:
    \begin{equation}\label{eq:adv dpp rsas' - sarec}
        \forall \; \pi \in \PiS, \forall \;\bm{r} \in \R^{S \times A \times S}, \forall \;\gamma \in [0,1), \forall \; \bm{\mu} \in \Delta(\cS), \min_{\bm{P} \in \cP} \bm{\mu}\tr\bm{v}^{\pi,\bm{P}} = \bm{\mu}\tr\twhat{\bm{u}}^{\pi}
    \end{equation}
        \item The vector $\twhat{\bm{u}}^{\pi}$ is feasible:
        \begin{equation}\label{eq:feas upi rsas' - sarec}
            \forall \; \pi \in \PiS, \forall \;\bm{r} \in \R^{S \times A \times S}, \forall \;\gamma \in [0,1), \exists \; \twhat{\bm{P}} \in \cP, \twhat{\bm{u}}^{\pi} = \bm{v}^{\pi,\twhat{\bm{P}}}.
        \end{equation}
        \item The following SSP holds:
        \begin{equation}\label{eq:adv ssp rsas' - sarec} 
            \forall \; \bm{V}  \in \R^{S \times A \times S},  \cap_{(s,a) \in \cS \times \A} \arg \min_{\bm{P} \in \cP} \; \langle \bm{P}_{sa},\bm{V}_{sa}\rangle   \neq \emptyset.
        \end{equation}
    \end{enumerate}
\end{theorem}
For the sake of conciseness, we provide the proof in Appendix \ref{app:proof rsas' sarec}.
We note that the SSP~\eqref{eq:adv ssp rsas' - sarec} is stronger than the SSP~\eqref{eq:adv ssp rsas' - srec}, in the sense that if \eqref{eq:adv ssp rsas' - sarec} holds then \eqref{eq:adv ssp rsas' - srec} holds. This is consistent with the fact that sa-rectangular uncertainty sets are also s-rectangular and with Proposition \ref{prop:sa tractable are s tractable}.
We are now ready to characterize the uncertainty sets satisfying \eqref{eq:adv ssp rsas' - sarec}. Example \ref{ex:non s-rec satisfy ssp rsas'} also applies for the case of the sa-rectangular extension (because $|\A|=1$ in Example \ref{ex:non s-rec satisfy ssp rsas'}) and provides an example of a non-convex, non-sa-rectangular uncertainty set satisfying \eqref{eq:adv ssp rsas' - sarec}. However, among convex compact uncertainty sets, only sa-rectangular models of uncertainty may satisfy \eqref{eq:adv ssp rsas' - sarec}.
\begin{theorem}
    \label{th:reformulation adv ssp rsas' - sarec}
   Assume that $\cP$ is convex and compact. Then $\eqref{eq:adv ssp rsas' - sarec} \iff \cP = \times_{(s,a) \in \cS \times \A} \cP_{sa}$, \; for $\cP_{sa}$ the marginals of $\cP$ for the state-action pair $(s,a) \in \cS \times \A$, defined in \eqref{eq:marginal sa}.

\end{theorem}
\paragraph{Discussion.} Our results in this section provide a complete characterization of DP-compatible uncertainty sets, and show that only s-rectangular and sa-rectangular are DP-compatible {\em in all generality}, i.e., for all values of the MDP parameters. {In retrospect, this result is natural, since the operator $T^{\pi}$ and $\hat{T}^{\pi}$ involves independent optimization programs over $\bm{P} \in \cP$ for each $s \in \cS$.} This answers an important open question in the robust MDP literature, where new non-rectangular models of uncertainty are still being invented even in recent years, often without a clear understanding of the compatibility of these models with dynamic programming~\citep{hu2024efficient}. {Our results show that it is hopeless to search for non-rectangular models that are compatible with dynamic programming {\em in all generality}. However, we will uncover several new non-rectangular models that are compatible with dynamic programming under some more restricted reward models in the next section.}

We also provide a new understanding of the rectangularity property for robust MDPs. For the class of convex compact uncertainty sets, Theorem \ref{th:reformulation adv ssp rsas' - srec} and Theorem \ref{th:reformulation adv ssp rsas' - sarec} show that s-rectangularity and sa-rectangularity can be interpreted as {\em simultaneous solvability properties}, which directly relates to the Bellman operators $T^{\pi}$ and $\twhat{T}^{\pi}$ defined in Section \ref{sec:preliminaries} involving optimization problems over $\cP$ across different states and different state-action pairs. To the best of our knowledge, these reformulations of the rectangularity assumption are new, and they provide interesting insights on the underlying reason behind the DP-compatibility of s-rectangular and sa-rectangular models of uncertainty. We summarize the main takeaways of this section in Figure \ref{fig:contributions_first_part}.
\begin{figure}
    \centering
    \includegraphics[width=0.7\linewidth]{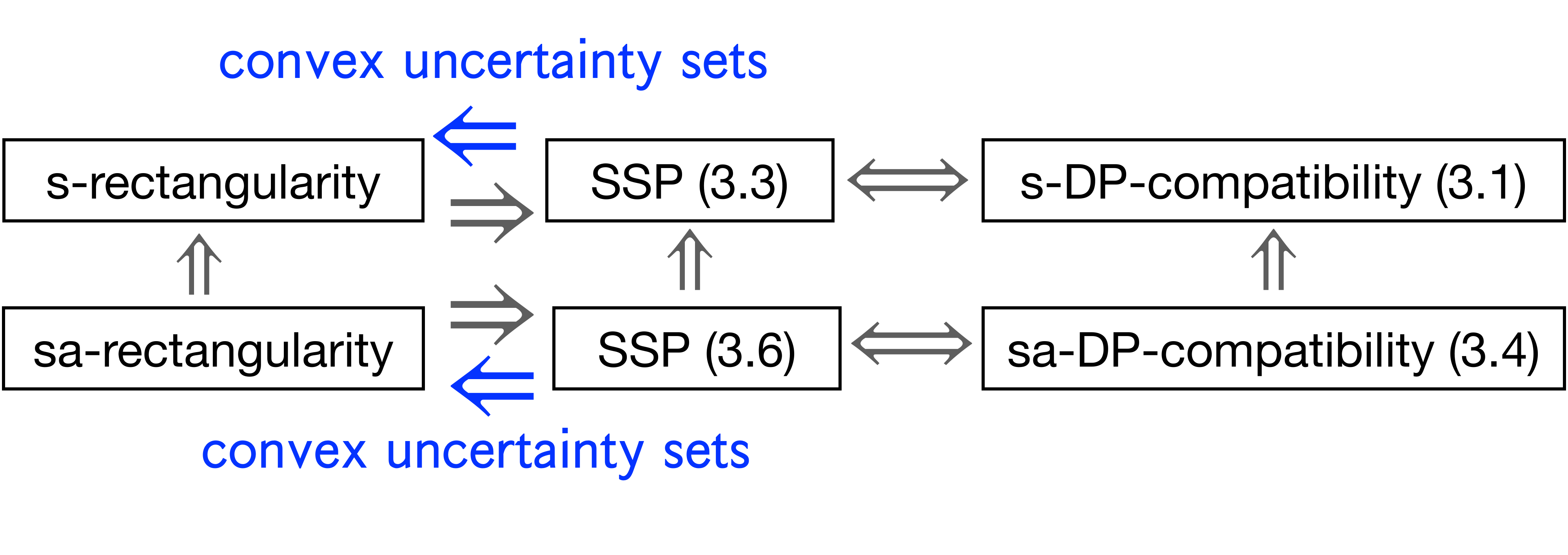}
    \caption{Summary of our main results from Section \ref{sec:ssp}. The uncertainty sets are assumed to be compact. The implications in blue require the additional assumption that the uncertainty set is convex. Recall that we defined $R(\pi,\bm{P})=\bm{\mu}\tr\bm{v}^{\pi,\bm{P}}$.}
    \label{fig:contributions_first_part}
\end{figure}
\section{Weakly DP-compatible uncertainty sets}\label{sec:weak ssp}
{In this section, we revisit the results from Section \ref{sec:ssp} under an additional assumption on the rewards. Section~\ref{sec:weak sDPcomp} first focuses a weaker version of s-DP-compatibility and the equivalence with a weak notion of simultaneous solvability property, before showing all the consequences of weak-SP-compatibility (Section~\ref{sec:implications of ddp - srec}) and identifying new non-rectangular models that are weakly s-DP-compatible (Section~\ref{sec:uncertainty sets for adv ssp}). Section~\ref{sec:weak ssp - sarec} extends the analysis to weak sa-DP-compatibility, and Figure~\ref{fig:inclusion_model_simplified} and Figure~\ref{fig:contributions_third_part} summarize the inclusions between the corresponding models of uncertainty and the consequences for robust MDPs. 
}

In particular, we will use the assumption that the rewards are independent of the next state.
\begin{assumption}\label{assumption:rewards}
    The instantaneous rewards are independent of the next state:
    \[r_{sas'} = r_{sas''},\forall \; (s,a) \in \cS \times \A, \forall \; (s',s'') \in \cS \times \cS.\]
\end{assumption}
This assumption is used in several RMDP papers, although it is not stated explicitly as an assumption but simply stated in the RMDP model~\cite{goh2018data,goyal2022robust,hu2024efficient},. With a slight abuse of notation, under Assumption \ref{assumption:rewards}, we write $r_{sa} \in \R$ for the value of $r_{sas'}$ (common across all the next states $s' \in \cS$). We will now study the following weaker notions of s-DP-compatibility and sa-DP-compatibility.
{
\begin{definition}[Weakly DP-Compatible Uncertainty Sets]\label{def:weak tractability}
Let $\cP$ be an uncertainty set.
\begin{enumerate}
    \item The set $\cP$ is {\em weakly s-DP-compatible} if Equation~\eqref{eq:adv dpp intro srec} holds for all policies $\pi \in \PiS$ and for all choices of the rewards $\bm{r} \in \R^{S \times A \times S}$ {\em satisfying Assumption \ref{assumption:rewards}}, the initial distribution $\bm{\mu} \in \Delta(\cS)$ and the discount factor $\gamma \in [0,1)$. 
    \item The set $\cP$ is {\em weakly sa-DP-compatible} if Equation~\eqref{eq:adv dpp intro sarec} holds for all policies $\pi \in \PiS$ and for all choices of the rewards $\bm{r} \in \R^{S \times A \times S}$ {\em satisfying Assumption \ref{assumption:rewards}}, the initial distribution $\bm{\mu} \in \Delta(\cS)$ and the discount factor $\gamma \in [0,1)$. 
\end{enumerate}
\end{definition}
}
One can interpret the notion of weak DP-compatibility as {\em ``DP-compatibility when Assumption \ref{assumption:rewards} holds for the rewards"}, whereas our original definition of DP-compatibility (Definition \ref{def:compatibility}) does not impose any condition on the rewards.

\subsection{Weakly s-DP-compatible uncertainty sets}\label{sec:weak sDPcomp}
We now provide necessary and sufficient conditions for the weak s-DP-compatibility of an uncertainty set $\cP$.
\begin{theorem}\label{th:adv DPP, feasible upi, adv SSP equivalent - rsa - srec}
Let $\cP$ be compact. The following statements are equivalent. 
    \begin{enumerate}
    \item The set $\cP$ is weakly s-DP-compatible:
\begin{equation}\label{eq:adv dpp - rsa - srec}
        \forall \; \pi \in \PiS,\forall \; \bm{r} \in \R^{S \times A},\forall \;  \gamma \in [0,1), \forall \; \bm{\mu} \in \Delta(\cS), \min_{\bm{P} \in \cP} \bm{\mu}\tr\bm{v}^{\pi,\bm{P}} = \bm{\mu}\tr\bm{u}^{\pi}
    \end{equation}
        \item The vector $\bm{u}^{\pi}$ is feasible under Assumption \ref{assumption:rewards}:
\begin{equation}\label{eq:feas upi - rsa - srec}
            \forall \; \pi \in \PiS,\forall \; \bm{r} \in \R^{S \times A},\forall \;  \gamma \in [0,1), \exists \; \twhat{\bm{P}} \in \cP, \bm{u}^{\pi} = \bm{v}^{\pi,\twhat{\bm{P}}}.
        \end{equation}
        \item The following weak simultaneous solvability property holds:
        \begin{equation}\label{eq:adv ssp - rsa - srec}
            \forall \;(\pi,\bm{V}) \in \PiS\times \R^{S}, \cap_{s \in \cS} \arg \min_{\bm{P} \in \cP}  \; \langle \bm{P}_{s},\bm{\pi}_{s}\bm{V}\tr \rangle \neq \emptyset.
        \end{equation}
    \end{enumerate}
\end{theorem}
\begin{proof}{Proof of Theorem \ref{th:adv DPP, feasible upi, adv SSP equivalent - rsa - srec}.}

{\bf Part 1:} \eqref{eq:adv dpp - rsa - srec} $\Rightarrow$ \eqref{eq:feas upi - rsa - srec}. This part follows verbatim as Part 1 in the proof of Theorem \ref{th:adv DPP, feasible upi, adv SSP equivalent - rsas' - srec}.

{\bf Part 2:} \eqref{eq:feas upi - rsa - srec} $\Rightarrow$ \eqref{eq:adv ssp - rsa - srec}. 
Let $\bm{V} \in \R^{S}$. Let $\pi \in \PiS$. We claim that we can find $\bm{r} \in \R^{S \times A}$ and $\gamma \in [0,1)$ for which $\bm{V} = \bm{u}^{\pi}$. To do so, it suffices to fix any $\gamma \in [0,1)$ and to pick $\bm{r} \in \R^{S \times A}$ such that $r_{sa} = r_{s}, \forall \; (s,a) \in \cS \times \A$ and
$r_{s} := V_{s} - \min_{\bm{P} \in \cP} \sum_{a \in \A} \gamma \pi_{sa}\bm{P}_{sa}\tr\bm{V}, \forall \; s \in \cS.$
For such a choice of the rewards, we indeed have $\bm{V} = T^{\pi}(\bm{V})$, and therefore $\bm{V} = \bm{u}^{\pi}$. 
From \eqref{eq:feas upi - rsa - srec} we have that there exists $ \twhat{\bm{P}} \in \cP$ such that $\bm{u}^{\pi} = \bm{v}^{\pi,\twhat{\bm{P}}}$, i.e., we know that 
\[V_{s} = \min_{\bm{P} \in \cP} \sum_{a \in \A} \pi_{sa}\left(r_{sa} + \gamma \bm{P}_{sa}\tr\bm{V}\right) = \sum_{a \in \A} \pi_{sa}\left(r_{sa} + \gamma \twhat{\bm{P}}_{sa}\tr\bm{V}\right), \forall \; s \in \cS.\]
Since $\langle \bm{P}_{s},\bm{\pi}_s\bm{V}\tr \rangle = \sum_{a \in \A} \pi_{sa}\bm{P}_{sa}\tr\bm{V}$, this shows that 
\[\cap_{s \in \cS} \arg\min_{\bm{P} \in \cP} \; \langle \bm{P}_{s},\bm{\pi}_s\bm{V}\tr \rangle \neq \emptyset.\]

{\bf Part 3:} \eqref{eq:adv ssp - rsa - srec} $\Rightarrow$ \eqref{eq:adv dpp - rsa - srec}. Let $\pi \in \PiS, \bm{r} \in \R^{S \times A}, \gamma \in [0,1)$. Note that \eqref{eq:adv ssp - rsa - srec} ensures that {$\bm{u}^{\pi} = \bm{v}^{\pi,\twhat{\bm{P}}}$} for some $\twhat{\bm{P}} \in \cP$. Similarly as in part 3 of the proof of Theorem \ref{th:adv DPP, feasible upi, adv SSP equivalent - rsas' - srec}, we conclude that \eqref{eq:adv dpp - rsa - srec} holds.
    \hfill \Halmos    
\end{proof}
Recall that Theorem \ref{th:adv DPP, feasible upi, adv SSP equivalent - rsas' - srec} shows the equivalence (in general) of s-DP-compatibility with the SSP~\eqref{eq:adv ssp rsas' - srec}. Under Assumption \ref{assumption:rewards}, Theorem \ref{th:adv DPP, feasible upi, adv SSP equivalent - rsa - srec} shows the equivalence of weak s-DP-compatibility with the weak SSP~\eqref{eq:adv ssp - rsa - srec}. In particular, whereas \eqref{eq:adv ssp rsas' - srec} allows for different objective functions $\bm{V}_{s_{1}},\bm{V}_{s_{2}} \in \R^{A \times S}$ across different states $s_{1},s_{2} \in \cS$, the objective functions in \eqref{eq:adv ssp - rsa - srec} all arise from the same vector $\bm{V} \in \R^{S}$, common across all states, but combined with (potentially) different probability distributions $\bm{\pi}_{s} \in \Delta(\cA)$ to form the matrix $\bm{\pi}_{s}\bm{V}\tr$ arising in the objective function $\bm{P} \mapsto \langle \bm{P}_{s},\bm{\pi}_{s} \bm{V}\tr \rangle$ for each $s \in \cS$. Since \eqref{eq:adv ssp - rsa - srec} is a weaker condition than \eqref{eq:adv ssp rsas' - srec}, we can expect non-rectangular uncertainty sets to satisfy it. We will provide an analysis of the uncertainty sets satisfying \eqref{eq:adv ssp - rsa - srec} in Section \ref{sec:uncertainty sets for adv ssp}. Before doing so, we first analyze the implications of weak s-DP-compatibility in the next section.
\subsubsection{Implications of weak s-DP-compatibility}\label{sec:implications of ddp - srec}
The weak s-DP-compatibility has several important consequences.
\paragraph{Tractability of the policy evaluation problem.} {When $\cP$ is weakly s-DP-compatible, solving $\displaystyle
\inf_{\bm{P} \in \cP} \bm{\mu}\tr\bm{v}^{\pi,\bm{P}}$, the policy evaluation problem over $\cP$, boils down to solving the policy evaluation problem over $\twtilde{\cP}$, the s-rectangular extension of $\cP$:
\begin{equation}\label{eq:PE over cP = PE over tildecP}
  \min_{\bm{P} \in \cP} \bm{\mu}\tr\bm{v}^{\pi,\bm{P}}  =   \min_{\bm{P} \in \twtilde{\cP}} \bm{\mu}\tr\bm{v}^{\pi,\bm{P}}.
\end{equation}
This is because, from Lemma~\ref{lem:u pi solves PE on s-rec extension}, we always have $\displaystyle \bm{\mu}\tr\bm{u}^{\pi}=\min_{\bm{P} \in \twtilde{\cP}} \bm{\mu}\tr\bm{v}^{\pi,\bm{P}}$, while weak s-DP-compatibility is exactly equivalent to $\min_{\bm{P} \in \cP} \bm{\mu}\tr\bm{v}^{\pi,\bm{P}} = \bm{\mu}\tr\bm{u}^{\pi}$.
}
{
There is an extensive literature on solving the policy evaluation problem for s-rectangular models of uncertainty (e.g.~\cite{ho2018fast,ho2021partial,behzadian2021fast,ho2022robust}), and we refer to the seminal paper~\cite{wiesemann2013robust} for precise complexity statements (more precisely~Section 3 in~\cite{wiesemann2013robust}). For the sake of completeness, we provide the description and complexity of the fixed point iteration algorithm (also called value iteration in our setting) in Appendix~\ref{app:fixed point iteration}.
}
\paragraph{Solving the max-min problem.} 
We show that under \eqref{eq:adv dpp - rsa - srec}, one can efficiently compute the best policy in the class of stationary policies, i.e. one can efficiently solve 
\begin{equation}\label{eq:max min problem}
    \max_{\pi \in \PiS} \min_{\bm{P} \in \cP} \bm{\mu}\tr\bm{v}^{\pi,\bm{P}}
\end{equation}
In particular, let us define the contracting operator $T:\R^{S} \rightarrow \R^{S}$ as 
\begin{equation}\label{eq:operator T}
    T(\bm{v})_{s}  := \max_{\pi \in \PiS} \; T^{\pi}(\bm{v})_{s} = \max_{\pi \in \PiS} \min_{\bm{P} \in \cP} T^{\pi}_{\bm{P}}(\bm{v})_{s}, \forall \; \bm{v} \in \R^{S},\forall \; s \in \cS
\end{equation}
and we define $\bm{u}\opt \in \R^{S}$ as the unique fixed point of $T$: $\bm{u}\opt = T(\bm{u}\opt)$. The operator $T$ is the Bellman operator for the s-rectangular extension of $\cP$~\citep{wiesemann2013robust}. Note that since $\PiS = \Delta(\cA)^{S} = \times_{s \in \cS} \Delta(\cA)$ and since the optimization problem for the s-th component of $T$ only involves the s-th component of $\pi$, we can always find a policy $\pi\opt \in \PiS$ such that $\bm{u}\opt = \bm{u}^{\pi\opt}$. We have the following theorem.
\begin{theorem}\label{th:solving max min}
    Let $\cP$ be a compact uncertainty set and assume that \eqref{eq:adv dpp - rsa - srec} holds. Then
    \[ \max_{\pi \in \PiS} \min_{\bm{P} \in \cP} \bm{\mu}\tr\bm{v}^{\pi,\bm{P}} = \bm{\mu}\tr\bm{u}\opt.\]
    Additionally, for any $\pi\opt \in \PiS$ such that $\bm{u}\opt = \bm{u}^{\pi\opt}$, we have 
    \[ \pi\opt \in \arg \max_{\pi \in \PiS} \min_{\bm{P} \in \cP} \bm{\mu}\tr\bm{v}^{\pi,\bm{P}}.\]
\end{theorem}
\begin{proof}{Proof of Theorem \ref{th:solving max min}.}
Note that for any $\pi \in \PiS$, we have $T^{\pi}(\bm{u}\opt) \leq T(\bm{u}\opt) = \bm{u}\opt$, so that $\bm{u}^{\pi} \leq \bm{u}\opt$. Now for $\pi\opt$ such that $T(\bm{u}\opt) = T^{\pi\opt}(\bm{u}\opt)$, we find that $\bm{u}\opt$ is a fixed point of the contraction $T^{\pi\opt}$. But $T^{\pi\opt}$ has a unique fixed point $\bm{u}^{\pi\opt}$, so that $\bm{u}\opt = \bm{u}^{\pi\opt}$. Thus $\max_{\pi \in \PiS} \bm{\mu}\tr\bm{u}^{\pi} = \bm{\mu}\tr\bm{u}^{\opt}.$
Now that the last equality does not require \eqref{eq:adv dpp - rsa - srec}. If we assume \eqref{eq:adv dpp - rsa - srec}, we know additionally that
\[ \bm{\mu}\tr\bm{u}^{\pi} = \min_{\bm{P} \in \cP} \bm{\mu}\tr\bm{v}^{\pi,\bm{P}}\]
so that $\max_{\pi \in \PiS} \min_{\bm{P} \in \cP} \bm{\mu}\tr\bm{v}^{\pi,\bm{P}} = \bm{\mu}\tr\bm{u}\opt.$
Since $\pi\opt$ is such that $\bm{u}^{\pi\opt} = \bm{u}\opt \geq \bm{u}^{\pi}$ for any $\pi \in \PiS$ (see Appendix \ref{app:proof value functions}),
we conclude that $\pi\opt \in \arg \max_{\pi \in \PiS} \min_{\bm{P} \in \cP} \bm{\mu}\tr\bm{v}^{\pi,\bm{P}}.$
        \hfill \Halmos
\end{proof}
{The vector $\bm{u}\opt$ is the fixed point of the contracting operator $T$, so it can be computed efficiently with value iteration (see Appendix~\ref{app:fixed point iteration} for further details). }
We also have the following direct but important corollary, which shows that we can solve~{\em independently of the initial distribution $\bm{\mu}$.}
\begin{corollary}\label{cor: adv dpp implies max min opt independent of mu}
    Let $\cP$ be a compact uncertainty set and assume that \eqref{eq:adv dpp - rsa - srec} holds. Then 
    \[ \exists \; \pi\opt \in \PiS, \forall \; \bm{\mu} \in \Delta(\cS), \pi\opt \in \arg \max_{\pi \in \PiS} \min_{\bm{P} \in \cP} \bm{\mu}\tr\bm{v}^{\pi,\bm{P}}.\]
\end{corollary}
\begin{proof}{Proof of Corollary \ref{cor: adv dpp implies max min opt independent of mu}.}
   From the proof of Theorem \ref{th:solving max min}, for any $\bm{\mu} \in \Delta(\cS)$, a policy $\pi\opt \in \PiS$ solving $\max_{\pi \in \PiS} \min_{\bm{P} \in \cP} \bm{\mu}\tr\bm{v}^{\pi,\bm{P}}$ can be found by computing $\pi\opt$ such that $\bm{u}^{\pi\opt}=\bm{u}\opt$, and this last equation is independent of $\bm{\mu}$. \hfill \Halmos
\end{proof}
We conclude by noting that even though we can solve \eqref{eq:max min problem}, i.e. we can find the best stationary policy, we are not guaranteed to find an optimal policy over $\pi \in \PiH$. In particular, there may be a gap between history-dependent and stationary policies if $\cP$ is not convex, e.g. there are some instances s-rectangular instances for which
        \[\max_{\pi \in \PiS} \min_{\bm{P} \in \cP} \bm{\mu}\tr\bm{v}^{\pi,\bm{P}} < \sup_{\pi \in \PiH} \min_{\bm{P} \in \cP} \bm{\mu}\tr\bm{v}^{\pi,\bm{P}}.\]
We refer to Proposition 2 in~\cite{wiesemann2013robust} and Section 5.1 in \cite{wang2023foundation} for more details on this issue.
\paragraph{Existence of stationary optimal policies for convex uncertainty sets.}
We now show that \eqref{eq:adv dpp - rsa - srec} implies the existence of a stationary optimal policy when $\cP$ is compact {\em and} convex.
\begin{theorem}\label{th:existence of stationary optimal policies}
    Let $\cP$ be compact convex and assume that \eqref{eq:adv dpp - rsa - srec} holds. Then there exists a stationary optimal policy:
    \[ \sup_{\pi \in \PiH} \inf_{\bm{P} \in \cP} \bm{\mu}\tr\bm{v}^{\pi,\bm{P}} = \max_{\pi \in \PiS} \min_{\bm{P} \in \cP} \bm{\mu}\tr\bm{v}^{\pi,\bm{P}}.\]
    Additionally, there exists a stationary policy $\pi\opt \in \PiS$ which is optimal for any $\bm{\mu} \in \Delta(\cS)$.
\end{theorem}
Our strategy for proving Theorem \ref{th:existence of stationary optimal policies} relies on the following lemma.
\begin{lemma}\label{lem:strong duality implies stationary opt policies}
Let $\cP$ be a compact uncertainty set. Assume that strong duality holds:
    \begin{equation}\label{eq:strong duality max min}
    \max_{\pi \in \PiS} \min_{\bm{P} \in \cP} \bm{\mu}\tr\bm{v}^{\pi,\bm{P}} = \min_{\bm{P} \in \cP} \max_{\pi \in \PiS} \bm{\mu}\tr\bm{v}^{\pi,\bm{P}}.
\end{equation}
Then there exists a stationary optimal policy: \[ \sup_{\pi \in \PiH} \inf_{\bm{P} \in \cP} \bm{\mu}\tr\bm{v}^{\pi,\bm{P}} = \max_{\pi \in \PiS} \min_{\bm{P} \in \cP} \bm{\mu}\tr\bm{v}^{\pi,\bm{P}}.\]
\end{lemma}
Lemma \ref{lem:strong duality implies stationary opt policies} is used for the existence of stationary optimal policies for s-rectangular uncertainty sets~\citep{wiesemann2013robust} and r-rectangular uncertainty sets~\citep{goyal2022robust}. For completeness, we include the proof in Appendix \ref{app:proof duality implies opt stationary policies}.
From Corollary \ref{cor: adv dpp implies max min opt independent of mu}, $\max_{\pi \in \PiS} \min_{\bm{P} \in \cP} \bm{\mu}\tr\bm{v}^{\pi,\bm{P}}$ admits an optimal solution $\pi \opt \in \PiS$ which can be picked independent of $\bm{\mu} \in \Delta(\cS)$. Therefore, to prove Theorem \ref{th:existence of stationary optimal policies}, we only need to show that strong duality~\eqref{eq:strong duality max min} holds when weak s-DP-compatibility~\eqref{eq:adv dpp - rsa - srec} holds. The main idea of the proof is to notice that strong duality holds for the convex-concave saddle-point problems defining $\bm{u}\opt$ as the fixed point of the operator $T$. {This was proved in Proposition 9 in~\cite{wiesemann2013robust} for the case of s-rectangular uncertainty sets. It is worth noting that that strong duality may not hold when \eqref{eq:adv dpp - rsa - srec} fails to hold (i.e. when the uncertainty set is not weakly s-DP-compatible), as we show in Appendix \ref{app:counterexample strong duality}.}
\begin{proposition}\label{prop: adv ssp implies max min duality}
Let $\cP$ be compact convex and assume that \eqref{eq:adv dpp - rsa - srec} holds.
    Then for any $\bm{\mu} \in \Delta(\cS)$, the following strong duality result holds:
    \[ \max_{\pi \in \PiS} \min_{\bm{P} \in \cP} \bm{\mu}\tr\bm{v}^{\pi,\bm{P}} = \min_{\bm{P} \in \cP} \max_{\pi \in \PiS} \bm{\mu}\tr\bm{v}^{\pi,\bm{P}}.\]
\end{proposition}
\begin{proof}{Proof of Proposition \ref{prop: adv ssp implies max min duality}.}
Recall that \eqref{eq:adv dpp - rsa - srec} implies that $\max_{\pi \in \PiS} \min_{\bm{P} \in \cP} \bm{\mu}\tr\bm{v}^{\pi,\bm{P}}  = \bm{\mu}\tr\bm{u}\opt$ from Theorem \ref{th:solving max min}. We will show that $\bm{\mu}\tr\bm{u}\opt = \min_{\bm{P} \in \cP} \max_{\pi \in \PiS} \bm{\mu}\tr\bm{v}^{\pi,\bm{P}}$.
First, weak duality gives
\[ \bm{\mu}\tr\bm{u}\opt \leq \min_{\bm{P} \in \cP} \max_{\pi \in \PiS} \bm{\mu}\tr\bm{v}^{\pi,\bm{P}}.\]
To show the equality in the above equation, we show the existence of $\bm{P}\opt \in \cP$ such that 
\begin{equation*}
     u\opt_{s} = \max_{\pi \in \PiS} T^{\pi}_{\bm{P}\opt}(\bm{u}\opt)_{s}, \forall \; s \in \cS.
\end{equation*}
By definition of $\bm{u}\opt$, we have, for each $s \in \cS$,
$u\opt_s  =  \max_{\pi \in \PiS} \min_{\bm{P} \in \cP} T^{\pi}_{\bm{P}}(\bm{u}\opt)_{s}.$
Now recall that $\PiS = \times_{s \in \cS} \Delta(\cA)$, so there exists a policy $\pi\opt \in \PiS$ such that for each $s \in \cS$,
$u\opt_s = \min_{\bm{P} \in \cP} T^{\pi\opt}_{\bm{P}}(\bm{u}\opt)_{s}.$
Recall that \eqref{eq:adv dpp - rsa - srec} is equivalent to \eqref{eq:adv ssp - rsa - srec}, and \eqref{eq:adv ssp - rsa - srec} implies that
$\cap_{s \in \cS} \arg \min_{\bm{P} \in \cP} \langle \bm{P}_{sa},\bm{\pi}_{s}\opt\bm{u}^{\star \; \top} \rangle \neq \emptyset.$
From the definition of $T^{\pi\opt}_{\bm{P}}(\bm{u}\opt)$, we therefore also know that
$\cap_{s \in \cS} \arg \min_{\bm{P} \in \cP} T^{\pi\opt}_{\bm{P}}(\bm{u}\opt)_{s} \neq \emptyset.$
From the definition of the operator $T_{\bm{P}}^{\pi}$, we note that for each $s \in \cS$ the map $\bm{v} \mapsto T_{\bm{P}}^{\pi}(\bm{v})_{s}$ is convex-concave (it is actually bilinear). Additionally, the sets $\cP$ and $\PiS$ are convex and compact. Therefore, the minimax theorem holds (see Theorem 4.3.1 in \cite{hiriart1996convex}), that is, for each $s \in \cS$,
\[\max_{\pi \in \PiS} \min_{\bm{P} \in \cP} T^{\pi}_{\bm{P}}(\bm{u}\opt)_{s}=\min_{\bm{P} \in \cP} T^{\pi\opt}_{\bm{P}}(\bm{u}\opt)_{s} = \min_{\bm{P} \in \cP} \max_{\pi \in \PiS} T^{\pi}_{\bm{P}}(\bm{u}\opt)_{s}\]
so overall, we have shown
$\cap_{s \in \cS} \arg \min_{\bm{P} \in \cP} \max_{\pi \in \Delta(\cA)} T^{\pi}_{\bm{P}}(\bm{u}\opt)_{s} \neq \emptyset,$
i.e., we have shown the existence of some $\bm{P}\opt \in \cP$ such that 
$u\opt_{s} = \max_{\pi \in \PiS} T^{\pi}_{\bm{P}\opt}(\bm{u}\opt)_{s}, \forall \; s \in \cS.$
The above equation can be interpreted as the Bellman equation for the MDP where the transition probabilities is $\bm{P}\opt$. Thus 
$\bm{\mu}\tr\bm{u}\opt = \max_{\pi \in \PiS} \bm{\mu}\tr\bm{v}^{\pi,\bm{P}\opt} \geq \min_{\bm{P} \in \cP} \max_{\pi \in \PiS} \bm{\mu}\tr\bm{v}^{\pi,\bm{P}}$
which concludes our proof.
\hfill \Halmos
\end{proof}
\begin{remark}[The finite horizon case.]\label{rmk:finite horizon}
    We state all our results for the case of the infinite horizon, with values defined as in \eqref{eq:value function}. It is straightforward to extend our results to the case of a {\em finite horizon}, i.e. to the case where the value functions are defined as
    \[v^{\pi,\bm{P}}_{s} = \E^{\pi,\bm{P}}\left[ \left. \sum_{t=0}^{H} \gamma^t r_{s_{t}a_{t}s_{t+1}} \; \right| \; s_{0} =s  \right], \forall \; s \in \cS\]
    for some finite horizon $H \in \N$. To do so, one can consider a finite horizon MDP with copies of the state set at each epoch $t \in \{0,...,H\}$ and consider that the states attained at period $H$ are absorbing with a reward of $0$. {\em Stationary} policies and adversaries in this MDP interpreted as an infinite horizon decision problem corresponds to {\em Markovian} policies and {\em Markovian} adversaries in this MDP interpreted as a finite decision problem (without duplicating the state space). Therefore Theorem \ref{th:existence of stationary optimal policies} shows that the existence of {\em Markovian} optimal policies in robust MDPs with finite horizon, when the uncertainty set is weakly s-DP-compatible.

    We note that two important RMDP papers focus solely on the finite horizon case. \cite{ma2022distributionally} focus on a model called {\em d-rectangularity}, which exactly corresponds to r-rectangular uncertainty with a finite horizon. The authors in \cite{mannor2016robust} consider a model called {\em k-rectangularity}, where conditioned on some transition probabilities out of any subset of states $\cS' \subset \cS$, the transition probabilities out of the other states in $\cS \setminus \cS'$ take at most $k$ distinct values. For k-rectangular models, an optimal policy may be history-dependent, but one can compute it efficiently by augmenting the state space to keep track of the number of realizations of the values of the transition probabilities until the current period of decision. The results in \cite{mannor2016robust} rely on the decision-maker observing the entire realizations of the transition values over a trajectory and the finiteness over the possible remaining transitions induced by the k-rectangularity property. These two properties are very particular to the finite horizon case, and for this reason, k-rectangular models are not part of the models of uncertainty covered by this paper.
\end{remark}
\paragraph{History-dependent transitions are equivalent to stationary transitions.}
{We conclude this section by showing that stationary and {\em history-dependent} transition probabilities are equivalent under weak s-DP-compatibility, i.e., when ~\eqref{eq:adv dpp - rsa - srec} holds.  History-dependent transitions correspond to transition probabilities that may change at different periods, and we follow \cite{wang2023foundation} to define this notion in the RMDP setting. Formally, given an uncertainty set $\cP \subset \Delta(\cS)^{A \times S}$ (potentially non-rectangular),
the set $\cP_{\sf H}$ of $\cP$-constrained history-dependent transition probabilities is defined as
\begin{equation}\label{eq:history dependent transitions}
    \cP_{\sf H}
    := \{\left(\bm{P}_{t}\right)_{t \in \N} \; | \;
        \bm{P}_t(\cdot| s_0,a_0,\dots,a_{t-1},\cdot,\cdot) \in \cP,\;
        \forall (s_0,a_0,\dots,a_{t-1}) \in (\cS \times \cA)^t,\forall t \in \N\}.
\end{equation}
The constraint  
$\bm{P}_t(\cdot |s_0,a_0,\dots,a_{t-1},\cdot,\cdot) \in \cP$
is interpreted as restricting the transition probabilities $\bm{P}_t$ at period $t \in \N$ to be chosen in $\cP$ but based only on the past state-action pairs up to period $t-1$: $(s_0,a_0,\ldots,a_{t-1})$, and not on the current state–action pair $(s_t,a_t)$.}
We have the following theorem.
\begin{theorem}\label{th:stationary vs history dependent adversaries}
    Let $\cP$ be a compact convex uncertainty set and assume that weak s-DP-compatibility~\eqref{eq:adv dpp - rsa - srec} holds. Let $\pi \in \PiS$. Then 
    \begin{equation}\label{eq:stationary vs history dependent adversaries}
        \inf_{\bm{P} \in \cP_{\sf H}} \bm{\mu}\tr\bm{v}^{\pi,\bm{P}} = \min_{\bm{P} \in \cP} \bm{\mu}\tr\bm{v}^{\pi,\bm{P}}. 
        \end{equation}
\end{theorem}
\proof{Proof of Theorem \ref{th:stationary vs history dependent adversaries}.}
{Assume that weak s-DP-compatibility holds. As highlighted in Equation~\eqref{eq:PE over cP = PE over tildecP}, we know that  $\min_{\bm{P} \in \cP} \bm{\mu}\tr\bm{v}^{\pi,\bm{P}} = \min_{\bm{P} \in \twtilde{\cP}} \bm{\mu}\tr\bm{v}^{\pi,\bm{P}}$, with $\twtilde{\cP}$ the s-rectangular extension of $\cP$, as defined in~\eqref{eq:s-rectangular extension}.
We then claim that
\begin{equation}\label{eq:tildecP history dep =tildecP stationary}
    \min_{\bm{P} \in \twtilde{\cP}} \bm{\mu}\tr\bm{v}^{\pi,\bm{P}} = \inf_{\bm{P} \in \twtilde{\cP}_{\sf H}} \bm{\mu}\tr\bm{v}^{\pi,\bm{P}},
\end{equation}
where $\twtilde{\cP}_{\sf H}$ is the set of history-dependent transition probabilities over the s-rectangular extension $\twtilde{\cP}$ of the uncertainty set $\cP$, with a definition analogous to~\eqref{eq:history dependent transitions}:
\begin{equation}\label{eq:history dependent transitions-tilde cP}
    \twtilde{\cP}_{\sf H}
    := \{\left(\bm{P}_{t}\right)_{t \in \N} \; | \;
        \bm{P}_t(\cdot| s_0,a_0,\dots,a_{t-1},\cdot,\cdot) \in \twtilde{\cP},\;
        \forall (s_0,a_0,\dots,a_{t-1}) \in (\cS \times \cA)^t,\forall t \in \N\}.
\end{equation}
Equation~\eqref{eq:tildecP history dep =tildecP stationary} holds because $\twtilde{\cP}$ is s-rectangular, so that we actually also have
\begin{equation}\label{eq:history-dep = s-rec history-dep}
    \begin{aligned}
    \twtilde{\cP}_{\sf H}
    = \{\left(\bm{P}_{t}\right)_{t \in \N} \; | \;
       & \bm{P}_t(\cdot| s_0,a_0,\dots,a_{t-1},s_t,\cdot) \in \twtilde{\cP}_{s_t},\;
        \forall (s_0,a_0,\dots,a_{t-1}) \in (\cS \times \cA)^t, \\
        & \forall \; s_t \in \cS, \forall \; t \in \N\}
\end{aligned}
\end{equation}
i.e., in $\twtilde{\cP}_{\sf H}$ we can assume that the transition probabilities are chosen after observing the state $s_t$. 
\tb{
The authors in~\cite{ho2021partial} prove that we can interpret $\inf_{\bm{P} \in \twtilde{\cP}_{\sf H}} \bm{\mu}\tr\bm{v}^{\pi,\bm{P}}$ as an {\em adversarial MDP} defined as follows (see Definition 4.1 in~\cite{ho2021partial}): the set of states is $\cS$ and the set of actions available at each state $s \in \cS$ is $\twtilde{\cP}_{s}$. The initial probability distribution is $\mu \in \Delta(\cS)$ and the discount factor is $\gamma \in [0,1)$. The instantaneous reward obtained for choosing action $P_s \in \twtilde{\cP}_{s}$ in state $s \in \cS$ is $-\sum_{a \in \cA} \pi_{sa} P_{sa}\tr r_{sa} \in \R$. Recall that by convention MDPs are maximizing return, so we use a minus sign in the definition of the instantaneous rewards of the adversarial MDP to turn the minimization problem over $P$ into a maximization problem over $P$: $\inf_{\bm{P} \in \twtilde{\cP}_{\sf H}} \bm{\mu}\tr\bm{v}^{\pi,\bm{P}} = \sup_{\bm{P} \in \twtilde{\cP}_{\sf H}} \bm{\mu}\tr(-\bm{v}^{\pi,\bm{P}})$, and $-\bm{v}^{\pi,\bm{P}} \in \R^{\cS}$ is the value function when choosing policy $P \in \twtilde{\cP}_{\sf H}$ in this adversarial MDP. The probability to transition from state $s$ to state $s'$  after choosing action $P_s \in \twtilde{\cP}_{s}$ is $\sum_{a \in \cA} \pi_{sa}P_{sas'} \in [0,1]$. Note that the set of states and the initial distribution of the adversarial MDP coincide with those of the original nominal MDP, but the set of actions available at each state is now compact (compared to finite action sets for the original nominal MDP).
}
\tb{
Standard results for discounted MDPs with compact convex action sets, a finite state set, and continuous rewards and transitions show that in this adversarial MDP, an optimal policy may be chosen stationary (e.g., we are in the setup where Assumptions 6.0.1-6.0.4 of Chapter 6 in the classical textbook~\cite{puterman2014markov} hold, so Theorem 6.2.7 in~\cite{puterman2014markov} applies).  A history-dependent policy in this adversarial MDP is an element in $\twtilde{\cP}_{\sf H}$, and a stationary policy in this MDP is an element in $\twtilde{\cP}$. Therefore, the existence of an optimal stationary policy in the adversarial MDP means that minimizing the discounted return $P \mapsto \bm{\mu} \tr \bm{v}^{\pi,P}$ over $\twtilde{\cP}_{\sf H}$ (the set of history-dependent policies) yields the same objective value as minimizing the discounted return over $\twtilde{\cP}$ (the set of stationary policies).
This means that
$\displaystyle \min_{\bm{P} \in \tilde{\mathcal{P}}} \bm{\mu}^\top \bm{v}^{\pi, \bm{P}} = \min_{\bm{P} \in \tilde{\mathcal{P}}_{\sf H}} \bm{\mu}^\top \bm{v}^{\pi, \bm{P}}$, i.e., that Equation~\eqref{eq:tildecP history dep =tildecP stationary} holds.  We refer to Section 4 in ~\cite{ho2021partial} for more context on the interpretation of the policy evaluation problem as an adversarial MDP when the uncertainty set is s-rectangular.
}
 
Finally, since $\cP \subseteq \twtilde{\cP}$ we also have $\cP_{\sf H} \subseteq \twtilde{\cP}_{\sf H}$, hence
\[ \inf_{\bm{P} \in \twtilde{\cP}_{\sf H}} \bm{\mu}\tr\bm{v}^{\pi,\bm{P}} \leq  \inf_{\bm{P} \in \cP_{\sf H}} \bm{\mu}\tr\bm{v}^{\pi,\bm{P}}.\]
 Overall we have shown that $\displaystyle \min_{\bm{P} \in \cP} \bm{\mu}\tr\bm{v}^{\pi,\bm{P}} \leq \inf_{\bm{P} \in \cP_{\sf H}} \bm{\mu}\tr\bm{v}^{\pi,\bm{P}}$, while the reverse inequality is always true since $\cP \subset \cP_{\sf H}$ (identifying elements in $\cP$ with elements in $\cP_{\sf H}$ that remain constant across time), from which we can conclude that $\displaystyle  \inf_{\bm{P} \in \cP_{\sf H}} \bm{\mu}\tr\bm{v}^{\pi,\bm{P}} = \min_{\bm{P} \in \cP} \bm{\mu}\tr\bm{v}^{\pi,\bm{P}}$.
    \hfill \Halmos
}
\endproof
The existence of stationary optimal policies (for which we require convex uncertainty sets satisfying the weak SSP, as from Theorem \ref{th:existence of stationary optimal policies}) implies the following theorem. 
\begin{theorem}\label{th:sup inf sup min}
    Let $\cP$ be a compact convex uncertainty set and assume that weak s-DP-compatibility~\eqref{eq:adv dpp - rsa - srec} holds. Then 
\[ \sup_{\pi \in \PiH}  \inf_{\bm{P} \in \cP_{\sf H}} \bm{\mu}\tr\bm{v}^{\pi,\bm{P}} = \sup_{\pi \in \PiH} \min_{\bm{P} \in \cP} \bm{\mu}\tr\bm{v}^{\pi,\bm{P}}.
\]
\end{theorem}
\proof{Proof of Theorem \ref{th:sup inf sup min}.}
We have 
\[ \sup_{\pi \in \PiH} \min_{\bm{P} \in \cP} \bm{\mu}\tr\bm{v}^{\pi,\bm{P}} = \max_{\pi \in \PiS} \min_{\bm{P} \in \cP} \bm{\mu}\tr\bm{v}^{\pi,\bm{P}} = \max_{\pi \in \PiS} \inf_{\bm{P} \in \cP_{\sf H}} \bm{\mu}\tr\bm{v}^{\pi,\bm{P}} 
 \leq \sup_{\pi \in \PiH} \inf_{\bm{P} \in \cP_{\sf H}} \bm{\mu}\tr\bm{v}^{\pi,\bm{P}}\]
 where the first equality is from Corollary \ref{cor: adv dpp implies max min opt independent of mu}, the second equality is from Theorem \ref{th:stationary vs history dependent adversaries}, and the inequality is from $\PiS \subset \PiH$.
This implies that $\displaystyle  \sup_{\pi \in \PiH}  \inf_{\bm{P} \in \cP_{\sf H}} \bm{\mu}\tr\bm{v}^{\pi,\bm{P}} = \sup_{\pi \in \PiH} \min_{\bm{P} \in \cP} \bm{\mu}\tr\bm{v}^{\pi,\bm{P}}$.
\hfill \Halmos
\endproof
The questions addressed in our last two theorems have been studied in several recent papers for the case of s-rectangular  models~\cite{wang2023foundation,grand2023beyond}. Theorem \ref{th:stationary vs history dependent adversaries} provides a generalization to any uncertainty set satisfying the weak SSP, and enables a concise and unifying proof.
\subsubsection{Uncertainty sets satisfying the weak SSP}\label{sec:uncertainty sets for adv ssp}
{We now describe a large class of uncertainty sets satisfying weak DP-compatibility, generalizing all known weakly DP-compatible models of uncertainty (s-rectangular, sa-rectangular, r-rectangular models). Given the results in Theorem \ref{th:adv DPP, feasible upi, adv SSP equivalent - rsa - srec}, one can obtain the weak DP-compatibility (and all its consequences as listed in Section \ref{sec:implications of ddp - srec}) of a model of uncertainty ``by design", i.e. by ensuring that the weak SSP~\eqref{eq:adv ssp - rsa - srec} holds.
In particular, we consider the class of uncertainty sets where we can write $\bm{P}_{sa} = \sum_{i=1}^{r} u^{i}_{sa} \bm{w}^{i}$ but where both the factors $\bm{w}^{1},...,\bm{w}^{r}$ {\em and} the coefficients $\bm{u}_s=\left(u^{i}_{sa}\right)_{ia}$ for $s \in \cS$ may vary and belong to sets that are Cartesian product.}
{
\begin{definition}[Generalized rectangularity]\label{def:generalized rectangularity}
Let $\cP$ be an uncertainty set. We say that $\cP$ is {\em generalized rectangular} if we can write
\begin{equation}\label{eq:generalized rectangularity}
    \begin{aligned}
    \mathcal{P} &  = \left\{ \left.\left(\sum_{i=1}^r u^i_{sa}\bm{w}_{i} \right)_{sa} \right|(\bm{w}^{1},\ldots,\bm{w}^{r}) \in \cW,(\bm{u}_1,\ldots,\bm{u}_S)\in \mathcal{U} \right\},\\
    \cW & = \times_{i \in [r]} \cW^{i}, \cW^{i} \subset \Delta(\cS) \forall \; i \in [r], \\
    \cU & = \times_{s \in \cS} \cU^{s}, \cU^{s} \subset \Delta([r])^{A} \forall \; s \in \cS
    \end{aligned}
\end{equation}
for some compact sets $\cW^{i}$ for $i \in [r]$ and some compact sets $\cU^{s}$ for $s \in \cS$.
\end{definition}
}
{
\paragraph{Weak s-DP-compatibility of generalized rectangular uncertainty.}
We start by showing that generalized rectangular uncertainty sets satisfy the weak SSP.
\begin{proposition}\label{prop:part 2 weak ssp} Let $\cP$ be a generalized rectangular uncertainty set. Then the weak SSP~\eqref{eq:adv ssp - rsa - srec} holds.
\end{proposition}
\begin{proof}{Proof of Proposition \ref{prop:part 2 weak ssp}.}
Let $s \in \cS, \pi \in \PiS,\bm{V} \in \R^{S}$. We have
   {\begin{align*}
    \min_{\bm{P} \in \cP} \langle \bm{P}_{s},\bm{\pi}_{s}\bm{V}\tr  \rangle
    & = \min_{\bm{P} \in \cP} \sum_{a \in \A} \pi_{sa} \sum_{i=1}^{r}u_{sa}^{i} \bm{w}^{i \top}\bm{V}  \\
    & = \min_{\bm{u}_{s} \in \cU^{s},(\bm{w}^{1},...,\bm{w}^{r}) \in \cW^{1} \times ... \times \cW^{r}} \sum_{a \in \A} \pi_{sa} \sum_{i=1}^{r}u_{sa}^{i} \bm{w}^{i \top}\bm{V}  \\
    & = \min_{\bm{u}_{s} \in \cU^{s},(\bm{w}^{1},...,\bm{w}^{r}) \in \cW^{1} \times ... \times \cW^{r}} \sum_{a \in \A} \pi_{sa}\bm{u}_{sa}\tr (\bm{w}^{i \top}\bm{V})_{i \in [r]} \\
    & = \min_{\bm{u}_{s} \in \cU^{s}} \sum_{a \in \A} \pi_{sa}\sum_{i=1}^{r}\bm{u}_{sa}\tr \left( \min_{\bm{w}^{i} \in \cW^{i}}\bm{w}^{i \top}\bm{V} \right)_{i \in [r]}
\end{align*}
}
Note that the inner optimization problem $\min_{\bm{w}^{i} \in \cW^{i}}\bm{w}^{i \top}\bm{V}$ is independent of $s \in \cS$ and $\bm{\pi}_s$, and it appears in each of the optimization problem $\min_{\bm{P} \in \cP} \langle \bm{P}_{s},\bm{\pi}_{s}\bm{V}\tr  \rangle$ across each $s \in \cS$. Therefore, if we pick $\bm{w}^{i \star} \in \arg \min_{\bm{w}^{i} \in \cW^{i}}\bm{w}^{i \top}\bm{V}$ for each $i \in [r]$ and then, for each $s \in \cS$,
\[ \bm{u}_{s}\opt \in \arg \min_{\bm{u}_{s} \in \cU^{s}} \sum_{a \in \A} \pi_{sa}\bm{u}_{sa}\tr \left( \bm{w}^{i \star \top}\bm{V} \right)_{i \in [r]}\]
we obtain that $\bm{P}\opt:= \left(\sum_{i=1}^{r}u_{sa}^{i \star} \bm{w}^{i \star}\right)_{sa}$ is feasible: $\bm{P}\opt \in \cP$, and that $\displaystyle\bm{P}\opt \in \cap_{s \in \cS} \arg \min_{\bm{P} \in \cP} \langle \bm{P}_{s},\bm{\pi}_{s}\bm{V}\tr  \rangle$.
\hfill \Halmos
\end{proof}
}
{
As a corollary, we directly obtain that all the important consequences of DP-compatibility highlighted in Section~\ref{sec:implications of ddp - srec} hold for generalized rectangular models.
}
{
\paragraph{Special cases of generalized rectangularity.}
As the name suggests, generalized rectangularity encompasses as special cases several well-known models of uncertainty.
\begin{lemma}\label{lem:generalized rec implies s-rec and r-rec}
\begin{enumerate}
    \item Let $\cP$ be s-rectangular. Then $\cP$ is generalized rectangular.
    \item Let $\cP$ be r-rectangular. Then $\cP$ is generalized rectangular.
\end{enumerate}
\end{lemma}
\begin{proof}{Proof of Lemma \ref{lem:generalized rec implies s-rec and r-rec}.}
\begin{enumerate}
    \item Let $\cP$ be s-rectangular: $\cP = \times_{s \in \cS} \cP_{s}$. Let $r = S$. For $i \in [r] = [S] = \cS$ (we identify $\cS$ and $[S]=\{1,...,S\}$), we define $\cW^{i}$ as the singleton containing only the $i$-th vector of the canonical basis of $\R^{S}$: $\cW^{i} = \{ \bm{e}_i\}$ with $e_{ij}=0$ if $j \neq i$ and $e_{ii}=1$. Let $\bm{P} \in \cP$. By definition, $\bm{P}_{sa} = \sum_{s' \in \cS} \left(\bm{P}_{sa}\tr\bm{e}_{s'}\right)\bm{e}_{s'} = \sum_{i \in [r]} \left( \bm{P}_{sa}\tr\bm{e}_{i} \right) \bm{e}_{i}$. If we define $\cU^{s} = \{\left(\left( \bm{P}_{sa}\tr\bm{e}_{i} \right)\right)_{ia} \; | \; \bm{P}_{s} \in \cP_{s}\} = \cP_{s}$ for each $s \in \cS$, we indeed obtain that $\cP$ is generalized rectangular. 
    \item Let $\cP$ be r-rectangular. Then $\cP$ can be written as \eqref{eq:generalized rectangularity} with $\cU$ a singleton: $\cU = \{\left(u^{i}_{sa}\right)_{isa}\}$. This also proves the same statement about sa-rectangular uncertainty sets (which are a special case of r-rectangular uncertainty sets with $r=|\cS| \times |\cA|$). \hfill \Halmos
\end{enumerate}
\end{proof}
}
{
The previous lemma immediately shows that sa-rectangular models are generalized rectangular. An interesting subclass of generalized rectangular models is the case where we can partition the set of states into $\cS = \cS_1 \cup \cS_2$, where the transitions out of $\cS_1$ are s-rectangular and the transitions out of $\cS_2$ are r-rectangular. This model can be used if the decision-maker believes that the transition probabilities are related across states, but only for a subset $\cS_{2}$ of the entire state set $\cS$. 
\begin{definition}[(s+r)-rectangularity]
    Let $\cP$ be an uncertainty set. We say that $\cP$ is {\em (s+r)-rectangular} if $\cP$ can be written as follows, for some partition $(\cS_1,\cS_2)$ of the set $\cS$ and $S_1:=\card(\cS_1),S_2:=\card(\cS_2)$:
    \begin{equation}\label{eq:(s+r)-rec}
    \begin{aligned}
    \cP &  = \cP_{1} \times \cP_{2}, \cP_{1} \subset \Delta(\cS)^{S_{1} \times A}, \cP_{2} \subset \Delta(\cS)^{S_{2}\times A}, \\
    \cP_{1} & = \times_{s \in \cS_{1}} \cP_{s}, \cP_{s} \subset \Delta(\cS)^{\A}, \\
    \cP_{2} & = \left\{ \left({\sum_{i=1}^{r}}
    u^{i}_{sa}\bm{w}^{i}\right)_{s \in \cS_{2},a \in \A} \; | \; (\bm{w}^{1},...,\bm{w}^{r}) \in \cW^{1} \times ... \times \cW^{r} \right\}
\end{aligned}
\end{equation}
for some compact sets $\cP_{1}, \cW^{1},...,\cW^{r}$.
\end{definition}
}
{
\begin{lemma}\label{lem:generalized rec implies (s+r)-rec}
    Let $\cP$ be (s+r)-rectangular. Then $\cP$ is generalized rectangular.
\end{lemma}
\begin{proof}{Proof of Lemma \ref{lem:generalized rec implies (s+r)-rec}.}
  Let $\cP$ be as in \eqref{eq:(s+r)-rec}.  Combining the first item from Lemma~\ref{lem:generalized rec implies s-rec and r-rec} (for the s-rectangular transitions out of $\cS_{1}$) and the second from Lemma~\ref{lem:generalized rec implies s-rec and r-rec} (for the r-rectangular transitions out of $\cS_{2}$) shows that $\cP$ can be written as in~\eqref{eq:generalized rectangularity}, i.e. that $\cP$ is generalized rectangular. 
\hfill \Halmos
\end{proof}
Importantly, we note that generalized rectangular sets are not necessarily convex even when the sets $\cU$ and $\cW$ are convex, because of the bilinear term $u^{i}_{sa}\bm{w}^{i}$ in their definition. In contrast, the sets $\cP$ as in~\eqref{eq:(s+r)-rec} are always convex  as soon as $\cW$ and $\cP_1$ are convex (we have $\cP = \cP_1 \times \cP_2$, and $\cP_2$ is the image of $\cW$ by a linear map).} 
{
\paragraph{Concrete examples and strict inclusions.} We now study some concrete examples of the models of uncertainty introduced above.
We provide a simple example of (s+r)-rectangular uncertainty set below, which shows that (s+r)-rectangular models strictly generalizes s-rectangular and r-rectangular models.
\begin{example}[A simple (s+r)-rectangular model]\label{ex:(s+r)-rectangular}
    Consider the robust MDP instance represented in Figure \ref{fig:mixing r-rec s-rec part 1}, {with 5 states $\{a,b,c,d,e\}$}. The transitions and rewards are indicated above the solid arcs for the two actions $\{a_{1},a_{2}\}$. We parametrize the uncertainty set $\cP$ describing all possible transitions by the two scalars $\xi \in [0,1]$ and $p \in [0,1]$.
    The authors in \cite{wiesemann2013robust} show that starting in State $a$ it is optimal to randomize and choose $a_{1}$ with probability $1/2$ and $a_{2}$ with probability $1/2$ (see Proposition 1 in \cite{wiesemann2013robust}). This shows that this uncertainty set $\cP$ cannot be r-rectangular, since for r-rectangular uncertainty sets, there always exists an optimal policy that can be chosen deterministic~\citep{goyal2022robust}. Additionally, we note that $\cP$ is not s-rectangular, since the parameter $p \in [0,1]$ influences the transitions out of two different states ({State $d$ and State $e$}). However, it is clear that $\cP$ is (s+r)-rectangular, by partitioning $\cS = \cS_{1} \cup \cS_{2}$ with $\cS_{1} = \{a,b,c\}$ and $\cS_{2} = \{d,e\}$.
\begin{figure}
\begin{center}
    \begin{subfigure}{0.4\textwidth}
\includegraphics[width=\linewidth]{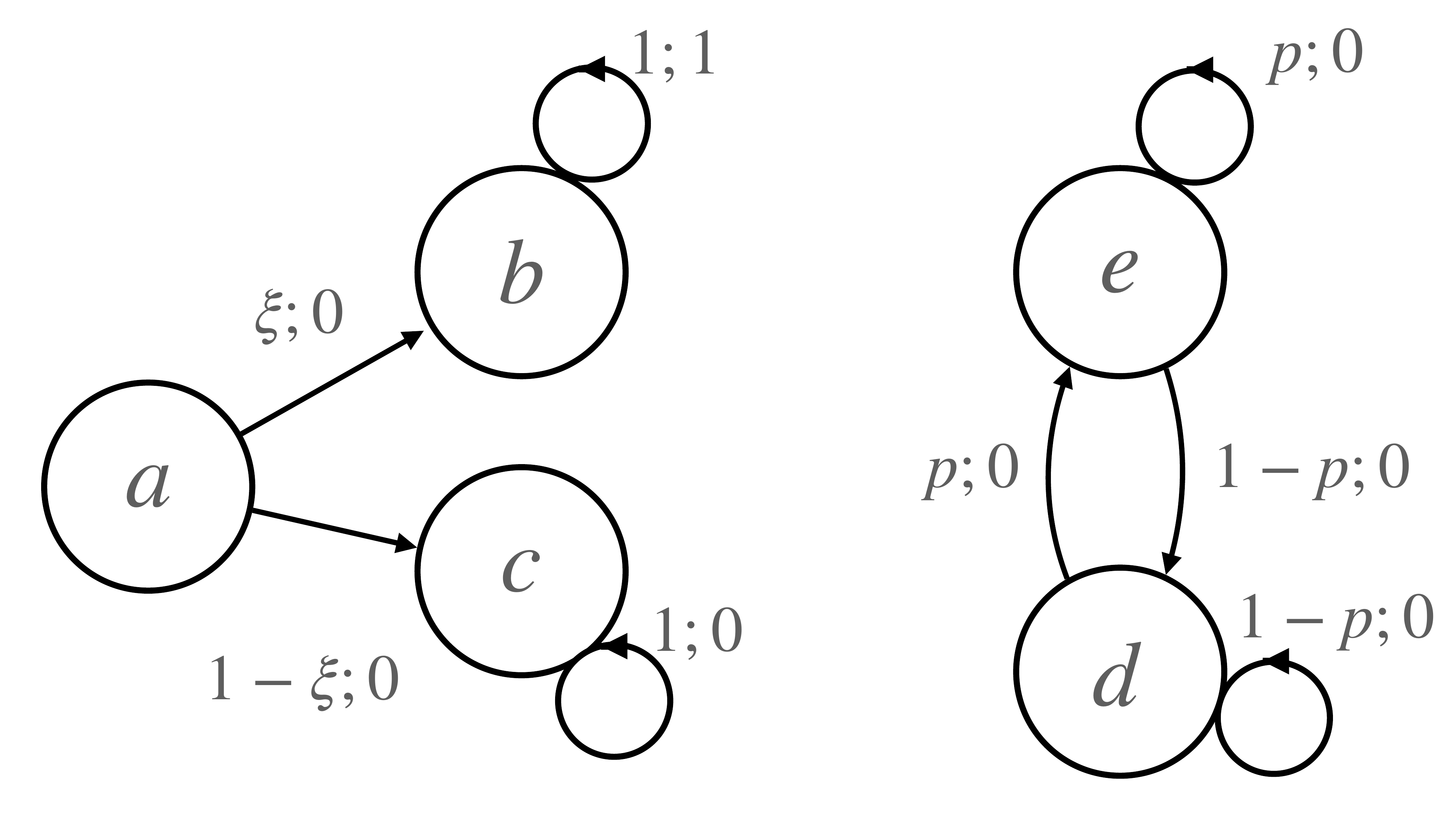}
    \caption{Transition for action $a_1$}
    \label{fig:rmdp_a1_dummy}
    \end{subfigure}
    \begin{subfigure}{0.4\textwidth}
\includegraphics[width=\linewidth]{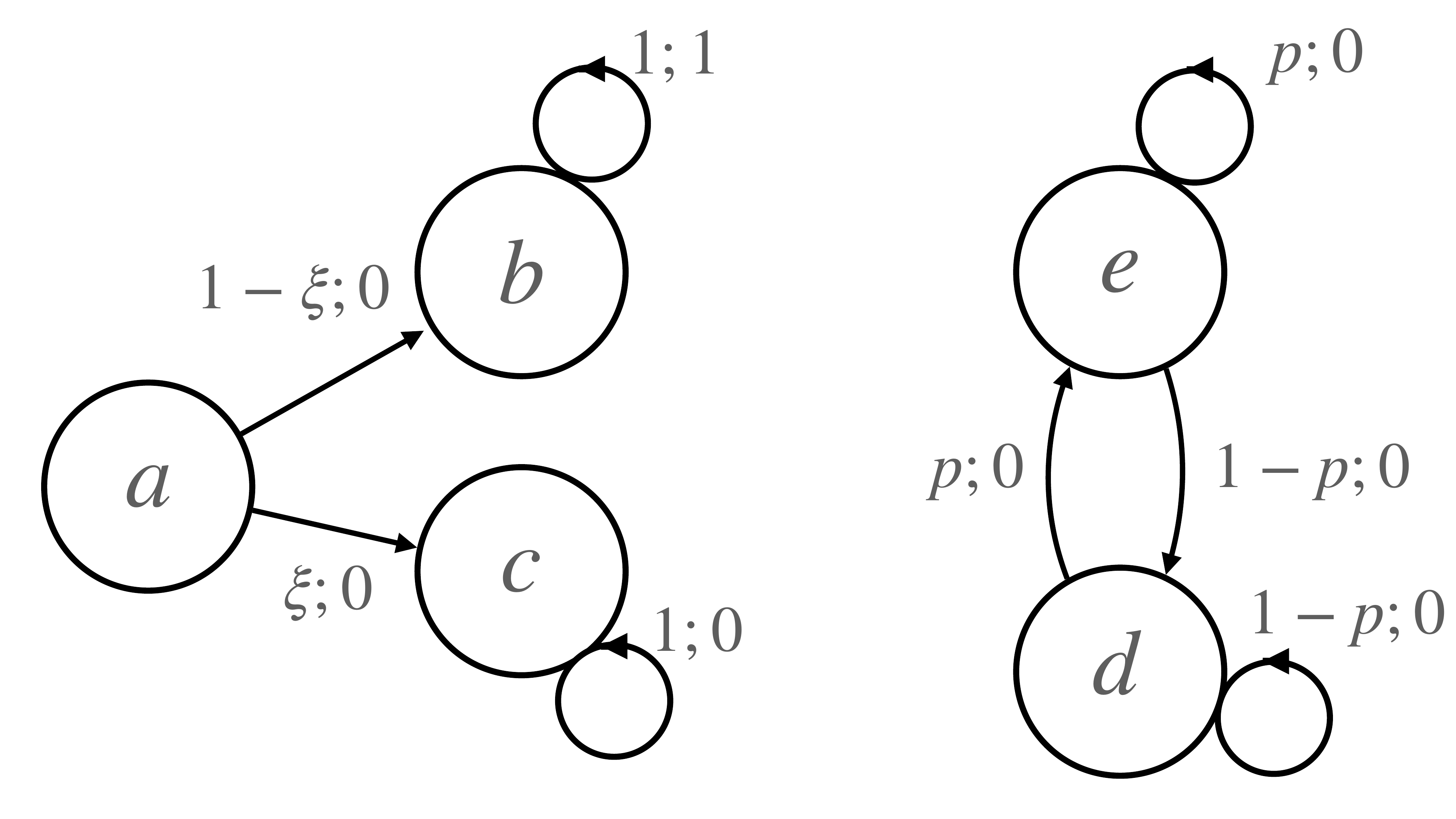}
    \caption{Transition for action $a_2$}
    \label{fig:rmdp_a2_dummy}
    \end{subfigure}
    \end{center}
\caption{Transitions and rewards for the RMDP instance for Example~\ref{ex:(s+r)-rectangular}.}\label{fig:mixing r-rec s-rec part 1}
\end{figure}
\end{example}
}
{
We now provide an example which shows that (s+r)-rectangular models are a strict subset of generalized rectangular models.
\begin{example}[A simple generalized rectangular model]\label{ex:part 2 generalization 2}
We consider the robust MDP instance and uncertainty set described in Figure \ref{fig:mixing r-rec s-rec part 2} with 4 states: $\cS=\{a,b,c,d\}$ and three actions: $\A=\{a_{1},a_{2},a_{3}\}$. The uncertainty set $\cP$ is parametrized by two parameters $\xi \in [0,1]$ and $p \in [0,1]$. We indicate the transitions and rewards above the solid arcs for the three actions. 
Note that only two states have uncertain transitions, State $a$ and State $d$. We can write the uncertainty set (restricted to the transitions out of State $a$ and State $d$) as follows:
    \begin{align*}
       \cP =  \{&\bm{P}_{aa_{1}} = \xi\bm{w}^{1} + (1-\xi)\bm{w}^{2}, \bm{P}_{aa_{2}} = (1-\xi)\bm{w}^{1} + \xi\bm{w}^{2}, \bm{P}_{aa_{3}} = \bm{w}^3, \bm{P}_{da_{3}} = \bm{w}^{3} \; | \\ & \; \bm{w}^{1} \in \{\bm{e}_{b}\}, \bm{w}^{2} \in \{\bm{e}_{c}\},\bm{w}^{3} = (p,0,0,1-p), p \in [0,1], \xi \in [0,1]\}
    \end{align*}
with $\bm{e}_{b}$ and $\bm{e}_{c}$ the vectors of the canonical basis indexed by $\X$: $\bm{e}_{b} = (0,1,0,0)$ and $\bm{e}_{c} = (0,0,1,0)$.
\begin{figure}
\begin{center}
    \begin{subfigure}{0.3\textwidth}
\includegraphics[width=\linewidth]{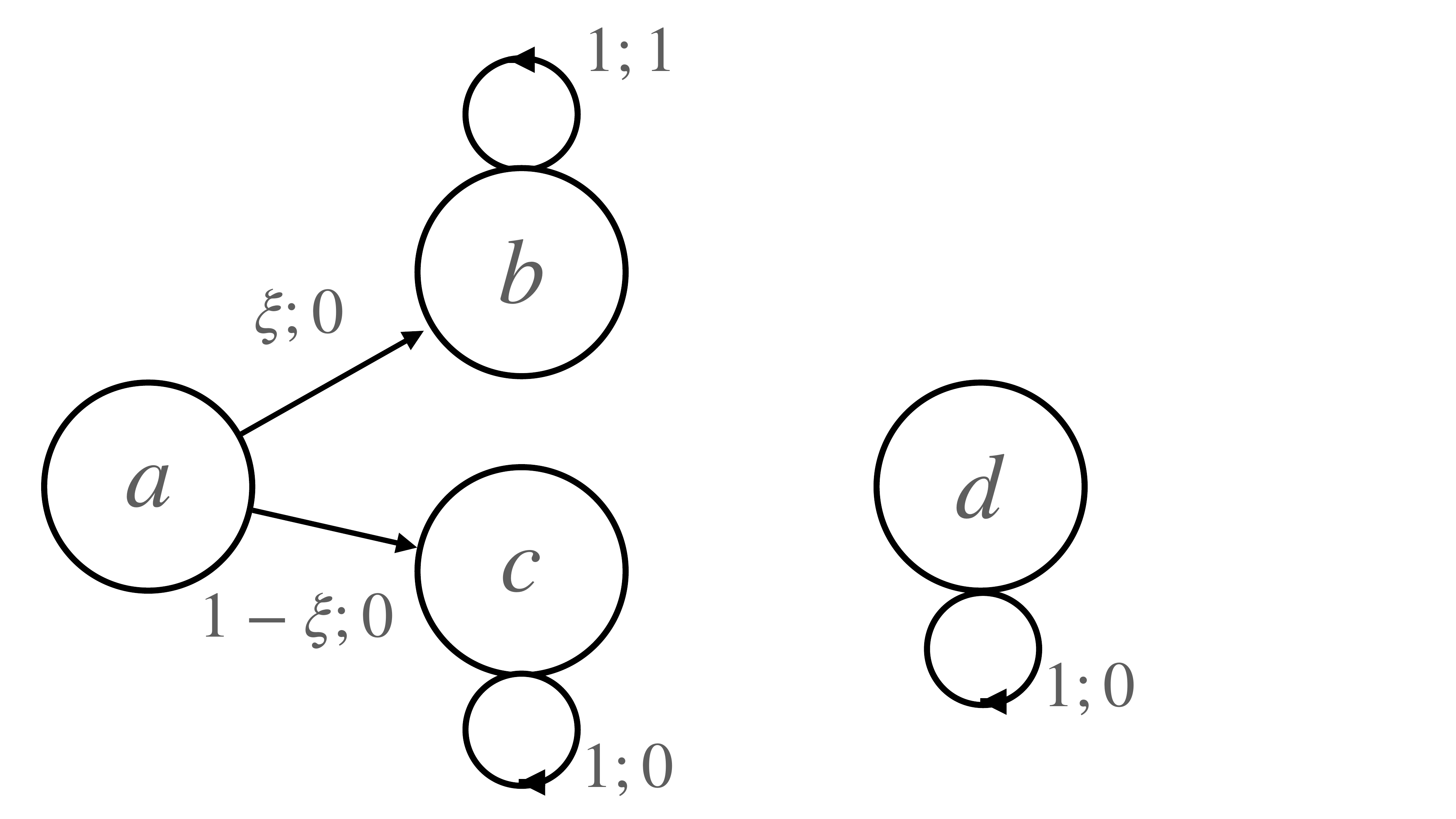}
    \caption{Transition for action $a_1$}
    \label{fig:rmdp_a1_part 2}
    \end{subfigure}
    \begin{subfigure}{0.3\textwidth}
\includegraphics[width=\linewidth]{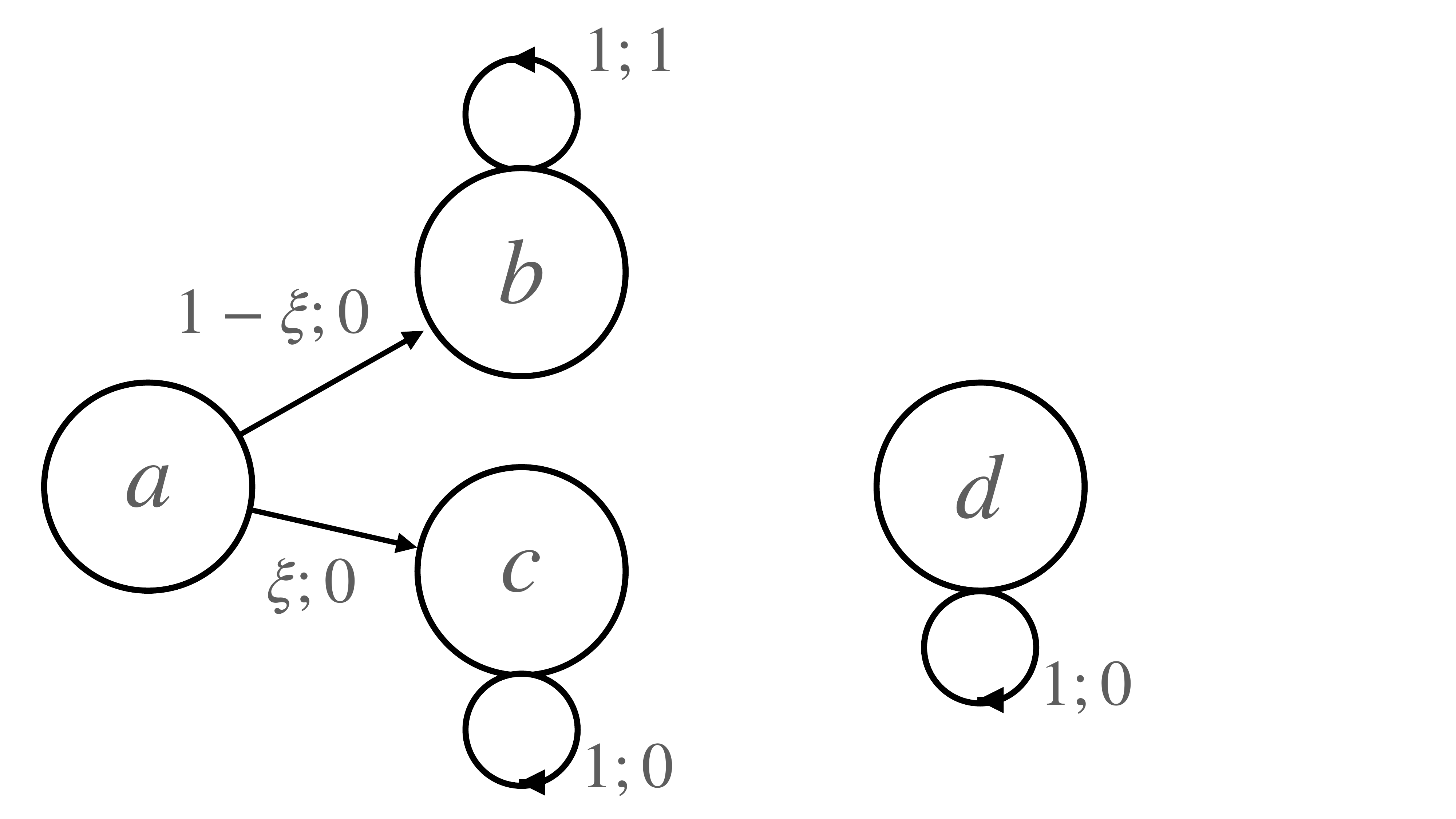}
    \caption{Transition for action $a_2$}
    \label{fig:rmdp_a2_part 2}
    \end{subfigure}
    \begin{subfigure}{0.3\textwidth}
\includegraphics[width=\linewidth]{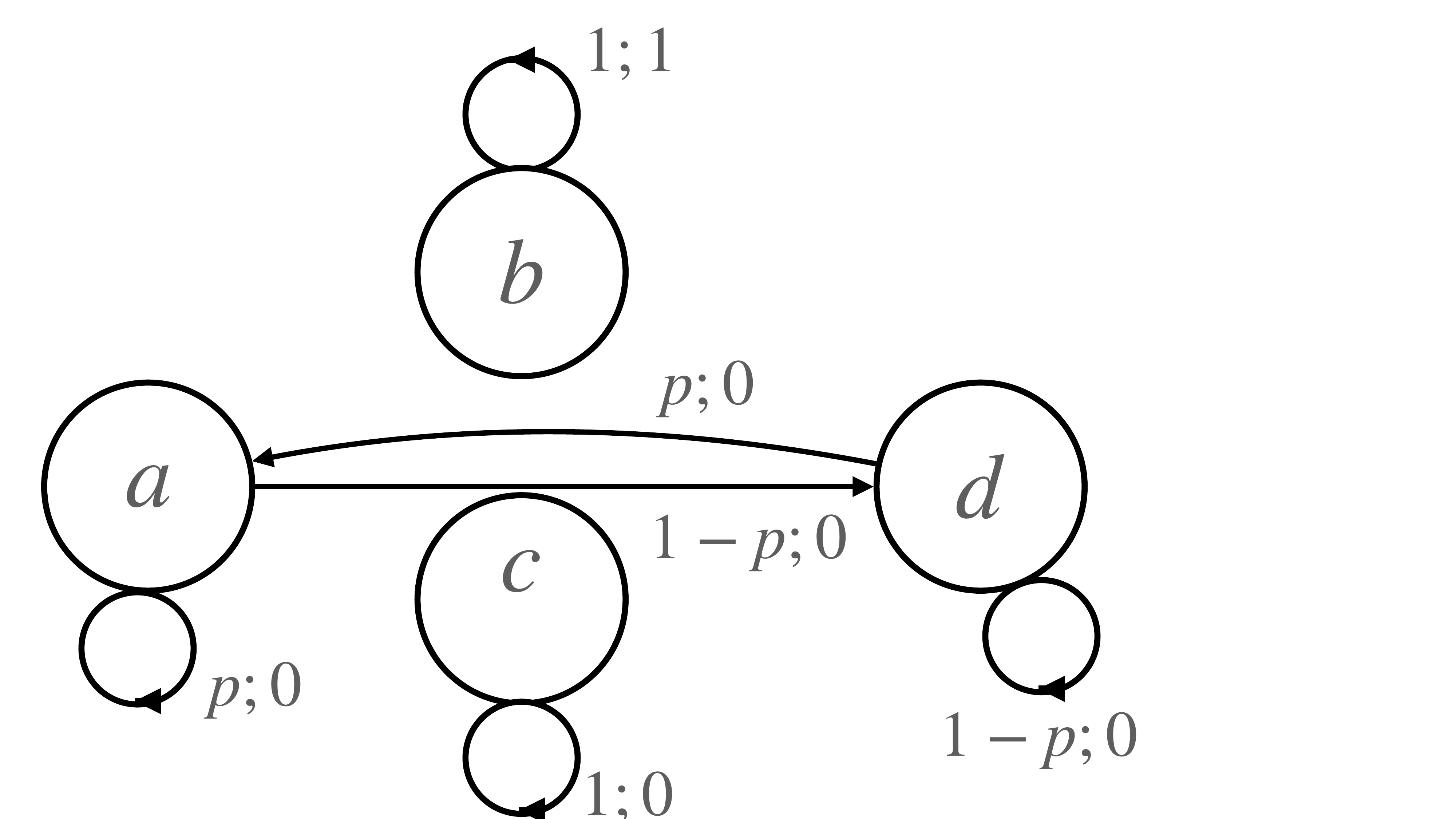}
    \caption{Transition for action $a_3$}
    \label{fig:rmdp_a3_part 2}
    \end{subfigure}
    \end{center}
\caption{Transitions and rewards for the RMDP instance for Example \ref{ex:part 2 generalization 2}. }\label{fig:mixing r-rec s-rec part 2}
\end{figure}
If the uncertainty set $\cP$ was (s+r)-rectangular, then we could partition $\{a,d\}$ into two subsets $\cS_{1}$ and $\cS_{2}$ such that the transitions out of $\cS_{1}$ are s-rectangular and the transitions out of $\cS_{2}$ are r-rectangular. Since $\{a,d\}$ has only two elements, this would mean that $\cP$ is s-rectangular or that $\cP$ is r-rectangular. However, neither of these two cases is possible.
    Indeed, $\cP$ is not r-rectangular since \cite{wiesemann2013robust} proved that an optimal robust policy must be randomized if starting in State $a$ (whereas optimal robust policies for r-rectangular sets can always be chosen deterministic). 
    Additionally, $\cP$ cannot be s-rectangular because the transitions out of the state-action pairs $(a,a_{3})$ and $(d,a_{3})$ must be equal.
\end{example}
}
{
\begin{remark}[$(\xi,\eta)$-rectangularity]
    The authors in~\cite{hu2024efficient} introduce the $(\xi,\eta)$-rectangularity models of uncertainty, which is a model where one can write $\bm{P}_{sa} = \sum_{i=1}^{r}u^{i}_{sa}\bm{w}^{i}$ where both the coefficients $\bm{u}^{1},...,\bm{u}^{r}$ and the factors $\bm{w}^{1},...,\bm{w}^{r}$ can vary, but where there are constraints of the form $\|(w^{i}_{s'})_{i \in [r]} - (\hat{w}^{i}_{s'})_{i \in [r]}\|_{2} \leq R, \|(u^{i}_{sa})_{i \in [r]} - (\hat{u}^{i}_{sa})_{i \in [r]}\|_{2} \leq R'$ for some radii $R,R'>0$ and some nominal factors $\hat{\bm{w}}^{1},...,\hat{\bm{w}}^{r}$ and nominal coefficients $\twhat{\bm{u}}^{1},...,\twhat{\bm{u}}^{r}$  (see section 3.2 in~\cite{hu2024efficient}). These constraints imply that $(\bm{w}^{1},...,\bm{w}^{r})$ do not belong to a Cartesian product sets (and similarly for the coefficients), and therefore, $(\xi,\eta)$-rectangular models are not generalized rectangular models. It is worth emphasizing that the positioning in \cite{hu2024efficient} is very different from ours: \cite{hu2024efficient} focus on the case of {\em finite} horizon objective, on uncertainty in the transitions {\em and} in the rewards, on policy gradient algorithms, and does not discuss the properties of optimal policies (e.g. history-dependent vs. stationary policies) for their model. Our technical setting is different (we focus on finite {\em and} infinite objective function, and on general compact sets $\cW$ and $\cU$), and most importantly, we focus on finding necessary and sufficient conditions for efficiently solving the policy evaluation problem and the relation between DP-compatibility and all the important consequences for robust MDPs (strong duality, existence of stationary policies, etc.).
\end{remark}
}
\subsection{Weakly sa-DP-compatible uncertainty sets}\label{sec:weak ssp - sarec}
We now investigate the weak sa-DP-compatibility property. In particular, the following theorem shows the equivalence between weak sa-DP-compatibility, the feasibility of the vector $\twhat{\bm{u}}^{\pi}$, and a certain weak SSP condition.
\begin{theorem}\label{th:adv DPP, feasible upi, adv SSP equivalent - rsa - sarec}
Let $\cP$ be compact. The following statements are equivalent. 
    \begin{enumerate}
    \item The set $\cP$ is weakly sa-DP-compatible:
    \begin{equation}\label{eq:adv dpp rsa - sarec}
        \forall \; \pi \in \PiS,\forall \; \bm{r} \in \R^{S \times A}, \forall \; \gamma \in [0,1), \forall \; \bm{\mu} \in \Delta(\cS), \min_{\bm{P} \in \cP} \bm{\mu}\tr\bm{v}^{\pi,\bm{P}} = \bm{\mu}\tr\twhat{\bm{u}}^{\pi}
    \end{equation}
        \item The vector $\twhat{\bm{u}}^{\pi}$ is feasible under Assumption \ref{assumption:rewards}:
        \begin{equation}\label{eq:feas upi rsa - sarec}
            \forall \; \pi \in \PiS,\forall \; \bm{r} \in \R^{S \times A}, \forall \; \gamma \in [0,1), \exists \; \twhat{\bm{P}} \in \cP, \twhat{\bm{u}}^{\pi} = \bm{v}^{\pi,\twhat{\bm{P}}}.
        \end{equation}
        \item The following weak SSP holds:
        \begin{equation}\label{eq:adv ssp - rsa - sarec} 
            \forall \; \bm{V}  \in \R^{S},  \cap_{(s,a) \in \cS \times \A} \arg \min_{\bm{P} \in \cP} \; \langle \bm{P}_{sa},\bm{V}\rangle   \neq \emptyset.
        \end{equation}
    \end{enumerate}
\end{theorem}
{The proof of Theorem~\ref{th:adv DPP, feasible upi, adv SSP equivalent - rsa - sarec} is provided in Appendix~\ref{app:proof th:adv DPP, feasible upi, adv SSP equivalent - rsa - sarec}}.
We now discuss the important consequences of weak sa-DP-compatibility.
\paragraph{Implications of weak sa-DP-compatibility.}  We first note that following Proposition \ref{prop:sa tractable are s tractable}, weakly sa-DP-compatible uncertainty sets are also weakly s-DP-compatible, and $\twhat{\bm{u}}^{\pi} = \bm{u}^{\pi}$ for any $\pi \in \PiS$. Therefore, all the conclusions from Section \ref{sec:implications of ddp - srec} also hold for weakly sa-DP-compatible uncertainty sets, {namely:}
{
\begin{itemize}
    \item We can solve the policy evaluation problem over a weakly sa-DP-compatible $\cP$ by solving the policy evaluation problem over its sa-rectangular extension $\twtilde{\cP}$, i.e. we have
    \[ \min_{\bm{P} \in \cP} \bm{\mu}\tr\bm{v}^{\pi,\bm{P}} = \min_{\bm{P} \in \twhat{\cP}} \bm{\mu}\tr\bm{v}^{\pi,\bm{P}}.\] 
    One can therefore resorts to the wealth of optimization algorithms designed for sa-rectangular uncertainty sets, e.g. policy iteration and convex programming (Lemma 3.2 in~\cite{iyengar2005robust}) or value iteration as presented in Section~\ref{sec:implications of ddp - srec}.
    \item We can efficiently compute an optimal stationary policy that is independent on the initial probability distribution $\bm{\mu}$  by computing the fixed point $\bm{u}\opt$ (Theorem~\ref{th:solving max min} and Corollary~\ref{cor: adv dpp implies max min opt independent of mu}). The fixed point $\bm{u}\opt$ is the optimal robust value function of the sa-rectangular RMDP with uncertainty set $\twtilde{\cP}$, and again there exist several efficient algorithms, e.g.~robust value iteration (Figure 1 in~\cite{iyengar2005robust}) and robust policy iteration (Figure 2 in~\cite{iyengar2005robust}). Additionally, if $\cP$ is convex, then strong duality holds, and there exists a stationary optimal policy (Theorem~\ref{th:existence of stationary optimal policies}).
\end{itemize}
} 
An important additional property that holds for weakly sa-DP-compatible uncertainty sets is that an optimal policy may be chosen deterministic, as we prove next.
\begin{proposition}\label{prop:sa-DP-compatible deterministic policy}
    Let $\cP$ be a compact and convex uncertainty set. Assume that $\cP$ is weakly sa-DP-compatible. Then there exists a stationary deterministic policy that is optimal for any choice of $\bm{\mu} \in \Delta(\cS)$.
\end{proposition}
\begin{proof}{Proof of Proposition \ref{prop:sa-DP-compatible deterministic policy}.}
Recall that from the proof of Theorem \ref{th:existence of stationary optimal policies}, we can choose a stationary policy  $\pi\opt$ that remains optimal for any $\bm{\mu} \in \Delta(\cS)$ as $\pi\opt$ such that $\bm{u}\opt = \bm{u}^{\pi\opt}$, i.e. such that for any $s \in \cS$,
\begin{equation}\label{eq:piopt in arg max T pi}
    \pi\opt \in \arg \max_{\pi \in \PiS}T^{\pi}(\bm{u}^{\pi\opt})_{s}.
\end{equation}
From the weak sa-DP-compatibility of $\cP$ and Proposition \ref{prop:sa tractable are s tractable}, we know that $\bm{u}^{\pi\opt} = \twhat{\bm{u}}^{\pi\opt}$. Additionally, for any $\pi \in \PiS$ and $s \in \cS$, we have
\[\twhat{T}^{\pi}(\twhat{\bm{u}}^{\pi\opt}) \leq T^{\pi}(\twhat{\bm{u}}^{\pi\opt}) = T^{\pi}(\bm{u}^{\pi\opt}) \leq T^{\pi\opt}(\bm{u}^{\pi\opt}) = \bm{u}^{\pi\opt} = \twhat{\bm{u}}^{\pi\opt} = \twhat{T}^{\pi\opt}(\twhat{\bm{u}}^{\pi\opt})\]
where the first inequality follows from $\twhat{T}^{\pi}(\bm{v}) \leq T^{\pi}(\bm{v})$ for any $\bm{v} \in \R^{S}$, the first equality follows from $\bm{u}^{\pi\opt} = \twhat{\bm{u}}^{\pi\opt}$, the second equality follows from \eqref{eq:piopt in arg max T pi}, and the last equalities follow the definition of $\bm{u}^{\pi\opt}$ and $\twhat{\bm{u}}^{\pi\opt}$ as the fixed points of $T^{\pi\opt}$ and $\twhat{T}^{\pi\opt}.$
We therefore conclude that $\pi\opt \in \arg \max_{\pi \in \PiS}\twhat{T}^{\pi}(\bm{u}^{\pi\opt})_{s}.$
Given the expression of the operator $\twhat{T}^{\pi}$ as in \eqref{eq:operator hat T pi}, we conclude that $\pi\opt$ can be chosen deterministic, which concludes our proof.
\hfill \Halmos
\end{proof}
{An interesting subclass of models satisfying the weak SSP~\eqref{eq:adv ssp - rsa - sarec} is described below. It consists of generalized rectangular models where the coefficients can be chosen in a Cartesian product set over all pairs of state-actions.
}
{
\begin{definition}\label{def:generalized r-rectangularity}
Let $\cP$ be an uncertainty set. We say that $\cP$ is {\em generalized r-rectangular} if $\cP$ is a generalized rectangular uncertainty set~\eqref{eq:generalized rectangularity} such that $\cU = \times_{(s,a) \in \cS \times \A} \cU^{sa}$ for some compact sets $\cU^{sa} \subset \Delta([r])$ for each $(s,a) \in \cS \times \cA$.
\end{definition}
}
{Clearly, r-rectangular models are a strict subclass of generalized r-rectangular uncertainty sets (corresponding to the case where $\cU$ is a singleton). Additionally, generalized r-rectangular models are a strict subclass of generalized rectangular models (since the latter contains s-rectangular models, while the former can not contain all s-rectangular models, for which optimal policies may have to be randomized).}
It is straightforward to verify that the weak SSP~\eqref{eq:adv ssp - rsa - sarec} holds for generalized r-rectangular uncertainty sets: for $\bm{V} \in \R^{S}$, we can recover a feasible optimal transition probabilities {$\bm{P}\opt = \left(\sum_{i=1}^{r}u_{sa}^{\star i}\bm{w}^{ \star i}\right)_{sa}$ with
\begin{align*}
    \bm{w}^{\star i} & \in \arg \min_{\bm{w}^{i} \in \cW^{i}} \bm{w}^{i \; \top}\bm{V}, \forall \; i \in [r],\\
    \bm{u}_{sa}^{\star} & \in \arg \min_{\bm{u} \in \cU_{sa}} \bm{u}_{sa}\tr\left(\bm{w}^{\star i\; \top}\bm{V}\right)_{i \in [r]}, \forall \; (s,a) \in \cS \times \A.
\end{align*}
In particular, this shows that r-rectangular uncertainty sets satisfy \eqref{eq:adv ssp - rsa - sarec}, and this that r-rectangular uncertainty sets are weakly sa-DP-compatible: the problems of policy evaluation and of finding an optimal policy can be solved by solving a robust MDP over their sa-rectangular extension. This is a surprising result since r-rectangular models were originally introduced to overcome the conservativeness of rectangular models and since r-rectangularity generalizes sa-rectangularity.}

{
We summarize the different inclusions between the models of uncertainty described in our paper in Figure \ref{fig:inclusion_model_simplified}, and we summarize our main results for this section in Figure \ref{fig:contributions_third_part}.
}
\begin{figure}
    \centering
    \includegraphics[width=0.7\linewidth]{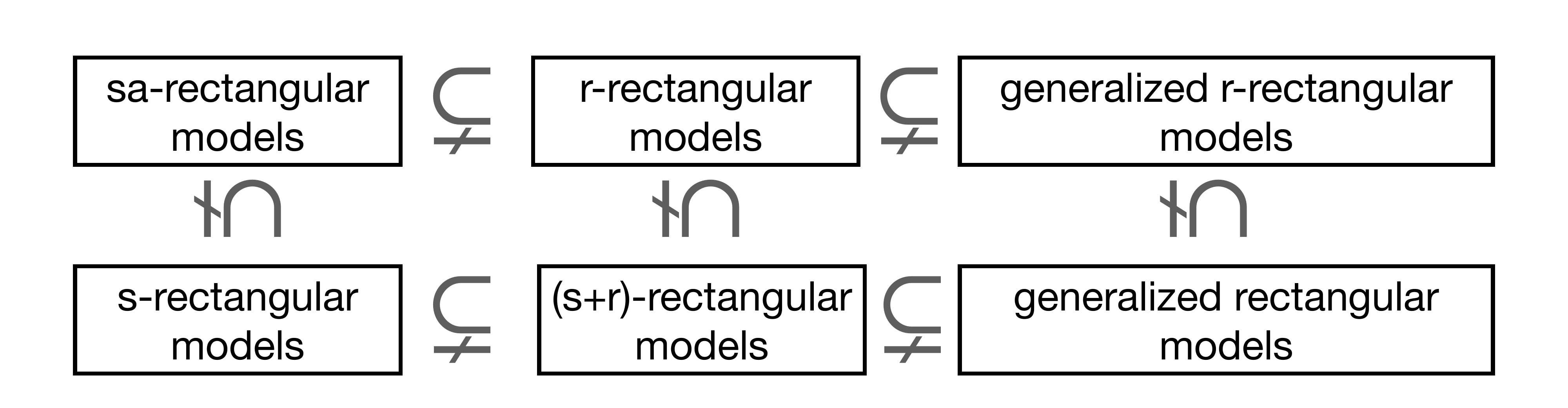}
    \caption{Inclusions between the models of uncertainty studied in this paper.}
    \label{fig:inclusion_model_simplified}
\end{figure}
\begin{figure}
    \centering
    \includegraphics[width=1\linewidth]{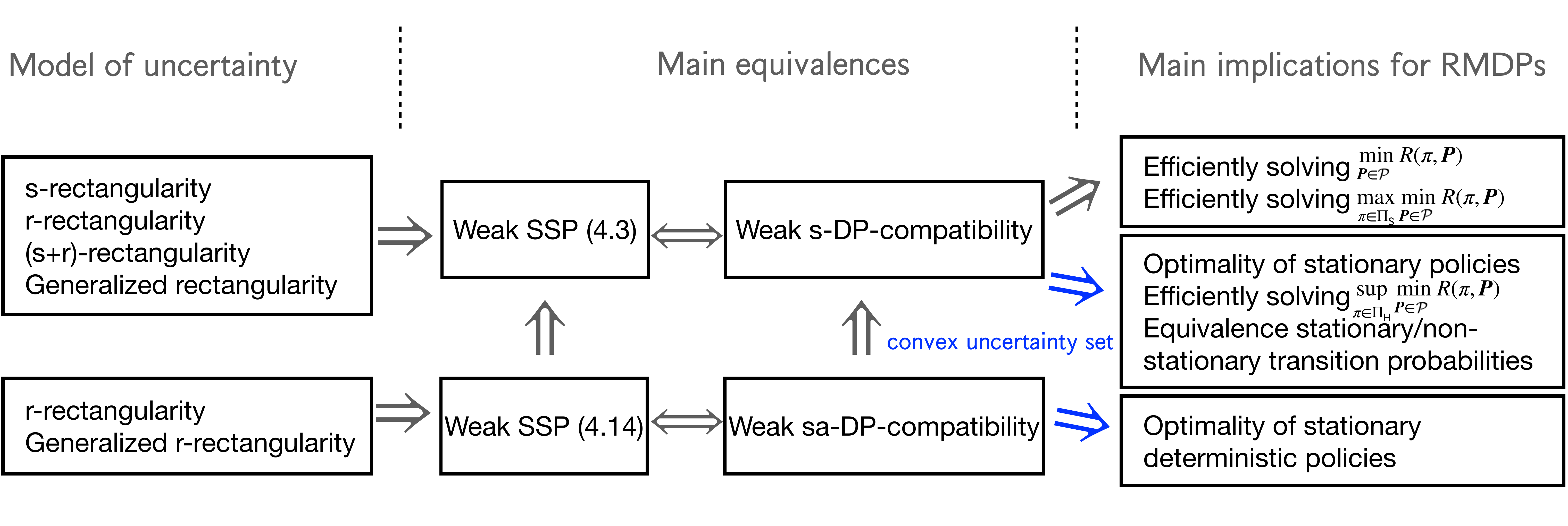}
    \caption{Summary of our main results from Section \ref{sec:weak ssp}. The uncertainty sets are assumed to be compact. The implications in blue require the additional assumption that the uncertainty set is convex. We denote $R(\pi,\bm{P}) = \bm{\mu}\tr\bm{v}^{\pi,\bm{P}}$}
    \label{fig:contributions_third_part}
\end{figure}
\section{Discussion}\label{sec:discussion}
We now discuss some of the most important results of our paper.
\paragraph{Properties of robust MDPs.} Our results highlight interesting connections between the important properties of robust MDPs. 

In particular, we are the first to show that the weak DP-compatibility of the policy evaluation problem implies that we can efficiently solve the max-min problem, see Theorem \ref{th:solving max min}. This suggests that the main difficulty in solving robust MDPs lies in the ability to compute worst-case returns over $\cP$ for each policy $\pi \in \PiS$. Because of the Cartesian product structure of the set of stationary policies: $\PiS = \times_{s \in \cS} \Delta(\cA)$, it is possible to obtain a dynamic programming equation for solving the max-min problem, if we have a dynamic programming equation for the minimization problem, i.e. if $\cP$ is (weakly) DP-compatible. Interestingly, the main results for r-rectangular models were obtained in this order: the r-rectangular model was first introduced in~\cite{goh2018data} where the authors only show how to solve policy evaluation via dynamic programming, before~\cite{goyal2022robust} show how to solve the robust MDP (max-min) problem via dynamic programming.

Our key theorems yielding sufficient and necessary conditions for (weak) DP-compatibility highlight the main difficulty behind obtaining a dynamic programming equation for policy evaluation: this is possible {\em if and only if} one can solve several linear programs over $\cP$ simultaneously, a condition that we call {\em simultaneous solvability property} (SSP). At a high level, these conditions arise from the feasibility of the operators $T^{\pi}$ and $\twhat{T}^{\pi}$ (which independently minimize over $\cP$ on each component) within the set of operators $\{T^{\pi,\bm{P}} \; | \; \bm{P} \in \cP\}$, in the precise sense that
\begin{align*}
     \eqref{eq:adv ssp - rsa - srec} & \iff \forall \; \bm{V} \in \R^{S}, \forall \; \pi \in \PiS, T^{\pi}(\bm{V}) \in \{T^{\pi,\bm{P}}(\bm{V}) \; | \; \bm{P} \in \cP\} \\
          \eqref{eq:adv ssp - rsa - sarec} & \iff \forall \; \bm{V} \in \R^{S}, \forall \; \pi \in \PiS, \twhat{T}^{\pi}(\bm{V}) \in \{T^{\pi,\bm{P}}(\bm{V}) \; | \; \bm{P} \in \cP\}
\end{align*}
We note that the weak~SSPs \eqref{eq:adv ssp - rsa - srec} and \eqref{eq:adv ssp - rsa - sarec} could be reformulated in various ways, and we have done our best to provide a formulation of these conditions that is both readable and easily verifiable.

Finally, we also highlight explicitly that the existence of stationary optimal policies follows from strong duality, see Lemma \ref{lem:strong duality implies stationary opt policies}. We note that a similar approach is undertaken in \cite{wiesemann2013robust,goyal2022robust} without emphasizing the key role of strong duality. Interestingly, strong duality follows from Sion's minimax theorem for convex-concave saddle-point problems over convex compact sets, {\em even though robust MDPs are not convex-concave problems}. This is because the optimal value function $\bm{u}\opt$ is the solution to the convex-concave saddle-point problems for each component of $T(\bm{u}\opt)$, with $T$ the Bellman operator as defined in \eqref{eq:operator T}. While some of the properties that we show in Section \ref{sec:implications of ddp - srec} were already known for some models of uncertainty (e.g. for r-rectangular models), the main contribution of this paper is to provide a unifying approach to the DP-compatibility of uncertainty models, as we describe next.
\paragraph{(Weakly) DP-compatible uncertainty models by design.} \tb{From a practical standpoint, our results suggest a change of paradigm in the design of uncertainty sets.} In particular, different models of uncertainty, such as s-rectangularity or r-rectangularity, have been introduced independently, each requiring a particular approach for showing their important properties, such as efficient algorithms for solving policy evaluation and the existence of stationary optimal policies. In contrast, based on Theorem \ref{th:adv DPP, feasible upi, adv SSP equivalent - rsa - srec} and Theorem \ref{th:adv DPP, feasible upi, adv SSP equivalent - rsa - sarec}, the weak DP-compatibility of an uncertainty set $\cP$ can be verified easily and in a unified way, by verifying that the weak SSP conditions hold, which we show are {\em necessary and sufficient} for weak s-DP-compatibility and weak sa-DP-compatibility. 
{
We also uncover a wealth of non-rectangular uncertainty sets that are weakly s-DP-compatible uncertainty sets ((s+r)-rectangular models, generalized r-rectangular models, generalized rectangular models).
}
It is likely that several other weakly DP-compatible models of uncertainty could be discovered since, based on our results, one can obtain the weak DP-compatibility of new uncertainty models ``by design", ensuring that they satisfy the weak SSP conditions. However, we believe that our main contribution lies in obtaining simple {\em necessary and sufficient} conditions for verifying the weak DP-compatibility of an uncertainty set, and that our results actually undermine the case for non-rectangular uncertainty models, as we discuss below.

\paragraph{The case for rectangular uncertainty sets.} Overall, our results from Section \ref{sec:ssp} show that in all generality (i.e. without any assumption on the rewards), only rectangular uncertainty sets are DP-compatible. Under an additional assumption (see Assumption \ref{assumption:rewards}), there exists several non-rectangular uncertainty sets that may be weakly DP-compatible, but all these models of uncertainty sets are ``equivalent'' to rectangular models, since we proved that dynamic programming-based approaches always reduce to solving policy evaluation {\em over rectangular models}. \tb{We emphasize that this is one of the key conceptual and practical contributions of our paper: we prove that we can find an optimal policy or evaluate the worst-case return of a policy for these {\em non-rectangular} models by {\em solving the same optimization problems over their s-rectangular or sa-rectangular extensions}.} This is true for all the non-rectangular models known in the literature, i.e., for the r-rectangularity model introduced in \cite{goh2018data,goyal2022robust}, and for the new non-rectangular models studied in this paper. This undermines the narrative behind the introduction of non-rectangular uncertainty models, which were introduced to overcome the conservativeness of rectangular models. 
\paragraph{Beyond dynamic programming.} As an important next direction of research for the RMDP community, we argue that the search for new models of uncertainty should focus on finding models that are not s-DP-compatible or sa-DP-compatible but for which we can still efficiently compute an optimal policy, i.e. to focus on finding models of uncertainty that do not rely on dynamic programming for their solutions. This represents a significant departure from the usual solution methods for robust MDPs. In particular, we have shown in Section \ref{sec:implications of ddp - srec} that the (weak) DP-compatibility of an uncertainty set has several consequences, including the existence of stationary optimal policies {\em that do not depend on the initial distribution} (Corollary \ref{cor: adv dpp implies max min opt independent of mu} and Theorem \ref{th:existence of stationary optimal policies}). This property is attractive from an optimization standpoint and follows from an optimal value function dominating the value function of any other policy (see Appendix \ref{app:proof value functions}). However, in practice, it is conceivable that the decision-maker is willing to stick to stationary policies but would like to adapt their policies depending on where the decision process starts, i.e., depending on the initial distribution, and in fact, the existence itself of stationary optimal policies may also be dependent on the choice of initial distribution. We provide examples of non-rectangular uncertainty sets satisfying these properties in Appendix \ref{app:example discussion}, and we leave generalizing these examples for future work.
\section{Conclusion}
We show that the existence of dynamic programming-based approach for policy evaluation is equivalent to a property of the uncertainty sets that we term {\em simultaneous solvability property} (SSP). This provides a unifying approach to investigating the compatibility of various models of uncertainty with dynamic programming. Crucially, we prove that in general, only the s-rectangular and sa-rectangular models are compatible with dynamic programming. 
Some non-rectangular models can be {\em weakly} DP-compatible, requiring the assumption that the rewards do not depend on the next state, and in this case weak compatibility is equivalent to a weaker form of SSP. However, even weakly DP-compatible non-rectangular models are equivalent to rectangular models in the sense that their worst-case and optimal value functions can be computed by solving worst-case and optimal value functions over their rectangular extensions. Our work opens interesting directions for future research, perhaps the most exciting being finding uncertainty models that do not rely on dynamic programming, thereby potentially identifying entirely novel approaches for solving robust MDPs. 
\bibliographystyle{abbrv}
\bibliography{ref}

\begin{thebibliography}{10}

\bibitem{auer2008near}
P.~Auer, T.~Jaksch, and R.~Ortner.
\newblock Near-optimal regret bounds for reinforcement learning.
\newblock {\em Advances in neural information processing systems}, 21, 2008.

\bibitem{behzadian2021fast}
B.~Behzadian, M.~Petrik, and C.~P. Ho.
\newblock Fast algorithms for $l_{\infty}$ constrained s-rectangular robust
  {MDP}s.
\newblock {\em Advances in Neural Information Processing Systems},
  34:25982--25992, 2021.

\bibitem{ben2021lectures}
A.~Ben-Tal and A.~Nemirovski.
\newblock Lectures on modern convex optimization 2020.
\newblock {\em SIAM, Philadelphia}, 2021.

\bibitem{berinde2007iterative}
V.~Berinde.
\newblock {\em Iterative approximation of fixed points}.
\newblock Springer, 2007.

\bibitem{chambolle2016ergodic}
A.~Chambolle and T.~Pock.
\newblock On the ergodic convergence rates of a first-order primal--dual
  algorithm.
\newblock {\em Mathematical Programming}, 159(1):253--287, 2016.

\bibitem{delage2010percentile}
E.~Delage and S.~Mannor.
\newblock Percentile optimization for {Markov} decision processes with
  parameter uncertainty.
\newblock {\em Operations research}, 58(1):203--213, 2010.

\bibitem{filippi2010optimism}
S.~Filippi, O.~Capp{\'e}, and A.~Garivier.
\newblock Optimism in reinforcement learning and kullback-leibler divergence.
\newblock In {\em 2010 48th Annual Allerton Conference on Communication,
  Control, and Computing (Allerton)}, pages 115--122. IEEE, 2010.

\bibitem{givan2000bounded}
R.~Givan, S.~Leach, and T.~Dean.
\newblock Bounded-parameter markov decision processes.
\newblock {\em Artificial Intelligence}, 122(1-2):71--109, 2000.

\bibitem{goh2018data}
J.~Goh, M.~Bayati, S.~A. Zenios, S.~Singh, and D.~Moore.
\newblock Data uncertainty in {M}arkov chains: Application to
  cost-effectiveness analyses of medical innovations.
\newblock {\em Operations Research}, 66(3):697--715, 2018.

\bibitem{goyal2022robust}
V.~Goyal and J.~Grand-Cl{\'e}ment.
\newblock Robust {M}arkov decision processes: Beyond rectangularity.
\newblock {\em Mathematics of Operations Research}, 2022.

\bibitem{grand2024convex}
J.~Grand-Cl{\'e}ment and M.~Petrik.
\newblock On the convex formulations of robust markov decision processes.
\newblock {\em Mathematics of Operations Research}, 2024.

\bibitem{grand2023beyond}
J.~Grand-Cl{\'e}ment, M.~Petrik, and N.~Vieille.
\newblock Beyond discounted returns: Robust markov decision processes with
  average and blackwell optimality.
\newblock {\em arXiv preprint arXiv:2312.03618}, 2023.

\bibitem{hiriart1996convex}
J.-B. Hiriart-Urruty and C.~Lemar{\'e}chal.
\newblock {\em Convex analysis and minimization algorithms I: Fundamentals},
  volume 305.
\newblock Springer science \& business media, 1996.

\bibitem{ho2018fast}
C.~P. Ho, M.~Petrik, and W.~Wiesemann.
\newblock Fast bellman updates for robust mdps.
\newblock In {\em International Conference on Machine Learning}, pages
  1979--1988. PMLR, 2018.

\bibitem{ho2021partial}
C.~P. Ho, M.~Petrik, and W.~Wiesemann.
\newblock Partial policy iteration for {L}1-robust {Markov} decision processes.
\newblock {\em The Journal of Machine Learning Research}, 22(1):12612--12657,
  2021.

\bibitem{ho2022robust}
C.~P. Ho, M.~Petrik, and W.~Wiesemann.
\newblock Robust phi-divergence {MDP}s.
\newblock {\em arXiv preprint arXiv:2205.14202}, 2022.

\bibitem{hu2024efficient}
Y.~Hu, H.~Ma, B.~Dai, and N.~Li.
\newblock Efficient duple perturbation robustness in low-rank mdps.
\newblock {\em arXiv preprint arXiv:2404.08089}, 2024.

\bibitem{iyengar2005robust}
G.~Iyengar.
\newblock Robust dynamic programming.
\newblock {\em Mathematics of Operations Research}, 30(2):257--280, 2005.

\bibitem{kumar2025non}
N.~Kumar, A.~Gupta, M.~M. Elfatihi, G.~Ramponi, K.~Y. Levy, and S.~Mannor.
\newblock Non-rectangular robust mdps with normed uncertainty sets.
\newblock In {\em The Thirty-ninth Annual Conference on Neural Information
  Processing Systems}, 2025.

\bibitem{li2025towards}
M.~Li, D.~Kuhn, and T.~Sutter.
\newblock Towards optimal offline reinforcement learning.
\newblock {\em arXiv preprint arXiv:2503.12283}, 2025.

\bibitem{li2022first}
Y.~Li, T.~Zhao, and G.~Lan.
\newblock First-order policy optimization for robust {Markov} decision process.
\newblock {\em arXiv preprint arXiv:2209.10579}, 2022.

\bibitem{li2022robust}
Y.~Li, T.~Zhao, and G.~Lan.
\newblock Robust policy mirror descent for controlling uncertain markov
  decision process.
\newblock {\em arXiv preprint arXiv:2209.10579}, 2022.

\bibitem{ma2022distributionally}
X.~Ma, Z.~Liang, J.~Blanchet, M.~Liu, L.~Xia, J.~Zhang, Q.~Zhao, and Z.~Zhou.
\newblock Distributionally robust offline reinforcement learning with linear
  function approximation.
\newblock {\em arXiv preprint arXiv:2209.06620}, 2022.

\bibitem{mannor2016robust}
S.~Mannor, O.~Mebel, and H.~Xu.
\newblock Robust {MDP}s with k-rectangular uncertainty.
\newblock {\em Mathematics of Operations Research}, 41(4):1484--1509, 2016.

\bibitem{nilim2005robust}
A.~Nilim and L.~E. Ghaoui.
\newblock Robust control of {M}arkov decision processes with uncertain
  transition probabilities.
\newblock {\em Operations Research}, 53(5):780--798, 2005.

\bibitem{puterman2014markov}
M.~L. Puterman.
\newblock {\em Markov decision processes: discrete stochastic dynamic
  programming}.
\newblock John Wiley \& Sons, 2014.

\bibitem{rockafellar2015convex}
R.~T. Rockafellar.
\newblock {\em Convex Analysis:(PMS-28)}.
\newblock Princeton university press, 2015.

\bibitem{straszewicz1935exponierte}
S.~Straszewicz.
\newblock {\"U}ber exponierte punkte abgeschlossener punktmengen.
\newblock {\em Fundamenta Mathematicae}, 24(1):139--143, 1935.

\bibitem{tewari2007bounded}
A.~Tewari and P.~L. Bartlett.
\newblock Bounded parameter {M}arkov decision processes with average reward
  criterion.
\newblock In {\em International Conference on Computational Learning Theory},
  pages 263--277. Springer, 2007.

\bibitem{wang2023policy}
Q.~Wang, C.~P. Ho, and M.~Petrik.
\newblock Policy gradient in robust mdps with global convergence guarantee.
\newblock In {\em International Conference on Machine Learning}, pages
  35763--35797. PMLR, 2023.

\bibitem{wang2023foundation}
S.~Wang, N.~Si, J.~Blanchet, and Z.~Zhou.
\newblock On the foundation of distributionally robust reinforcement learning.
\newblock {\em arXiv preprint arXiv:2311.09018}, 2023.

\bibitem{wang2023robust}
Y.~Wang, A.~Velasquez, G.~Atia, A.~Prater-Bennette, and S.~Zou.
\newblock Robust average-reward {{M}arkov} decision processes.
\newblock {\em arXiv preprint arXiv:2301.00858}, 2023.

\bibitem{wang2022policy}
Y.~Wang and S.~Zou.
\newblock Policy gradient method for robust reinforcement learning.
\newblock In {\em International conference on machine learning}, pages
  23484--23526. PMLR, 2022.

\bibitem{wiesemann2013robust}
W.~Wiesemann, D.~Kuhn, and B.~Rustem.
\newblock Robust {M}arkov decision processes.
\newblock {\em Mathematics of Operations Research}, 38(1):153--183, 2013.

\end{thebibliography}

\begin{APPENDICES}

\section{Proofs of Theorem \ref{th:adv DPP, feasible upi, adv SSP equivalent - rsas' - sarec} and Theorem \ref{th:reformulation adv ssp rsas' - sarec} }\label{app:proof rsas' sarec}
The proof of Theorem \ref{th:adv DPP, feasible upi, adv SSP equivalent - rsas' - sarec} is very close to the proof of Theorem \ref{th:adv DPP, feasible upi, adv SSP equivalent - rsas' - srec}. We only describe the main changes below.
\begin{proof}{Proof of Theorem \ref{th:adv DPP, feasible upi, adv SSP equivalent - rsas' - sarec}.}
The proof proceeds in three parts.

{\bf Part 1:} \eqref{eq:adv dpp rsas' - sarec} $\Rightarrow$ \eqref{eq:feas upi rsas' - sarec}. This implication follows verbatim from the fact that \eqref{eq:adv dpp rsas' - srec} $\Rightarrow$ \eqref{eq:feas upi rsas' - srec} in the proof of Theorem \ref{th:adv DPP, feasible upi, adv SSP equivalent - rsas' - srec}.

{\bf Part 2:} \eqref{eq:feas upi rsas' - sarec} $\Rightarrow$ \eqref{eq:adv ssp rsas' - sarec}. Let $\bm{V} \in \R^{S \times A \times S}$, $\pi \in \PiS$ be the uniform policy, $\gamma=0$, and $\bm{r}=\bm{V}$. By \eqref{eq:feas upi rsas' - sarec} it follows that $\twhat{\bm{u}}^{\pi} = \bm{v}^{\pi,\twhat{\bm{P}}}$ for some $\twhat{\bm{P}} \in \cP$. Note that by definition,
\begin{align*}
    v^{\pi,\twhat{\bm{P}}}_{s} & = \sum_{a \in \A} \pi_{sa}\twhat{\bm{P}}_{sa}\tr\bm{r}_{sa} = \frac{1}{|\A|} \sum_{a \in \A} \twhat{\bm{P}}_{sa}\tr\bm{r}_{sa} = \frac{1}{|\A|} \sum_{a \in \A} \twhat{\bm{P}}_{sa}\tr\bm{V}_{sa} \\
    \twhat{u}^{\pi}_{s} & = \frac{1}{|\A|}\sum_{a \in \A} \min_{\bm{P} \in \cP} \bm{P}_{sa}\tr\bm{r}_{sa} = \frac{1}{|\A|}\sum_{a \in \A} \min_{\bm{P} \in \cP} \bm{P}_{sa}\tr\bm{V}_{sa}
\end{align*}
Therefore,  $\twhat{\bm{P}} \in \cap_{(s,a) \in \cS \times \A} \arg \min_{\bm{P} \in \cP} \langle \twhat{\bm{P}}_{sa},\bm{V}_{sa}\rangle.$

{\bf Part 3:} \eqref{eq:adv ssp rsas' - sarec} $\Rightarrow$ \eqref{eq:adv dpp rsas' - sarec}. The same arguments as for \eqref{eq:adv ssp rsas' - srec} $\Rightarrow$ \eqref{eq:adv dpp rsas' - srec} shows that we always have 
$\bm{\mu}\tr\twhat{\bm{u}}^{\pi} \leq \min_{\bm{P} \in \cP} \bm{\mu}\tr\bm{v}^{\pi,\bm{P}}$
and that for some $\twhat{\bm{P}} \in \cP$ we have $\twhat{\bm{u}}^{\pi} = \bm{v}^{\pi,\twhat{\bm{P}}}$
from which we conclude that \eqref{eq:adv dpp rsas' - srec} holds.
\end{proof}

\begin{proof}{Proof of Theorem \ref{th:reformulation adv ssp rsas' - sarec}.}
    Similarly as for Theorem \ref{th:reformulation adv ssp rsas' - srec}, the difficult part of the proof is proving that 
$\eqref{eq:adv ssp rsas' - sarec} \Rightarrow \cP = \times_{(s,a) \in \cS \times \A} \cP_{sa}.$
    To do so, we follow the same lines as for the proof of Theorem \ref{th:reformulation adv ssp rsas' - srec}. We start by showing that $\times_{(s,a) \in \cS \times \A} \ext(\cP_{sa}) \subset \cP$. The exact same reasoning as in the proof of Theorem \ref{th:reformulation adv ssp rsas' - srec} shows that for any $\twhat{\bm{P}} \in \times_{(s,a) \in \cS \times \A} \ext(\cP_{sa})$, we can find a sequence of points $\left(\twhat{\bm{P}}^{n}\right)_{n \in \N}$ such that $\twhat{\bm{P}}^{n} \in \cP$ for all $n \in \N$ and such that $\lim_{n \rightarrow + \infty} \twhat{\bm{P}}^{n} = \twhat{\bm{P}}$, from which we conclude that $\twhat{\bm{P}} \in \cP$ since $\cP$ is compact. Concluding that $\times_{(s,a) \in \cS \times \A} \cP_{sa} \subset \cP$ follows by taking the convex hull on both sides of the inclusion $\times_{(s,a) \in \cS \times \A} \ext(\cP_{sa}) \subset \cP$
    and therefore we can conclude that $\cP = \times_{(s,a) \in \cS \times \A} \cP_{sa}.$
    \hfill \Halmos
\end{proof}
\section{Proofs for comparing value functions}\label{app:proof value functions}
In this appendix, we provide proof for the comparisons of the value functions $\bm{v}^{\pi,\bm{P}},\bm{u}^{\pi}$ and $\twhat{\bm{u}}^{\pi}$. The proofs of the results in this section follow from classical arguments from the RMDP literature and we only include them for completeness. The gist of the proofs in this section relies on: 
\begin{enumerate}
    \item All the operators in this paper being {\em order-preserving}, with an operator $F:\R^{S} \rightarrow \R^{S}$ being order-preserving if $\bm{v} \leq \bm{w} \Rightarrow F(\bm{v}) \leq F(\bm{w})$ for any $\bm{v},\bm{w} \in \R^{S}$.
    \item All the operators in this paper being {\em contractions}, with an operator $F:\R^{S} \rightarrow \R^{S}$ being a contraction (with factor $\gamma <1$) if $\| F(\bm{v}) - F(\bm{w})\|_{\infty} \leq \gamma \|\bm{v} - \bm{w}\|_{\infty}$ for any $\bm{v},\bm{w} \in \R^{S}$. A contraction $F$ admits a unique fixed point, and which is the limit of $\left(F^{k}(\bm{v}_{0})\right)_{k \in \N}$ for any $\bm{v}_{0} \in \R^{S}$.
\end{enumerate}
\begin{lemma}\label{lem:ordering u's} 
Let $(\pi,\bm{P}) \in \PiS \times \cP$. Then $\twhat{\bm{u}}^{\pi} \leq \bm{u}^{\pi} \leq \bm{v}^{\pi,\bm{P}}$.
\end{lemma}
\begin{proof}{Proof of Lemma \ref{lem:ordering u's}.}
    Let $s \in \cS$. {We have $\bm{v}^{\pi,\bm{P}} \geq T^{\pi}\left(\bm{v}^{\pi,\bm{P}}\right)$. Indeed,}
    {
    \begin{align*}
       v^{\pi,\bm{P}}_{s} = \sum_{a \in \A} \sum_{a \in \cA} \pi_{sa}\bm{P}_{sa}\tr\left(\bm{r}_{sa}+\gamma \bm{v}^{\pi,\bm{P}}\right) \geq \min_{\tilde{\bm{P}} \in \cP} \sum_{a \in \A} \sum_{a \in \cA} \pi_{sa}\tilde{\bm{P}}_{sa}\tr\left(\bm{r}_{sa}+\gamma \bm{v}^{\pi,\bm{P}}\right) =  T^{\pi}(\bm{v}^{\pi,\bm{P}})_{s} 
    \end{align*}
    }
    {where the first equality is from $\bm{v}^{\pi,\bm{P}}$ being the fixed point of the operator $T^{\pi,\bm{P}}$, the inequality is from taking the minimum over $\cP$, and the second equality is from the definition of the operator $T^{\pi}$.}
    Therefore, $T^{\pi}(\bm{v}^{\pi,\bm{P}}) \leq \bm{v}^{\pi,\bm{P}}$. Since $T^{\pi}$ is order-preserving, we can apply iteratively this operator on both sides of this last inequality to obtain that $\bm{u}^{\pi} \leq \bm{v}^{\pi,\bm{P}}$.

    To show $\twhat{\bm{u}}^{\pi} \leq \bm{u}^{\pi}$ we simply notice that $\twhat{T}^{\pi}(\bm{u}^{\pi}) \leq T^{\pi}(\bm{u}^{\pi}) = \bm{u}^{\pi}$ and we use the same argument as for proving the previous inequality.
    \hfill \Halmos
\end{proof}
{
\section{Complexity of fixed point iteration for our setting}\label{app:fixed point iteration}
In this paper, several problems are reduced to computing the fixed point of a contraction:
\begin{itemize}
    \item When the set $\cP$ is weakly s-DP-compatible, one can solve the policy evaluation problem by computing $\bm{u}^{\pi}$, the fixed point of the contracting operator $T^{\pi}$.
    \item When the set $\cP$ is weakly s-DP-compatible and the uncertainty set is convex, one can compute an optimal stationary policy by computing $\bm{u}\opt$, the fixed point of the contracting operator $T$.
    \item When the set $\cP$ is weakly sa-DP-compatible, one can solve the policy evaluation problem by computing $\twhat{\bm{u}}^{\pi}$, the fixed point of the contracting operator $\twhat{T}^{\pi}$.
\end{itemize}
Note that in all these cases, we have shown that one can solve the policy evaluation problem and compute optimal policies by solving for the same problems over the s-rectangular (or sa-rectangular) extension of $\cP$. Solving for s-rectangular and sa-rectangular MDPs is a very well-studied problems; for the sake of completeness, we provide here some background on fixed point iterations algorithms. We refer to the monograph~\cite{berinde2007iterative} for an extensive review of algorithms for computing the fixed point of contracting operators. We describe below the complexity of the value iteration algorithm for the setting studied in this paper.
{
\paragraph{The value iteration algorithm.} Consider the general problem of approximating the fixed point $\bm{v}_{F} \in \R^{n}$ of an operator $F:\R^{n} \rightarrow \R^{n}$ for some $n \in \N$, that is a contraction with factor $\gamma \in (0,1)$ (we can assume $\gamma>0$ here otherwise computing the fixed point is trivial). The complexity of computing the fixed point of $F$ is controled by (a) the complexity of evaluating $F(\bm{v})$ at a given point $\bm{v} \in \R^{n}$, and (b) how close is the discount factor $\gamma \in [0,1)$ to $1$. Indeed, consider the following 
    {\em value iteration} algorithm:}
    \begin{equation}\label{eq:value iteration}
        \bm{v}^{0} \in \R^{n},\bm{v}^{t+1} = F(\bm{v}^{t}), t \in \N.
    \end{equation}
    {
    From Theorem 6.3.1 in~\cite{puterman2014markov}, we know that the condition $\| \bm{v}^{t+1} - \bm{v}^{t}\|_{\infty} \leq \epsilon (1-\gamma)\gamma^{-1}$ is sufficient to ensure that $\| \bm{v}^{t+1} - \bm{v}_{F}\|_{\infty} \leq \epsilon$. Since $F$ is a contraction, we know that $\| \bm{v}^{t+1} - \bm{v}^{t}\|_{\infty} \leq \epsilon (1-\gamma)\gamma^{-1}$ will happen after at most $O\left((1-\gamma)^{-1}\log\left(\epsilon^{-1}(1-\gamma)^{-1}\right)\right)$ iterations, each iteration $t \in \N$ requiring to evaluate $F$ at a given point $F(\bm{v}^{t})$. If we write $\comp(F)$ the number of arithmetic operations for evaluating the function $F$ at a given point, we note that the overall number of arithmetic operations to compute an $\epsilon$-approximation of the fixed point of $F$ is given by $O\left(\comp(F)(1-\gamma)^{-1}\log\left(\epsilon^{-1}(1-\gamma)^{-1}\right)\right)$.}
}
    {
    \paragraph{The case of Bellman operators.} In the particular case of interest in this paper, the complexity of evaluating $T^{\pi}$ directly relates to optimizing a linear function over $\cP$, since this is the type of optimization problems that appear naturally in the component-wise definitions of $T^{\pi}$ as in~\ref{eq:operator T pi}). Optimizing a linear objective function can be done efficiently (in polynomial time) for a wide class of sets, including polytopes, ellipsoids, sets defined over the non-negative orthants and even sets involving semi-definite positive inequalities, see Section 4.6 in~\cite{ben2021lectures}. When the set $\cP$ is weakly sa-DP-compatible, we know that $\twhat{\bm{u}}^{\pi} = \bm{u}^{\pi}$ (see Section~\ref{sec:weak ssp - sarec}), so the complexity of computing $\twhat{\bm{u}}^{\pi}$ boils down to the complexity of computing $\bm{u}^{\pi}$.  The complexity of evaluating the operator $T$ boils down to solving a saddle-point problem with a bilinear objective and convex compact decision sets $\PiS$ and $\cP$, a well-studied optimization problems for which several efficient algorithms, e.g. section 5.5 and section 5.6 in~\cite{ben2021lectures} and~\cite{chambolle2016ergodic}.
    }
\section{Proof of Lemma \ref{lem:strong duality implies stationary opt policies}}\label{app:proof duality implies opt stationary policies}
\begin{proof}{Proof of Lemma \ref{lem:strong duality implies stationary opt policies}.}
We have
\begin{align*}
 \sup_{\pi \in \PiH} \inf_{\bm{P} \in \cP} 
 \bm{\mu}\tr\bm{v}^{\pi,\bm{P}} \leq  \inf_{\bm{P} \in \cP} \sup_{\pi \in \PiH} 
 \bm{\mu}\tr\bm{v}^{\pi,\bm{P}}  \leq \inf_{\bm{P} \in \cP} \max_{\pi \in \PiS} 
 \bm{\mu}\tr\bm{v}^{\pi,\bm{P}}  & = \max_{\pi \in \PiS} \min_{\bm{P} \in \cP}  
 \bm{\mu}\tr\bm{v}^{\pi,\bm{P}} \\
 &\leq \sup_{\pi \in \PiH} \inf_{\bm{P} \in \cP} 
 \bm{\mu}\tr\bm{v}^{\pi,\bm{P}}
\end{align*}
where the first inequality comes from weak duality, the second inequality comes from the inner problem being an MDP and therefore $\sup_{\pi \in \PiH} 
 \bm{\mu}\tr\bm{v}^{\pi,\bm{P}} = \max_{\pi \in \PiS} 
 \bm{\mu}\tr\bm{v}^{\pi,\bm{P}}$ for any fixed $\bm{P} \in \cP$, the equality comes from assuming strong duality, and the last inequality comes from $\PiS \subset \PiH$. We conclude that all terms in the equations above are equal.     \hfill \Halmos
\end{proof}

{
\section{A counterexample for strong duality}\label{app:counterexample strong duality}
\begin{figure}[htb]
\begin{subfigure}{0.45\linewidth}
    \centering
        \includegraphics[width=1\linewidth]{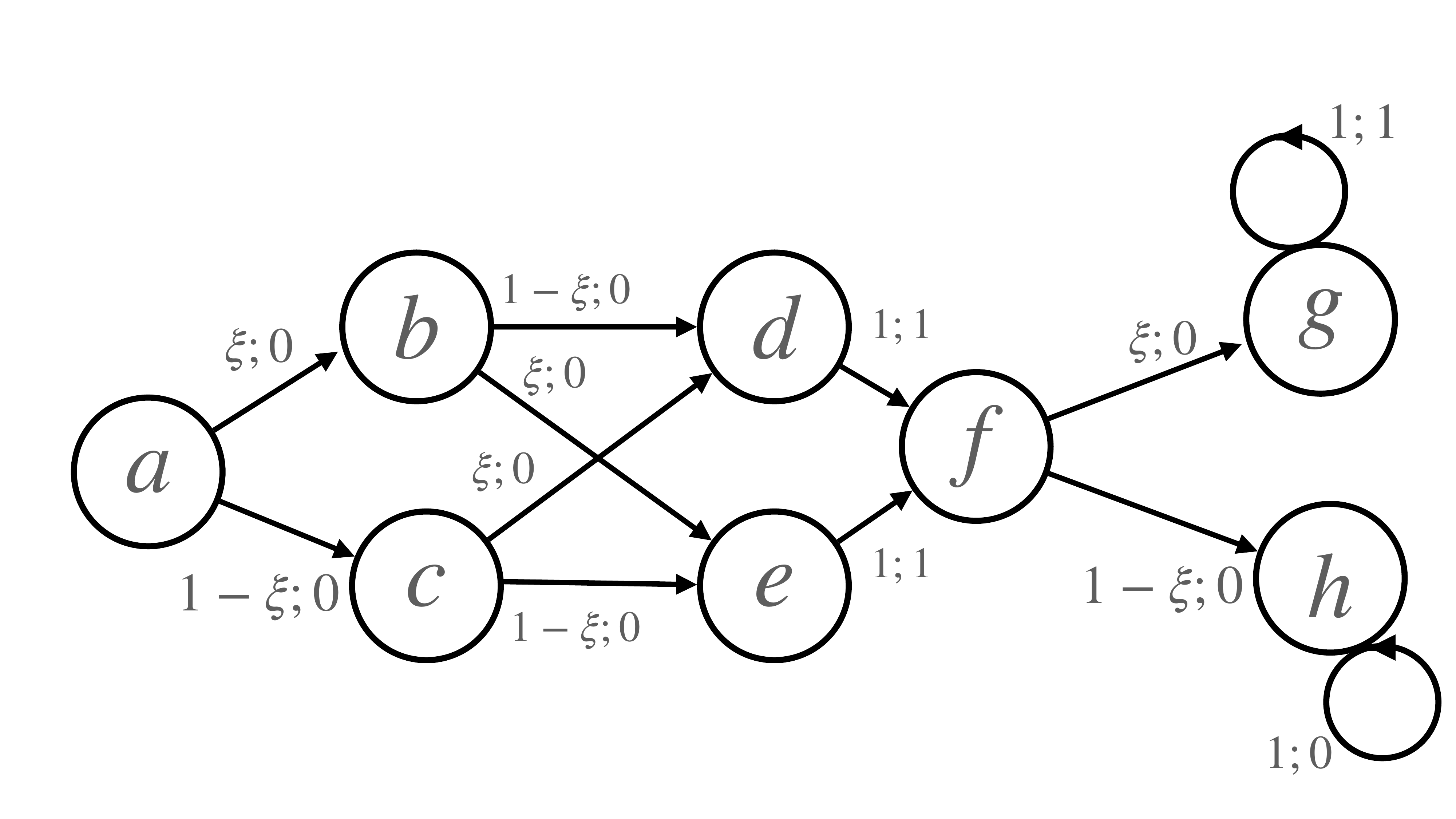}
    \caption{Transitions for action $a_1$.}
    \label{subfig:counterexample strong duality - action a_1}
\end{subfigure}
\begin{subfigure}{0.45\linewidth}
    \centering
        \includegraphics[width=1\linewidth]{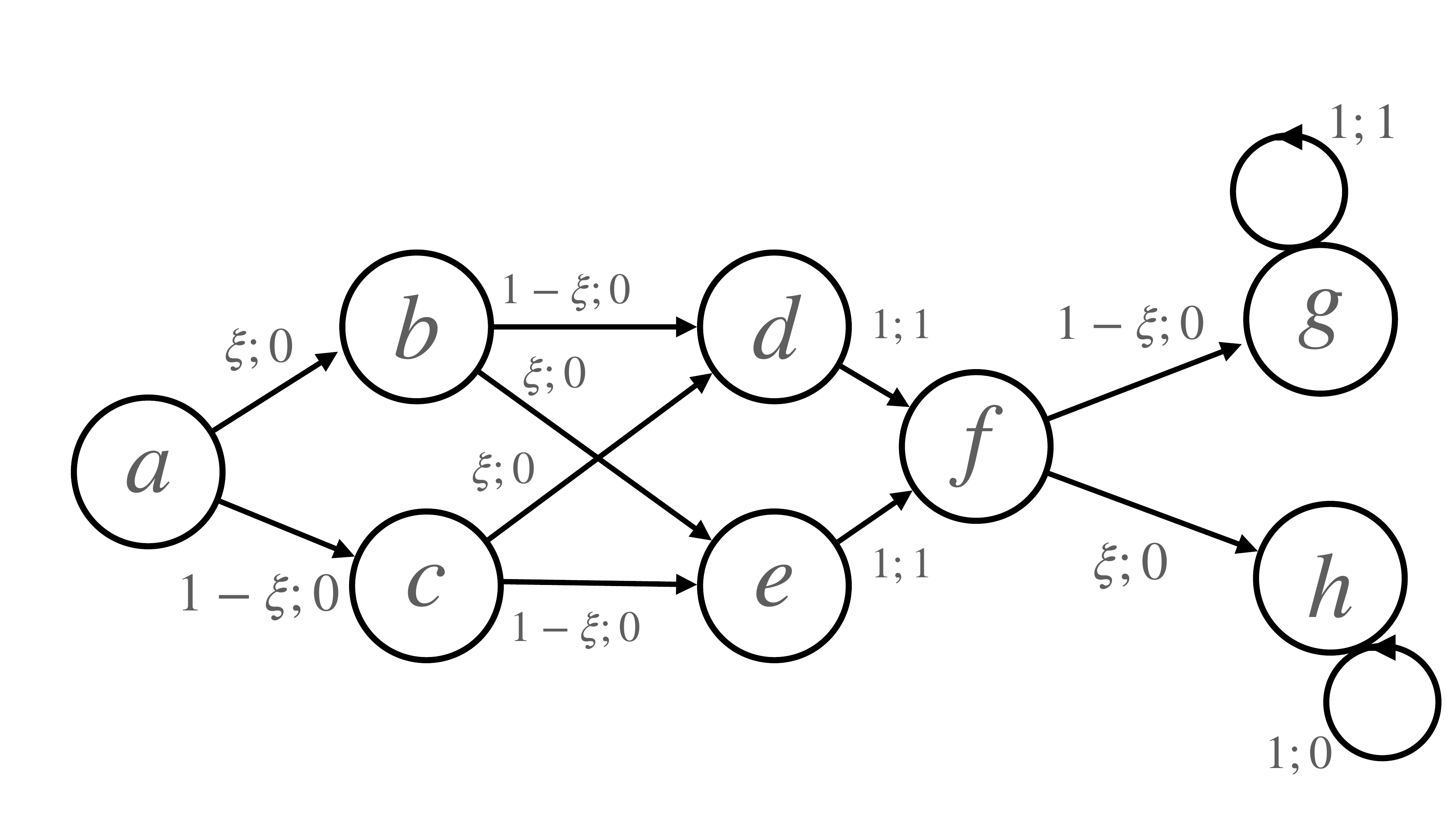}
    \caption{Transitions for action $a_2$.}
    \label{subfig:counterexample strong duality - action a_2}
\end{subfigure}
    \caption{{A non-rectangular robust MDP where strong duality does not hold. This RMDP instance is inspired from the RMDP instance from Figure 4 in~\cite{wiesemann2013robust}}.}
    \label{fig:counterexample strong duality}
\end{figure}
    In this section we analyze an RMDP instance where strong duality does not hold. Our RMDP instance, represented in Figure \ref{fig:counterexample strong duality}, is inspired from Figure 4 in the seminal paper on s-rectangularity~\cite{wiesemann2013robust}). There are eight states $\{a,b,c,d,e,f,g,h\}$ and two actions $\{a_1,a_2\}$ to be chosen in State $f$. The pairs (transition probabilities; instantaneous rewards) are represented on the top of each arc. The transitions for action $a_1$ are represented in Figure \ref{subfig:counterexample strong duality - action a_1} and the transitions for action $a_2$ are represented in Figure \ref{subfig:counterexample strong duality - action a_2}. The transition probabilities are parametrized by a single parameter $\xi \in [0,1]$, hence the set $\cP$ is not s-rectangular. Since the transition probabilities only differ out of state $f$, a stationary policy is parametrized by $\beta \in [0,1]$, the probability to pick action $a_1$ in state $f$. We choose $\gamma = 1/2$. 
Writing $R(\beta,\xi)$ the discounted return starting from state $a$ when the decision-maker chooses $\beta \in [0,1]$ and the transition probabilities are parametrized by $\xi \in [0,1]$, we note that we have
\[R(\beta,\xi) = \frac{1}{4} \left( \xi (1-\xi) + \frac{1}{4}(\beta \xi + (1-\beta)(1-\xi))\right).\]
We use this example to show that 
\[\max_{\beta \in [0,1]} \min_{\xi \in [0,1]} R(\beta,\xi) < \min_{\xi \in [0,1]} \max_{\beta \in [0,1]}  R(\beta,\xi)\]
i.e., we will use this example to show that there may be a duality gap when $\cP$ is not s-rectangular (and, a fortiori, when the weak simultaneous solvability property does not hold). One may expect that strong duality does not hold, since $\xi \mapsto R(\beta,\xi)$ is strictly concave, for any choice of $\beta \in [0,1]$. In particular, we have the following proposition.
    }
{\begin{proposition}\label{prop:duality gap}
    In the RMDP instance from Figure~\ref{fig:counterexample strong duality} $\gamma = 1/2$, we have
    \begin{align*}
        \max_{\beta \in [0,1]} \min_{\xi \in [0,1]} R(\beta,\xi) & = \frac{1}{32} \\
        \min_{\xi \in [0,1]} \max_{\beta \in [0,1]}  R(\beta,\xi) & = \frac{1}{16}.
    \end{align*}
\end{proposition}
\begin{proof}{Proof of Proposition \ref{prop:duality gap}.} 
 Recall that $\xi \mapsto R(\beta,\xi)$ is always a strictly concave function for any $\beta \in [0,1]$.
In particular, if we fix $\beta \in [0,1]$ we obtain that $\min_{\xi \in [0,1]} R(\beta,\xi) = \min \{R(\beta,0),R(\beta,1)\} = \frac{1}{16}\min\{\beta,1-\beta\}$ from which we get that \[\max_{\beta \in [0,1]} \min_{\xi \in [0,1]} R(\beta,\xi) = \max_{\beta \in [0,1]} \frac{1}{16}\min\{\beta,1-\beta\} = \frac{1}{32}\]
attained for $\beta = 1/2$ and $\xi \in \{0,1\}$.
Let us now fix $\xi \in [0,1]$. Because of the expression of $R(\beta,\xi)$, we see that $\max_{\beta \in [0,1]} R(\beta,\xi)$ is attained at $\beta = 1$ if $\xi \geq 1/2$ (in which case $\max_{\beta \in [0,1]} R(\beta,\xi) = \frac{1}{4} \cdot \xi \cdot \left(\frac{5}{4}-\xi\right)$), and attained at $\beta=0$ if $\xi \leq 1/2$, in which $\max_{\beta \in [0,1]} R(\beta,\xi) = \frac{1}{4} \cdot (1-\xi) \cdot \left(\frac{1}{4}+\xi\right).$ In both cases ($\xi \geq 1/2$ and $\xi \leq 1/2$), the corresponding function $\xi \mapsto \max_{\beta \in [0,1]} R(\beta,\xi)$ is concave, so we can find its minimum at one of the endpoint. In particular, we get
\begin{align*}
   \min_{\xi \in [1/2,1]} \max_{\beta \in [0,1]} R(\beta,\xi) & =\frac{1}{4} \cdot \min_{\xi \in [1/2,1]} \xi \cdot \left(\frac{5}{4}-\xi\right) = \frac{1}{4} \cdot\min\{\frac{1}{2} \cdot \frac{3}{4},\frac{1}{4}\} = \frac{1}{4} \cdot \min\{\frac{3}{8},\frac{1}{4}\} = \frac{1}{16} \\
    \min_{\xi \in [0,1/2]} \max_{\beta \in [0,1]} R(\beta,\xi) & =  \min_{\xi \in [0,1/2]} \frac{1}{4} \cdot (1-\xi) \cdot \left(\frac{1}{4}+\xi\right) = \min \frac{1}{4} \cdot \{ \frac{1}{4},\frac{1}{2} \cdot \frac{3}{4}\} = \frac{1}{16}.
\end{align*}
Overall, we get that $\min_{\xi \in [0,1]} \max_{\beta \in [0,1]} R(\beta,\xi) = \frac{1}{16}$, which concludes the proof.  \hfill \Halmos
\end{proof} 
}
{
\section{Proof of Theorem \ref{th:adv DPP, feasible upi, adv SSP equivalent - rsa - sarec}}\label{app:proof th:adv DPP, feasible upi, adv SSP equivalent - rsa - sarec}
The proof of Theorem \ref{th:adv DPP, feasible upi, adv SSP equivalent - rsa - sarec} follows the lines of the proof of Theorem \ref{th:adv DPP, feasible upi, adv SSP equivalent - rsa - srec}.}
\begin{proof}{Proof of Theorem \ref{th:adv DPP, feasible upi, adv SSP equivalent - rsa - sarec}.}
{
{\bf Part 1:} \eqref{eq:adv dpp rsa - sarec} $\Rightarrow$ \eqref{eq:feas upi rsa - sarec}.
Assume that $\cP$ is weakly sa-DP-compatible. From Proposition~\ref{prop:sa tractable are s tractable}, $\cP$ is weakly s-DP-compatible and $\bm{u}^{\pi} = \twhat{\bm{u}}^{\pi}$. From Theorem \ref{th:adv DPP, feasible upi, adv SSP equivalent - rsa - srec}, we know that $\bm{u}^{\pi}$ is feasible, i.e. there exists $\hat{\bm{P}} \in \cP$ such that $\bm{u}^{\pi} = \bm{v}^{\pi,\hat{\bm{P}}}$, from which we conclude that $\twhat{\bm{u}}^{\pi} = \bm{v}^{\pi,\hat{\bm{P}}}$.
}

{
{\bf Part 2:} \eqref{eq:feas upi rsa - sarec}  $\Rightarrow$ \eqref{eq:adv ssp - rsa - sarec}. Let $\pi \in \PiS$.
Take $\bm{V}\in \R^{S}$ and define $\bm{r} \in \R^{S \times A}$ as $r_{sa} = V_{s} - \sum_{a \in \cA} \gamma \pi_{sa} \min_{\bm{P} \in \cP} \bm{P}_{sa}\tr\bm{V}$ for all $(s,a) \in \cS \times \cA$. Then $\twhat{T}^{\pi}(\bm{V}) = \bm{V}$, so that $\bm{V} = \twhat{\bm{u}}^{\pi}$. From~\eqref{eq:feas upi rsa - sarec}, we can $\hat{\bm{P}} \in \cP$ such that $\twhat{\bm{u}}^{\pi} = \bm{v}^{\pi,\hat{\bm{P}}}$, which in turn implies that $\cap_{(s,a) \in \cS \times \cA} \langle \bm{P}_{sa},\bm{V}\rangle \neq \emptyset$.
}

{
{\bf Part 3:} \eqref{eq:adv ssp - rsa - sarec} $\Rightarrow$ \eqref{eq:adv dpp rsa - sarec}. We always have $\min_{\bm{P} \in \cP} \bm{\mu}\tr\bm{v}^{\pi,\bm{P}} \geq \bm{\mu}\tr\hat{\bm{u}^{\pi}}$. From \eqref{eq:adv ssp - rsa - sarec} we obtain that $\twhat{\bm{u}}^{\pi} = \bm{v}^{\pi,\hat{\bm{P}}}$ for some $\hat{\bm{P}} \in \cP$, so that $\min_{\bm{P} \in \cP} \bm{\mu}\tr\bm{v}^{\pi,\bm{P}} \geq \bm{\mu}\tr\hat{\bm{u}^{\pi}} = \bm{\mu}\tr\bm{v}^{\pi,\hat{\bm{P}}}$, which concludes the proof of Part 3.
}
\hfill \Halmos
\end{proof}

\section{Examples for Section \ref{sec:discussion}}\label{app:example discussion}
First, we note that the existence (or absence) of stationary optimal policies may depend on the initial distribution.
\begin{proposition}[Existence of stationary optimal policy may depend on initial distribution]\label{prop:existence of stat opt pol depend on mu}
    There exists a non-rectangular robust MDP instance with six states such that:
    \begin{itemize}
        \item If the RMDP starts in State $1$, then there is a unique optimal policy, and it is history-dependent.
        \item if the RMDP starts in State $4$, then there is a unique optimal policy and it is stationary.
    \end{itemize}
\end{proposition}
\begin{proof}{Proof of Proposition \ref{prop:existence of stat opt pol depend on mu}}
    Consider the example from Figure 4 in \cite{wiesemann2013robust}. If the initial state is state 1, then Proposition 2 in \cite{wiesemann2013robust} shows that there is a unique optimal policy, and it is history-dependent. But if the initial state is state 4, we are actually in the same situation as the RMDP instance from Figure 3 in \cite{wiesemann2013robust}, and Proposition 1 in \cite{wiesemann2013robust} shows that there is a unique optimal policy and it can be chosen stationary (randomized). 
    \hfill \Halmos
\end{proof}

In fact we have the next proposition, where we show that even in robust MDPs where an optimal policy may be chosen stationary, this stationary policy may depend on the initial distribution.
\begin{proposition}[Stationary optimal policies may depend on initial distribution]\label{prop:piopt depends on mu}
   Consider the non-rectangular robust MDP instance presented in Figure \ref{fig:piopt depends on mu}.
    \begin{enumerate}
        \item If the RMDP starts in State $b$, then we can choose an optimal policy to be stationary.
        \item If the RMDP starts in State $a$, then we can choose an optimal policy to be stationary.
        \item There is no stationary policy that is optimal when starting in $a$ {\em and} when starting in $b$.
        \item The weak SSP~\eqref{eq:adv ssp - rsa - srec} does not hold.
    \end{enumerate}
\end{proposition}
\begin{proof}{Proof of Proposition \ref{prop:piopt depends on mu}.} We consider an MDP instance with $5$ states in $\cS = \{a,b,c,d,e\}$ and two actions $a_{1}$ and $a_{2}$ to be chosen in State $a$ and State $b$. The transition probabilities are parametrized with a single parameter $p \in [0,1]$ and are represented in Figure \ref{fig:piopt depends on mu}. We parameterize policies by $(\alpha,\beta_1,\beta_2,\beta_3) \in [0,1]^{4}$, with 
\begin{itemize}
    \item $\alpha$ the probability to choose $a_{1}$ in State $a$.
    \item $\beta_1$ the probability to choose $a_{1}$ in State $b$, if the history is $(a,a_{1},b)$.
    \item $\beta_2$ the probability to choose $a_{1}$ in State $b$, if the history is $(a,a_{2},b)$.
    \item $\beta_3$ the probability to choose $a_{1}$ in State $b$, if the history is $(b)$.
\end{itemize}
We identify transition probabilities $\bm{P} \in \cP$ with scalar $p \in [0,1]$. Given an initial distribution $\bm{\mu} \in \Delta(\cS)$ with $\cS = \{a,b,c,d,e\}$, we can write the return of a policy $\pi$ as
\begin{align*}
    (1-\gamma)R(\pi,\bm{P}) = & \mu_{a}\left(\frac{\gamma}{2}(1-p) + \gamma^2 p (1-2p)(2(\alpha \beta_1 + 
     (1-\alpha)\beta_2)-1)\right)  \\ 
     & +  (\mu_{e}-\mu_{d}+\frac{1}{2}\mu_{c}) + \mu_{b}\gamma (2\beta_3-1)(1-2p).
\end{align*}
\begin{enumerate}
    \item {\bf The case $\mu_{b}=1$.} This is the case considered in Proposition 1/Figure 3 in \cite{wiesemann2013robust}. We include the proof of our conclusion here for completeness. With $\mu_{b}=1$, we have 
    \[ (1-\gamma)R(\pi,\bm{P}) = \gamma (2\beta_3-1)(1-2p).\]
    Since the return only depends on $\beta_3$, we can always choose a stationary optimal policy, and we find that $\max_{\beta_3 \in [0,1]} \min_{p \in [0,1]} \gamma (2\beta_3-1)(1-2p)$ is attained for $\beta_3 = 1/2$, i.e., the set of optimal policies is $\{ (\alpha,\beta_1,\beta_2,1/2) \; | \; \alpha,\beta_1,\beta_2 \in [0,1]\}$ and the set of stationary optimal policies is $\{ (\alpha,1/2,1/2,1/2) \; | \; \alpha \in [0,1]\}$, with an optimal objective of $\max_{\pi \in \PiS} \min_{\bm{P} \in \cP} R(\pi,\bm{P}) = 0.$
    \item {\bf The case $\mu_{a}=1$.}
    In this case 
    \[ (1-\gamma)R(\pi,\bm{P}) = \frac{\gamma}{2}(1-p) + \gamma^2 p (1-2p)(2(\alpha \beta_1 + (1-\alpha)\beta_2)-1).\]
    We note that for any policy $(\alpha,\beta_1,\beta_2,\beta_3)$, possibly history-dependent, there exists a stationary policy $\pi'=(\alpha,\beta(\alpha),\beta(\alpha),\beta(\alpha))$ with $\beta(\alpha) = \alpha \beta_1 + (1-\alpha)\beta_2$ such that
    \[ (1-\gamma)R(\pi,\bm{P}) = (1-\gamma)R(\pi',\bm{P}).\]
    Additionally, given a stationary policy $(\alpha,\beta,\beta,\beta)$, the return is
    \[ (1-\gamma)R(\pi,\bm{P}) = \frac{\gamma}{2}(1-p) + \gamma^2 p (1-2p)(2\beta-1).\]
For the rest of this example, we choose $\gamma = 1/4$, for which we obtain
\begin{align*}
R(\pi,\bm{P}) & = \frac{1}{6} \left((1-p) + \frac{1}{2} p (1-2p) (2\beta-1)\right)= \frac{1}{6} \left( 1 + p \left(\beta - \frac{3}{2}\right) + p^2 \left(1 - 2\beta)\right)\right).
\end{align*}
We now inspect different ranges for the values of $\beta$.
\begin{itemize}
\item When $\beta \geq 1/2$, the function $p \mapsto 1 + p \left(\beta - \frac{3}{2}\right) + p^2 \left(1 - 2\beta)\right)$ is concave since $1-2\beta \leq 0$, so that its minimum over $p \in [0,1]$ is attained at $p \in \{0,1\}$, i.e., 
\[ \min_{p \in [0,1]} R(\pi,\bm{P}) = \frac{1}{6} \min \{R(\pi,0),R(\pi,1)\} = \frac{1}{6} \{1, \frac{1}{2}-\beta\} = \frac{1}{6}\left(\frac{1}{2}-\beta\right)\]
i.e., the worst-case is always attained at $p=1$, and
$\max_{\beta \in [1/2,1]} \min_{p \in [0,1]} R(\pi,\bm{P}) = \max_{\beta \in [1/2,1] } \frac{1}{6}\left(\frac{1}{2}-\beta\right) = 0$, attained at $\beta=1/2$.
\item Take $\beta=0$. We obtain
$R(\pi,\bm{P}) = \frac{1}{6} \left(1 - \frac{3}{2}p+p^2\right).$ 
The function $p \mapsto 1 - \frac{3}{2}p+p^2$ attains its minimum at $p=3/4$ for which
\[ \min_{p \in [0,1]} R(\pi,\bm{P})  = R(\pi,3/4) = \frac{1}{6}\left(1 - \frac{3}{2}\frac{3}{4} + \frac{9}{16}\right) = \frac{7}{96}>0.\]
This suffices to show that $\beta=1/2$ is not an optimal solution when $\mu_a = 1$, despite being an optimal policy when $\mu_b = 1$. 
\item We now study the case $\beta \leq 1/2$. The function $p \mapsto 1 + p \left(\beta - \frac{3}{2}\right) + p^2 \left(1 - 2\beta)\right)$ attains its minimum on $\R$ at \[p^{*} = \frac{\frac{3}{2}-\beta}{2(2\beta-1)}.\]
Note that we always have $p^{*} >0$, and $p^* \leq 1 \iff \beta < \frac{1}{6}$. Therefore for $\beta \in [1/6,1/2]$, $p \mapsto 1 + p \left(\beta - \frac{3}{2}\right) + p^2 \left(1 - 2\beta)\right)$ attains its minimum on $[0,1]$ at $p=1$, and
$\min_{p \in [0,1]} R(\pi,\bm{P})  = R(\pi,1) = \frac{1}{6}\left(\frac{1}{2} - \beta\right)$ so that 
\[ \max_{\beta \in [1/6,1/2]}  \min_{p \in [0,1]} R(\pi,\bm{P}) = \max_{\beta \in [1/6,1/2]} \frac{1}{6}\left(\frac{1}{2} - \beta\right) = \frac{1}{6}\left(\frac{1}{2} - \frac{1}{6}\right) = \frac{1}{18}.\]

For $\beta \in [0,1/6]$, we have
\begin{align*}
\min_{p \in [0,1]} R(\pi,\bm{P}) = R(\pi,p^*) & = \frac{1}{6}1 + p^* \left(\beta - \frac{3}{2}\right) + p^{* 2} \left(1 - 2\beta)\right) = \frac{1}{6}\left(1-\frac{1}{4}\frac{\left(\frac{3}{2}-\beta\right)^2}{1-2\beta} \right)
\end{align*}
The function $\beta \mapsto \frac{\left(\frac{3}{2}-\beta\right)^2}{1-2\beta}$ is increasing over $[0,1/6]$, so that the maximum of $\min_{p \in [0,1]} R(\pi,\bm{P})$ is attained at $\beta=0$, for which we already computed that
$\min_{p \in [0,1]} R(\pi,\bm{P}) = \frac{7}{96}.$
\end{itemize}
Overall, we conclude that 
$\max_{\beta \in [0,1]} \min_{p \in [0,1]} R(\pi,\bm{P}) = \frac{7}{96}$, 
attained at $\beta=0$ with a worst-case transition probabilities of $p=\frac{3}{4}$.
Therefore, the set of stationary optimal policies when $\mu_a = 1$ is $\{(\alpha,0,0,0) \; | \; \alpha \in [0,1]\}$.
\item For $\mu_a = 1$, the set of stationary optimal policies  is $\{(\alpha,0,0,0) \; | \; \alpha \in [0,1]\}$. For $\mu_b = 1$ stationary optimal policies are $\{(\alpha,1/2,1/2,1/2) \; | \; \alpha \in [0,1]\}$, and these two sets are disjoint.
    \item Let $\bm{v} = (0,0,1,1,0)$.
\begin{itemize}
\item For State $s=\{a\}$ we have $\bm{P}_{sa_{1}} = \bm{P}_{sa_{2}} =(0,p,1-p,0,0)$ so that
\[ \arg \min_{\bm{P} \in \cP} \sum_{a \in \A} \pi_{sa}\bm{P}_{sa}\tr\bm{v} = \arg \min_{p \in [0,1]} p v_{b} + (1-p) v_{c} = \arg \min_{p \in [0,1]} (1-p) = \{1\}\]
where for the sake of conciseness we identify $p \in [0,1]$ and $\bm{P} \in \cP$.
\item For State $s=\{b\}$ we have $\bm{P}_{sa_{1}} = (0,0,0,p,1-p),\bm{P}_{sa_{2}} = (0,0,0,1-p,p)$, so that 
\[ \arg \min_{\bm{P} \in \cP} \sum_{a \in \A} \pi_{sa}\bm{P}_{sa}\tr\bm{v} = \arg \min_{p \in [0,1]} p v_{d} + (1-p) v_{e} = \arg \min_{p \in [0,1]} p = \{0\}\]
and therefore
$\cap_{s \in \cS} \arg \min_{\bm{P} \in \cP} \sum_{a \in \A} \pi_{sa}\bm{P}_{sa}\tr\bm{v} = \emptyset.$
\end{itemize}
\end{enumerate}
\hfill \Halmos
\end{proof}
    \begin{figure}
\begin{center}
    \begin{subfigure}{0.4\textwidth}
\includegraphics[width=\linewidth]{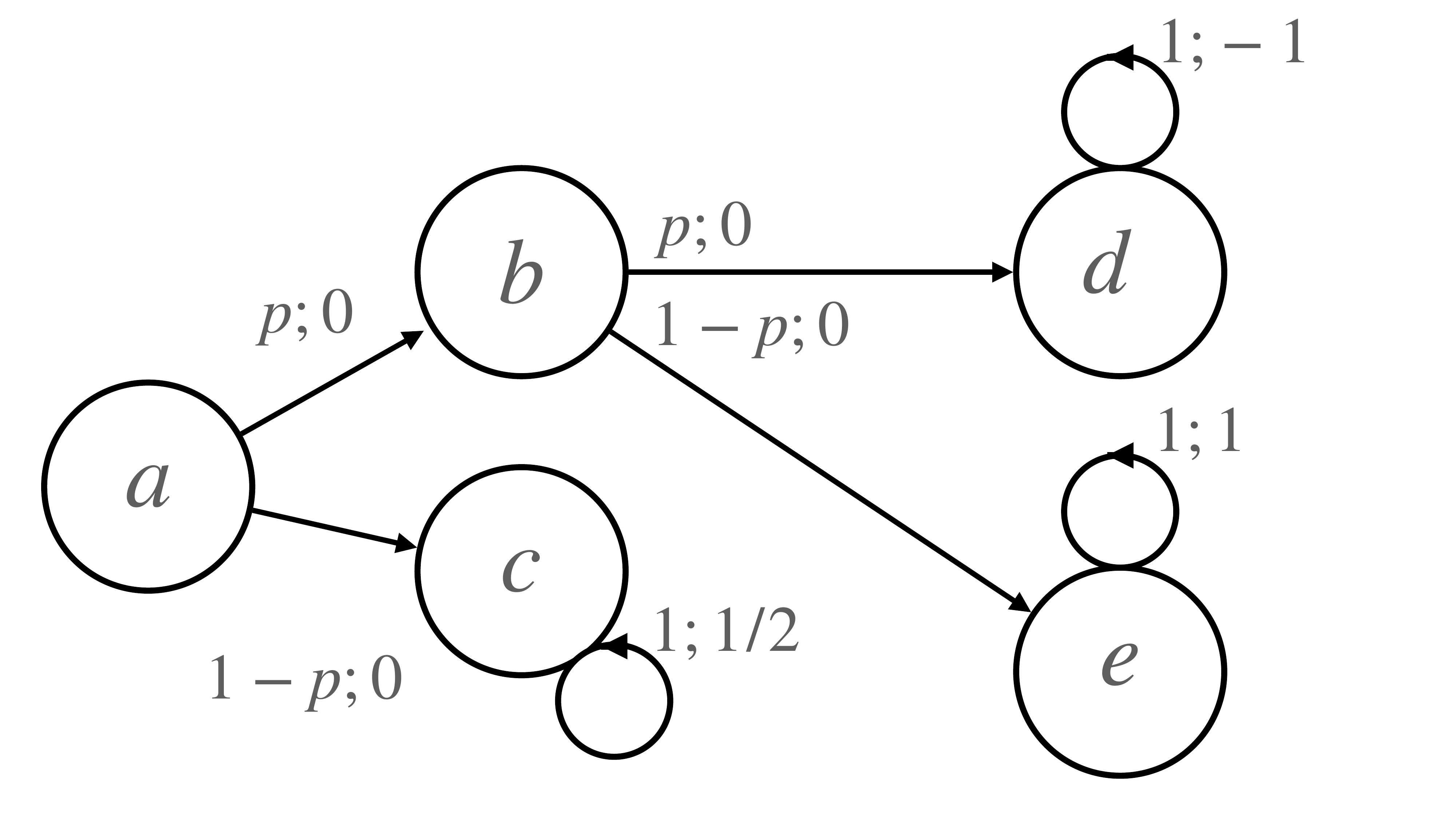}
    \caption{Transition for action $a_1$}
    \label{fig:rmdp_a1}
    \end{subfigure}
    \begin{subfigure}{0.4\textwidth}
\includegraphics[width=\linewidth]{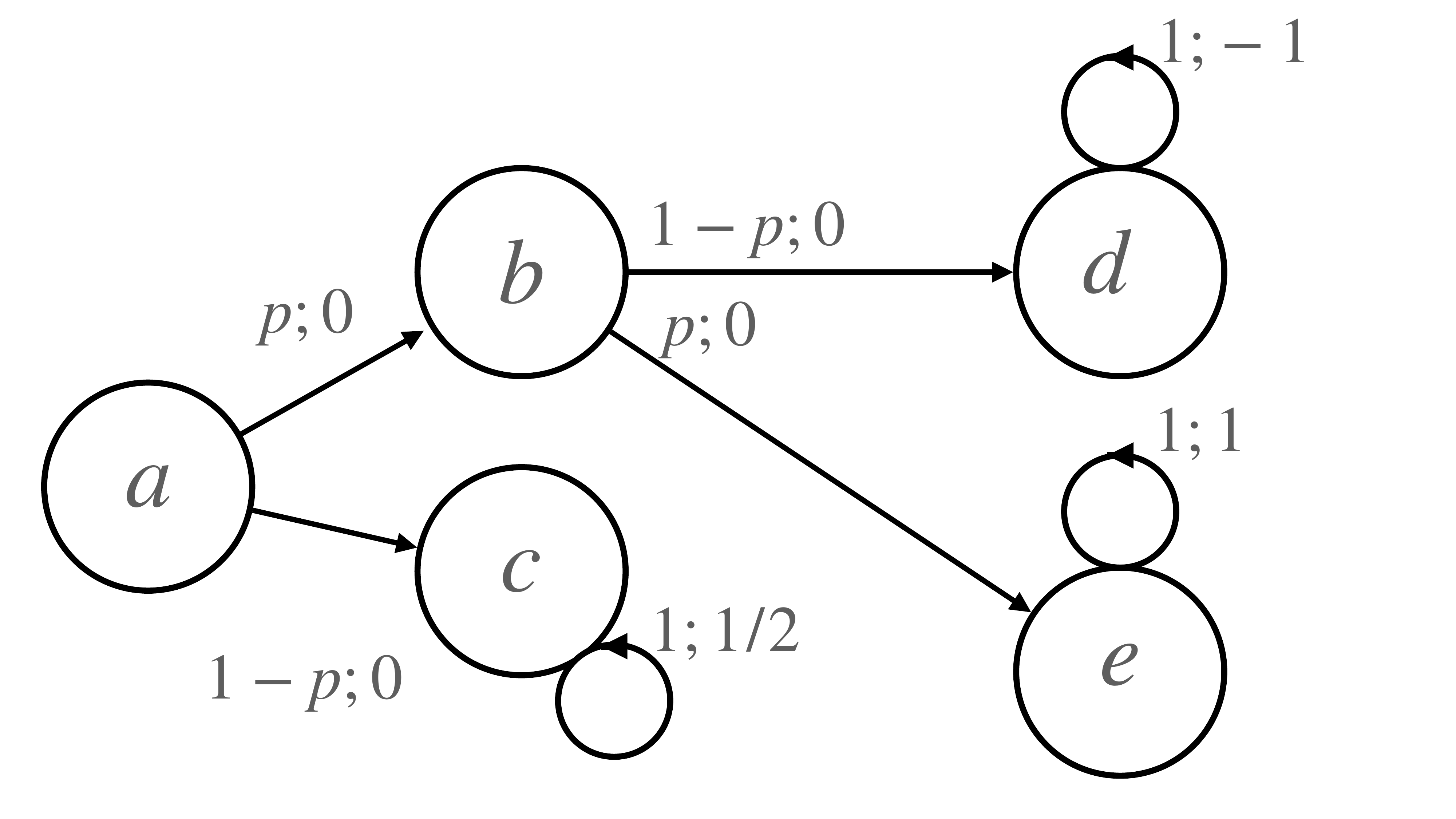}
    \caption{Transition for action $a_2$}
    \label{fig:rmdp_a2}
    \end{subfigure}
    \end{center}
\caption{Transitions and rewards for the RMDP instance for Proposition \ref{prop:piopt depends on mu}.}\label{fig:piopt depends on mu}
\end{figure}

\end{APPENDICES}

\end{document}